\documentclass[11pt,letterpaper]{amsart}

\usepackage{mathtools}
\usepackage{amsthm}
\usepackage{mathrsfs} 
\usepackage{amssymb}
\usepackage[all]{xy}
\usepackage[nobottomtitles*]{titlesec}
\usepackage{titletoc}
\usepackage{bbm}
\usepackage{upgreek}
\usepackage{ marvosym}
\usepackage{multicol}
\usepackage{rotating}
\usepackage{pdflscape}
\usepackage{ifluatex}
\usepackage[pdfencoding=auto, psdextra]{hyperref}
\usepackage{enumerate}
\usepackage{soul}
\usepackage{graphicx}
\usepackage{stmaryrd}

\usepackage{verbatim}
\usepackage{tikz-cd}
\makeatletter
\tikzcdset{
    cong/.style={"\cong" {sloped, description, yshift=0pt,#1}, phantom},
    simeq/.style={"\simeq" {sloped, description, yshift=0pt,#1}, phantom},
    snake/.style={
        out= east, in=west,
        to path={
            \pgfextra{
                \pgfextractx{\pgf@xa}{\pgfpointanchor{\tikztostart}{east}}
                \pgfextractx{\pgf@xb}{\pgfpointanchor{\tikztotarget}{west}}
                \pgfextracty{\pgf@ya}{\pgfpointanchor{\tikztostart}{center}}
                \pgfextracty{\pgf@yb}{\pgfpointanchor{\tikztotarget}{center}}
                \edef\tikzstartx{\the\pgf@xa}
                \edef\tikzendx{\the\pgf@xb}
                \edef\midy{\the\dimexpr0.5\dimexpr\pgf@ya\relax +0.5\dimexpr\pgf@yb\relax}
            }
            to[in=0,out=180,looseness=0.5] (\tikzstartx,\midy)
            -| ([xshift=-2ex]\tikztotarget.west)
            -- (\tikztotarget)}
    }
}
\makeatother
\usepackage{sseq}
\usepackage{tikzcdintertext}
\usepackage{etoolbox}
\usepackage{xpatch}

\newcommand{\cal}{\mathcal}

\usepackage[normalem]{ulem}
\newcommand{\stkout}[1]{\ifmmode\text{\sout{\ensuremath{#1}}}\else\sout{#1}\fi}

\newcommand{\B}{\mr{B}}

\makeatletter

\pretocmd{\maketitle}{%
    \edef\@title{\unexpanded{\protect\Large}\unexpanded\expandafter{\@title}}%
    \edef\authors{\unexpanded{\protect\normalsize}\unexpanded\expandafter{\authors}}%
}{}{\error}

\def\addlabeltolink#1{\addlabeltolink@#1}
\def\addlabeltolink@#1#2#3#4{#1{#2}{#3}{\thecontentslabel. #4}}

\dottedcontents{section}[0em]{}{2.3em}{1pc}
\dottedcontents{section}[3.8em]{}{2.3em}{1pc}

\titleformat{\section}[block]{\centering\bfseries\Large}{\thetitle. }{0pt}{}

\def\@secnumpunct{. }
\xpatchcmd{\proof}{\topsep6\p@\@plus6\p@\relax}{\topsep0pt\relax}{}{\error}
\makeatother
\usepackage[parfill]{parskip}

\allowdisplaybreaks[1]

\newcommand\Einfty{\mathbb{E}_{\infty}}

\newcommand{\Ho}{\mr{Ho}}

\newcommand{\ull}[1]{\underline{#1}}
\newcommand{\oll}[1]{\overline{#1}}

\newcommand{\mr}[1]{\mathrm{#1}}

\newcommand{\abs}[1]{\lvert #1 \rvert}

\newcommand{\td}[1]{\widetilde{#1}}
\newcommand{\wh}[1]{\widehat{#1}}

\renewcommand{\psi}{\uppsi}

\newcommand{\PP}{\mathbb{P}}

\newcommand\Ocal{\mathcal{O}}

\newcommand{\ku}{{\sf ku}}

\newcommand{\R}{\mr{R}}

\newcommand{\EO}{\mr{EO}}

\newcommand{\sfu}{{\sf u}}

\newcommand{\RO}{\mr{RO}}

\newcommand*{\rom}[1]{\expandafter\@slowromancap\romannumeral #1@} \makeatother

\hypersetup{%
  bookmarksnumbered=true,%
  bookmarks=true,%
  colorlinks=true,%
  pdfnewwindow=true,%
  pdfstartview=FitBH%
}

\definecolor{limegreen}{rgb}{0.2, 0.8, 0.2}
\definecolor{darkmagenta}{rgb}{0.55, 0.0, 0.55}
\definecolor{lavenderrose}{rgb}{0.91, 0.33, 0.5}
\definecolor{goldenpoppy}{rgb}{0.99, 0.76, 0.0}
\definecolor{seagreen}{rgb}{0.1, 0.4, 0.1}
\definecolor{maroon}{RGB}{128,0,0}
\definecolor{darkviolet}{RGB}{148,0,211}
\definecolor{darkbrown}{rgb}{.55,.3,.2}
\definecolor{snowblue}{rgb}{0.3, 0.4, 0.8}
\definecolor{DarkBlue}{rgb}{.1, 0.35, 0.6} 
\definecolor{DarkBrown}{rgb}{.5, 0.2, 0.2} 

\hypersetup{
    colorlinks,
    linkcolor={DarkBrown},
    citecolor={DarkBlue},
    urlcolor={blue!80!black}
}

\def\@url#1{{\tt\def~{\lower3.5pt\hbox{\char'176}}\def\_{\char'137}#1}}

\newtheorem{main}{Main Theorem}

\makeatletter
\let\c@lemma\c@theorem
\makeatother

\newtheorem{thm}{Theorem}[section]
\newtheorem{cor}{Corollary}[section]
\newtheorem{lem}{Lemma}[section]
\newtheorem{prop}{Proposition}[section]

\newtheorem{conj}{Conjecture}[section]

\theoremstyle{definition}
\newtheorem{defn}{Definition}[section]
\newtheorem{ex}{Example}[section]

\newtheorem{rmk}{Remark}[section]

\newtheorem{notn}{Notation}[section]

\newtheorem{question}{Question}[section]

\newtheorem{assump}{Assumption}[section]

\newtheorem*{thm*}{Theorem}
\newtheorem*{cor*}{Corollary}
\newtheorem*{lem*}{Lemma}
\newtheorem*{prop*}{Proposition}
\newtheorem*{defn*}{Definition}
\newtheorem*{ex*}{Example}
\newtheorem*{exs*}{Examples}
\newtheorem*{rmk*}{Remark}
\newtheorem*{claim*}{Claim}
\newtheorem*{not*}{Notation}
\newtheorem*{guess*}{Guess}
\newtheorem*{exer*}{Exercise}
\newtheorem*{conv}{Convention}

\numberwithin{equation}{section}
\numberwithin{figure}{section}

\makeatletter
\let\c@thm\c@equation
\let\c@cor\c@equation
\let\c@lem\c@equation
\let\c@prop\c@equation
\let\c@defn\c@equation
\let\c@ex\c@equation
\let\c@exs\c@equation
\let\c@rmk\c@equation
\let\c@prob\c@equation
\let\c@obs\c@equation
\let\c@conj\c@equation
\let\c@claim\c@equation
\let\c@question\c@equation
\let\c@warn\c@equation
\let\c@notn\c@equation
\let\c@construction\c@equation
\let\c@formula\c@equation
\let\c@assump\c@equation
\makeatother

\usepackage[T1]{fontenc}
\makeatletter
\newif\ifkp@upRm
\DeclareSymbolFont{Letters}{OML}{jkp}{m}{n}
\DeclareMathSymbol{\partialup}{\mathord}{Letters}{128}

\DeclareMathOperator{\Map}{Map}

\DeclareMathOperator{\Sp}{\mathcal{S}p}

\DeclareMathOperator{\Top}{\mathcal{T}op}

\DeclareMathOperator{\Ind}{Ind}

\DeclareMathOperator{\Th}{{\sf Th}} 

\DeclareMathOperator{\res}{{\sf res}}
\DeclareMathOperator{\tr}{{\sf tr}}

\newcommand{\All}{\cal{A}\ell \ell}

\renewcommand\H{\mathrm{H}}

\def\makecommands#1#2#3{
    \bgroup
    \def\tempcmdname##1{#1}
    \def\tempcmdbody##1{#2}
    \def\\##1{\expandafter\xdef\csname\tempcmdname{##1}\endcsname{\unexpanded\expandafter{\tempcmdbody{##1}}}}
    #3
    \egroup
}

\def\upperalphabet{\\A\\B\\C\\D\\E\\F\\G\\H\\I\\J\\K\\L\\M\\N\\O\\P\\Q\\R\\S\\T\\U\\V\\W\\X\\Y\\Z}
\def\loweralphabet{\\a\\b\\c\\d\\e\\f\\g\\h\\i\\j\\k\\l\\m\\n\\o\\p\\q\\r\\s\\t\\u\\v\\w\\x\\y\\z}
\def\lowergreekalphabet{\\\alpha\\\beta\\\gamma\\\delta\\\epsilon\\\zeta\\\eta\\\theta\\\kappa\\\lambda\\\mu\\\nu
    \\\xi\\\pi\\\rho\\\sigma\\\tau\\\upsilon\\\psi\\\chi\\\phi\\\omega}

\makecommands{#1mr}{\mathrm{#1}}{\upperalphabet}  
\makecommands{#1cal}{\mathcal{#1}}{\upperalphabet}  
\makecommands{#1#1}{\mathbb{#1}}{\upperalphabet} 
\makecommands{#1bar}{\overline{#1}}{\upperalphabet}
\makecommands{#1scr}{\mathscr{#1}}{\upperalphabet}
\makecommands{#1twee}{\widetilde{#1}}{\upperalphabet\loweralphabet} 
\makecommands{sf#1}{{\sf #1}}{\upperalphabet\loweralphabet}
\makecommands{#1frak}{\mathfrak{#1}}{\upperalphabet}  
\makeatletter
\makecommands{\expandafter\@gobble\string#1bar}{\overline{#1}}{\lowergreekalphabet}
\makecommands{\expandafter\@gobble\string#1twee}{\widetilde{#1}}{\lowergreekalphabet}
\makeatother

\DeclareMathOperator{\Sq}{Sq}

\newcommand{\sma}{\wedge}
\newcommand{\sm}{\wedge}

\makeatletter

\def\tikzcdequalsignoffset{0.1em}

\def\findedgesourcetarget#1#2{
    \let\sourcecoordinate\pgfutil@empty
    \ifx\tikzcd@startanchor\pgfutil@empty 
        \def\tempa{\pgfpointanchor{#1}{center}}
    \else
        \edef\tempa{\noexpand\pgfpointanchor{#1}{\expandafter\@gobble\tikzcd@startanchor}} 
        \let\sourcecoordinate\tempa
    \fi
    \ifx\tikzcd@endanchor\pgfutil@empty 
        \def\tempb{\pgfpointshapeborder{#2}{\tempa}}
    \else
        \edef\tempb{\noexpand\pgfpointanchor{#2}{\expandafter\@gobble\tikzcd@endanchor}}
    \fi
    \let\targetcoordinate\tempb
    \ifx\sourcecoordinate\pgfutil@empty%
        \def\sourcecoordinate{\pgfpointshapeborder{#1}{\tempb}}%
    \fi
}

\tikzset{/tikz/commutative diagrams/equal/.style=equals,
    /tikz/commutative diagrams/equals/.style = {
    -,
    to path={\pgfextra{
        \findedgesourcetarget{\tikzcd@ar@start}{\tikzcd@ar@target} 
        \ifx\tikzcd@startanchor\pgfutil@empty
            \def\tikzcd@startanchor{.center}
        \fi
        \ifx\tikzcd@endanchor\pgfutil@empty
            \def\tikzcd@endanchor{.center}
        \fi
        \pgfmathanglebetweenpoints{\pgfpointanchor{\tikzcd@ar@start}{\expandafter\@gobble\tikzcd@startanchor}}{\pgfpointanchor{\tikzcd@ar@target}{\expandafter\@gobble\tikzcd@endanchor}}
        \pgftransformrotate{\pgfmathresult}
        \pgfpathmoveto{\pgfpointadd{\sourcecoordinate}{\pgfpoint{0}{\tikzcdequalsignoffset}}}
        \pgfpathlineto{\pgfpointadd{\targetcoordinate}{\pgfpoint{0}{\tikzcdequalsignoffset}}}
        \pgfpathmoveto{\pgfpointadd{\sourcecoordinate}{\pgfpoint{0}{-\tikzcdequalsignoffset}}}
        \pgfpathlineto{\pgfpointadd{\targetcoordinate}{\pgfpoint{0}{-\tikzcdequalsignoffset}}}
        \pgfusepath{draw}
}}}}

\makeatother

\usepackage[noabbrev,nameinlink]{cleveref}

\crefname{equation}{}{}
\creflabelformat{equation}{#2(#1)#3}

\crefname{thm}{Theorem}{Theorems}
\Crefname{thm}{Theorem}{Theorems}

\crefname{lem}{Lemma}{Lemmas}
\Crefname{lem}{Lemma}{Lemmas}

\crefname{main}{Main Theorem}{Main Theorems}
\Crefname{main}{Main Theorem}{Main Theorems}

\crefname{cor}{Corollary}{Corollaries}
\Crefname{cor}{Corollary}{Corollaries}

\crefname{notn}{Notation}{Notations}
\Crefname{notn}{Notation}{Notations}

\crefname{prop}{Proposition}{Propositions}
\Crefname{prop}{Proposition}{Propositions}

\crefname{defn}{Definition}{Definitions}
\Crefname{defn}{Definition}{Definitions}

\crefname{ex}{Example}{Examples}
\Crefname{ex}{Example}{Examples}

\crefname{conj}{Conjecture}{Conjectures}
\Crefname{conj}{Conjecture}{Conjectures}

\crefname{rmk}{Remark}{Remarks}
\Crefname{rmk}{Remark}{Remarks}

\crefname{assump}{Assumption}{Assumptions}
\Crefname{assump}{Assumption}{Assumptions}

\crefname{diagram}{diagram}{diagrams}
\creflabelformat{diagram}{#2(#1)#3}
\crefname{sseq}{}{}
\creflabelformat{sseq}{#2(#1)#3}

\title{Equivariant Steenrod Operations}
\author{P. Bhattacharya}\address{New Mexico State University}\email{prasit@nmsu.edu}
\author{A. Waugh}\address{ University of Washington}\email{ajw48@uw.edu}
\author{M. Zeng}\address{Max Plank Institute for Mathematics}\email{mingcongzeng@gmail.com}
\author{F. Zou}\address{Chinese Academy of Sciences}\email{zoufoling@amss.ac.cn}

\thanks{}

\begin{document}

\begin{abstract} We introduce the notion of $\mathrm{R}$-Eulerian sequences for any $\mathcal{N}_\infty$-ring spectrum $\mathrm{R}$ of finite orientation order. We prove that each $\mathrm{R}$-Eulerian sequence determines a stable $\mathrm{R}$-cohomology operation. Furthermore, we show that the collection of $\mathrm{R}$-Eulerian sequences carries a natural additive and a multiplicative structure which is linear over the coefficient ring. As an application, we specialize to equivariant ordinary cohomology with coefficients in finite fields and construct genuine equivariant  Steenrod operations for all finite groups.
 \end{abstract}

\maketitle	
\tableofcontents

\section{Introduction} \label{Sec:intro}
In the 1950s, stable homotopy theory underwent a period of rapid development, driven by profound applications in geometry \cite{SteenWhite, ThomSt, ThomCob}. Key breakthroughs included Steenrod and Whitehead's work on the vector fields on spheres problem \cite{SteenWhite, SteenBook}, the calculation of cobordism rings \cite{ThomSt, ThomCob, MilCob}, the resolution of the Hopf invariant one problem \cite{AdamsHopf}, among many others. These results relied crucially on Steenrod operations, which Norman Steenrod introduced in 1947 \cite{SteenSq}. The usefulness of Steenrod operations  extends to differential topology via Wu's reformulation of Stiefel-Whitney classes \cite{Wu}, and in homotopy theory through the Adams spectral sequence \cite{AdamsSS}.  Over the years, the applications of Steenrod operations have only expanded,
establishing them as one of the most formidable tools in current homotopy theory research.

Equivariant homotopy theory is an extension of classical homotopy theory that is sensitive to symmetries. In 1967, Bredon introduced an equivariant cohomology theory that refines Borel cohomology by incorporating fixed-point data for all subgroups \cite{Bredon}. The coefficients for Bredon cohomology are provided by Mackey functors, which Dress introduced \cite{Dress}.  In 1981, Lewis-May-McClure \cite{LMM} extended the indexing system of a $\Gmr$-equivariant Bredon cohomology from integers to the real representation ring, usually denoted by $\RO(\Gmr)$, for any compact Lie group $\Gmr$. Equivariant stable homotopy theory was developed in the 1980s to address problems in equivariant geometry.  However, despite a few striking applications \cite{TomBordism, HK, KW, HHR, Pin}, its geometric utility has remained limited and sporadic. This constraint is largely due to the fact that equivariant Steenrod operations are not known beyond the group of order $2$ \cite{HK, VRed}.

The  $\RO(\Gmr)$-graded  Bredon cohomology with coefficients in $\ull{\FF}_p$ (the constant Tambara functor at the field of order $p$) is represented by a genuine $\Gmr$-equivariant $\EE_\infty^\Gmr$-ring spectrum $\Hmr \ull{\FF}_p$.  The {\bf  $\Gmr$-equivariant mod $p$ Steenrod algebra} is defined as the stable homotopy class of $\mr{G}$-equivariant self-maps  of $\Hmr\ull{\FF}_p$: 
\[ 
\Acal_{\Gmr, p}^{\star} := [\H\ull{\FF}_p, \H\ull{\FF}_p]^{\mr{G}}_{-\star}. 
\]
This Steenrod algebra is a module over its {\bf coefficient ring},  the $\RO(\Gmr)$-graded cohomology of a point: 
\[ 
\MM_{p}^{\Gmr} := \Hmr^{\star}_{\Gmr}(\mr{pt}_+; \ull{\FF}_p) \cong \pi_{-\star}^\Gmr \Hmr\ull{\FF}_p. 
\]
The calculation of this coefficient ring is a notoriously difficult problem.  After Stong's calculation of $\RO(\Cmr_2)$-graded cohomology of a point (see \cite[$\mathsection$2]{StongCalc}), progress was stalled for many years. However, a series of recent breakthroughs have revitalized the field: In  2017, Holler and Kriz \cite{HolK} calculated the coefficient ring for  $\Gmr = \Cmr_2^{\times n}$, soon after the third author identified the coefficient ring for  $\Gmr = \Cmr_{p^2}$ \cite{Zeng},  G. Yan \cite{Guoqi1}  for $\Gmr = \Cmr_{2^n}$, and Kriz-Lu \cite{KrizLu} and G. Yan \cite{Guoqi2} for the dihedral groups.

In 2001, Hu and Kriz, in their seminal work \cite{HK}, determined the structure of  $\Cmr_2$-equivariant dual Steenrod algebra. This was followed in 2003 by Voevodsky \cite{VRed}, who determined the $\RR$-motivic Steenrod algebra, whose Betti realization provided a complete description of the $\Cmr_2$-equivariant Steenrod algebra. The case $\Gmr= \Cmr_2$ (as well as the nonequivariant case) is often considered special, as its Steenrod algebra is free over its coefficient ring, a property that is crucial for its determination.  In fact, recent work by Sankar-Wilson \cite{WS} as well as Hu-Kriz-Somberg-Zou \cite{HKSZ} determines the $\Cmr_p$-equivariant mod $p$ dual Steenrod algebra as a module over the coefficient ring and shows that it is not free. This non-freeness may explain why the methods in \cite{SteenSq, VRed} did not adapt to identify the equivariant Steenrod algebra for other groups (also see \cite{HKSZ}).

\bigskip  
 In this paper, we introduce a general \emph{theoretical framework} to construct $\Gmr$-equivariant Steenrod operations for  $\Gmr$-equivariant cohomology theories. This framework requires the following assumptions:
 \begin{itemize}
 \item $\Gmr$ is a finite group.
   \item The cohomology theory is represented by a homotopy $\Ocal$-ring $\Rmr$ (as in \Cref{defn:EquivHinfty}),  for some $\Ncal_\infty$  $\Gmr$-operad  $\Ocal$ (in the sense of  \cite{Ninfty}). 
   \item The $\Gmr$-equivariant vector bundle of  \eqref{bundle:TautVF} admits an $\Rmr$-orientation (in the sense of \cite{BZ}). 
  \end{itemize}  
  Building on these conditions, we introduce the concept of a {\bf $\Vmr$-stable $\Rmr$-Eulerian sequence} (see \Cref{defn:Eulerian}), where $\Vmr$ is a finite orthogonal $\Gmr$-representation. Our main result is the following theorem:
\begin{main}[\Cref{thm:stable}] \label{main1} For  every  $\Vmr$-stable $\Rmr$-Eulerian sequence $\upchi$, there exists an $\Rmr$-cohomology operation 
\[ 
\begin{tikzcd}
\Sfrak^{\upchi} : \Rmr^{\star}(-) \rar & \Rmr^{\star + \| \upchi\|} (-),
\end{tikzcd}
\] 
of degree  $\| \upchi\|$ (see \Cref{defn:Eulerian}), where  $\star \in \RO(\Gmr, \Vmr)$ (see \Cref{notn:main}), which commutes with the $\Vmr$-suspension isomorphism 
 \begin{equation} \label{eqn:eqsusp}
\upsigma_{\Vmr}: \Rmr^{\star}(- ) \overset{\cong}\longrightarrow \Rmr^{\star + \Vmr}(\Sigma^{\Vmr}(-) ), 
\end{equation}
i.e., 
$\upsigma_\Vmr ( \Sfrak^{\upchi}(x))  = \Sfrak^{\upchi} ( \upsigma_\Vmr(x)) $
for any $\Rmr$-cohomology class $x$.  
\end{main}
\begin{defn} For a $\Gmr$-spectrum $\Rmr$, we say an $\Rmr$-cohomology operation is {\bf genuine stable} if it commutes with $\upsigma_{\uprho}$, where  $\uprho$ is the real regular representation of $\Gmr$. 
\end{defn}

In the nonequivariant case,  when $\Gmr$ is the trivial group  and $\Rmr$ is $\Hmr\FF_2$---the ordinary cohomology with $\FF_2$-coefficients, the vector bundle  \eqref{bundle:TautVF} can be chosen to be the tautological line bundle over $\B\Sigma_2 \simeq \RR\PP^\infty$.  We recall the graded module structure on the mod 2 homology of the classifying space: 
\[ \Hmr_* ((\Bmr\Sigma_2)_+; \FF_2) \cong \FF_2\{ {\bf b}_0, {\bf b}_1, \dots \},\]
where ${\bf b}_i$ are generators is degree $i$. As shown in  \Cref{rmk:classicalsl}, the $\Hmr\FF_2$-Eulerian sequences
\[ \upbeta[k] = (\overbrace{0, \dots, 0}^{\text{$k$}}, {\bf b}_0, {\bf b}_1, \dots )\] 
 generate the classical $k$-th Steenrod squaring operation $\Sq^k$ by \Cref{main1}. We also show that our framework recovers all classical odd primary Steenrod operations as well as $\Cmr_2$-equivariant Steenrod operations, as detailed in \Cref{rmk:classicalsl}  and \Cref{rmk:ctwoicalsl}.  In \Cref{sec:new}, we demonstrate the strength of our theory by constructing new $\Gmr$-equivariant Steenrod operations for all finite group $\Gmr$: 
\begin{main} \label{main2} Suppose $\Gmr$ is a finite group. Then for every $k \in \NN$  there exist genuine stable cohomology operations:
\[ 
\begin{tikzcd}
\Sq^{k \uprho_{\Gmr}}_{\uplambda}:  \Hmr^{\star}_\Gmr(- ; \ull{\FF}_2) \rar & \Hmr^{\star + k \uprho_{\Gmr}}_\Gmr(- ; \ull{\FF}_2)
\end{tikzcd}
\]
\[ 
\begin{tikzcd}
\Sq^{k \uprho_{\Gmr}+1}:  \Hmr^{\star}_\Gmr(- ; \ull{\FF}_2) \rar & \Hmr^{\star + k \uprho_{\Gmr} +1}_\Gmr(- ; \ull{\FF}_2)
\end{tikzcd}
\]
where $\lambda \in \mr{Irr}_1(\Gmr)$---the isomorphism classes of $1$-dimensional orthogonal $\Gmr$-representation. 

When $p$ is  an odd prime,  there exist genuine stable cohomology operations:
\[ 
\begin{tikzcd}
\Pmr^{2\epsilon k \uprho_{\Gmr}}_{\uplambda}:  \Hmr^{\star}_\Gmr(- ; \ull{\FF}_p) \rar & \Hmr^{\star + 2 \epsilon k \uprho_{\Gmr}}_\Gmr(- ; \ull{\FF}_p)
\end{tikzcd}
\]
\[ 
\begin{tikzcd}
\Pmr^{2\epsilon k \uprho_{\Gmr}+1}:  \Hmr^{\star}_\Gmr(- ; \ull{\FF}_p) \rar & \Hmr^{\star + 2 \epsilon k \uprho_{\Gmr} +1}_\Gmr(- ; \ull{\FF}_p)
\end{tikzcd}
\]
where $k \in \NN$, $\lambda \in \td{\mr{Irr}}_1(\Gmr)$---the isomorphism classes of complex $1$-dimensional orthogonal $\Gmr$-representation whose character factors through $\mr{C}_p \subset \mr{S}^1 \subset \CC^{\times}$, and 
\[ 
\epsilon = \left\lbrace \begin{array}{ccc}
(p-1)/2 & \text{if $|\Gmr|$ is even} \\
(p-1) & \text{if $|\Gmr|$ is odd.} 
\end{array} \right.
\]
\end{main}
\begin{rmk}  When $\Gmr $ is the trivial subgroup $\sfe$ in \Cref{main2}, then  the operations $ \Sq^{k \uprho_\sfe}_{1}$, $\Sq^{k \uprho_{\sfe} +1}$, $\Pmr^{2 \epsilon k \uprho_\sfe}_{1}$, $\Pmr^{2 \epsilon \uprho_{\sfe} +1 }$ are the classical Steenrod operations $\Sq^{k}$, $\Sq^{k+1}$, $ \Pmr^{k}$, $\beta \Pmr^{k}$ in the notation of \cite{SteenBook}, respectively. 
\end{rmk}
 An $\Rmr$-Eulerian sequence $\upchi$  is a sequence of homology classes. By restricting the action to a subgroup $\Kmr \subset \Gmr$, we define its restriction $\upiota_\Kmr(\upchi)$, which is an $\upiota_\Kmr\Rmr$-Eulerian sequence (see \Cref{defn:restrEulerianSeq}). The stable $\upiota_\Kmr \Rmr$-cohomology operation $\Sfrak^{\upiota_{\Kmr}(\upchi)}$ is then the restriction of the $\Sfrak^{\upchi}$ (see \Cref{thm:opsRestrAndGeoFixedPts}). From this we  notice that the underlying nonequivariant operations of the equivariant operations from \Cref{main2} are precisely the classical Steenrod operations:
\begin{main} \label{main3} Suppose $\Kmr$ is a subgroup of a finite group  $\Gmr$, and let $x \in \Hmr^{\star}_\Gmr(\Xmr; \ull{\FF}_p)$ denote an arbitrary cohomology class  for a $\Gmr$-space (or $\Gmr$-spectrum) $\Xmr$. Then: 
\begin{enumerate}
\item When $p =2$
\begin{enumerate}
\item $\upiota_{\Kmr*}(\Sq^{k \uprho_{\Gmr}}_{\lambda} (x))  = \Sq^{ |\Gmr/\Kmr| k \uprho_\Kmr}_{\upiota_\Kmr\lambda}( \upiota_{\Kmr \ast}(x)).$
\item $\upiota_{\Kmr*}(\Sq^{k \uprho_{\Gmr} +1} (x))  = \Sq^{ |\Gmr/\Kmr| k \uprho_\Kmr + 1}( \upiota_{\Kmr \ast}(x)) $. 
\end{enumerate}
\item When $p$ is odd 
\begin{enumerate}
\item $\upiota_{\Kmr*}(\Pmr^{2 \epsilon k \uprho_{\Gmr}}_{\lambda} (x))  = \Pmr^{ 2 \epsilon  |\Gmr/\Kmr| k \uprho_\Kmr}_{\upiota_\Kmr\lambda}( \upiota_{\Kmr \ast}(x)).$
\item $\upiota_{\Kmr*}(\Pmr^{2 \epsilon  k \uprho_{\Gmr} +1} (x))  = \Pmr^{ 2 \epsilon  |\Gmr/\Kmr| k \uprho_\Kmr + 1}( \upiota_{\Kmr \ast}(x)) $. 
\end{enumerate}
\end{enumerate}
\end{main}
\begin{notn} Given a subgroup $\Kmr$ of $\Gmr$, we let $\Nmr(\Kmr)$ denote the normalizer subgroup of $\Kmr$, and 
$\Wmr(\Kmr) = \Nmr(\Kmr)/\Kmr$ denote the Weyl group of $\Kmr$. 
\end{notn}
In \Cref{defn:geofixEulerSeq}, we define geometric fixed-points $\varphi^{\Kmr}(\upchi)$ of an $\Rmr$-Eulerian sequences $\upchi$. In \Cref{thm:opsRestrAndGeoFixedPts}, we  show that the geometric fixed-point of  $\Sfrak^{\upchi}$ on  a given cohomology class   is equal to  $\Sfrak^{\varphi^\Kmr}(\upchi)$ on the geometric fixed-point of that class. While this result is satisfying,  it does not quite compare the operations introduced in \Cref{main2} across  geometric fixed-point functors. This is because $\Kmr$-geometric fixed-point of  $\Hmr\ull{\FF}_p \in \Sp^\Gmr$, denote it by $\Phi^{\Kmr}(\Hmr\ull{\FF}_p)$,   is \emph{not} equivalent to   $\Hmr\ull{\FF}_p \in \Sp^{\Wmr(\Kmr)}$. However, $\Hmr\ull{\FF}_p$ is a split summand of $\Phi^\Kmr(\Hmr\ull{\FF}_p)$  as an $\EE^{\Wmr(\Kmr)}_\infty$-ring spectrum.  This leads us to consider a modified $\Kmr$-geometric fixed-point functor: 
\[ 
\begin{tikzcd} 
\td{\varphi}^\Kmr: \Hmr^{\Gmr}_{\star}(-; \ull{\FF}_p) \rar &  \Hmr^{\Wmr(\Kmr)}_{\star}(-; \ull{\FF}_p). 
\end{tikzcd}
\]
In \Cref{defn:modgeoES}, we define the modified geometric $\Kmr$-fixed-point $\td{\varphi}^\Kmr(\upchi)$ of an $\Hmr\ull{\FF}_p$-Eulerian sequence $\upchi$. We summarize the relation between $\Sfrak^{\upchi}$ and $\Sfrak^{\td{\varphi}^\Kmr(\upchi)}$ in \Cref{thm:modgeoES}. From this result we conclude: 
\begin{main} \label{main4}  Let $\Kmr$ be a subgroup of a finite group $\Gmr$.  For a $\Gmr$-space (or $\Gmr$-spectrum) $\Xmr$, let $x \in \Hmr^{\star}_\Gmr(\Xmr; \ull{\FF}_p)$ be a cohomology class. Then:
\begin{enumerate}
\item $\td{\varphi}^{\Kmr}\left( \Sq^{k \uprho_\Gmr}_{\lambda}(x) \right) = \Sq^{k |\Gmr/\Nmr(\Kmr)| \uprho_{\Wmr(\Kmr)}}_{\lambda^\Kmr}\left(\td{\varphi}^\Kmr(x) \right)$ when $p=2$
\item $\td{\varphi}^{\Kmr}\left( \Pmr^{ 2 \epsilon k \uprho_\Gmr}_{\lambda}(x) \right) = \Pmr^{2 \epsilon k |\Gmr/\Nmr(\Kmr)| \uprho_{\Wmr(\Kmr)}}_{\lambda^\Kmr}\left(\td{\varphi}^\Kmr(x) \right)$ when $p$ is odd
\end{enumerate}
where it is assumed that $\Sq^{k \uprho_{\Wmr(\Kmr)}}_{\lambda^\Kmr}$ and $\Pmr^{2 \epsilon k \uprho_{\Wmr(\Kmr)}}_{\lambda^\Kmr}$ are trivial operations when $\lambda^\Kmr = {\bf 0}$. 
\end{main}
The operations in \Cref{main2} are derived from $\H\ull{\FF}_p$-Eulerian sequences in the homology of $\Bmr_{\Gmr}\Sigma_p$. Calculation of $\Hmr_{\star}^\Gmr(\Bmr_{\Gmr}\Sigma_p; \ull{\FF}_p)$ is an extremely difficult problem  and is largely unsolved for groups larger than $\Cmr_2$. The technical part of this paper identifies infinite families of homology classes in $\Hmr_{\star}^\Gmr(\Bmr_{\Gmr}\Sigma_p; \ull{\FF}_p)$ that are specifically designed to form $\Hmr\ull{\FF}_p$-Eulerian sequences. These homology classes can be tracked along restrictions and geometric fixed-points  which leads  to the results in \Cref{main3} and \Cref{main4}. When $\Gmr = \Cmr_2$, the homology groups $\Hmr_{\star}^{\Cmr_2}(\Bmr_{\Cmr_2}\Sigma_2; \ull{\FF}_2)$ are fully known \cite{HK} (also see \eqref{eqn:cohHK} and \Cref{rmk:ctwoicalsl}) and  we show that our list of $\H\ull{\FF}_2$-Eulerian sequences in $\Hmr_{\star}^{\Cmr_2}(\Bmr_{\Cmr_2}\Sigma_2; \ull{\FF}_2)$ is complete. However, recent unpublished calculations of 
$\Hmr_{\star}^{\Cmr_4}(\Bmr_{\Cmr_4}\Sigma_2; \ull{\FF}_2)$ \cite{NickGProj} reveal that our list of  $\Hmr\ull{\FF}_2$-Eulerian sequences for $\Gmr = \Cmr_4$ is far from complete. Consequently, we do not expect the cohomology operations of \Cref{main2} to generate the full set of $\Gmr$-equivariant Steenrod operations for groups larger than $\Cmr_2$. 

The structural properties of  classical and $\Cmr_2$-equivariant Steenrod \linebreak operations---namely, the Cartan formula, the Adem relations, and the total squaring operation---make them a potent tool. Given the difficult nature of calculating $\Hmr_{\star}^{\Gmr}(\Bmr_{\Gmr}\Sigma_p; \ull{\FF}_p)$, our strategy is to pursue an abstract formulation of the Cartan formula and Adem relations solely through the Eulerian sequence framework. This will ensure that any new $\Gmr$-equivariant Eulerian sequence that is discovered will provide direct insight into the structural properties of the $\Gmr$-equivariant Steenrod algebra.

 In \Cref{thm:GenCartan}, we establish a generalized Cartan formula for Eulerian sequences that remains applicable even without a K\"unneth isomorphism. Furthermore, we develop the framework for defining Adem relations purely in terms of Eulerian sequences. Our \Cref{defn:Eulerian} provides a very general definition for Eulerian sequences, which, among other applications, allows us to define Eulerian sequences in $\Hmr_{\star}^{\Gmr}(\Bmr_{\Gmr}\Sigma_n; \ull{\FF}_p)$ for all $n \in \NN$.   We refer to these as $\uprho_\Gmr$-stable $\Hmr\ull{\FF}_p$-Eulerian sequences of weight $n$, and denote the collection of such integral sequences by $\Ecal^{(n)}_{\Gmr, p}$ (see \Cref{defn:intES}). We then define a strictly associative product
  \[ 
 \begin{tikzcd}
\odot :\Ecal^{(n)}_{\Gmr, p} \times \Ecal^{(m)}_{\Gmr, p} \rar & \Ecal^{(nm)}_{\Gmr, p}
 \end{tikzcd}
 \]
 and show that $\Sfrak^{\upchi_1 \odot \upchi_2} = \Sfrak^{\upchi_1} \circ \Sfrak^{\upchi_2}$ (see \Cref{thm:composeES}). This product thus realizes the composition of genuine stable $\Hmr\ull{\FF}_p$-cohomology operations.  The Adem relations arise from the fact that the map
 \[ 
 \begin{tikzcd} 
 \Sigma_{n} \times \Sigma_n \rar & \Sigma_n \wr \Sigma_n \rar & \Sigma_{n^2}
 \end{tikzcd}
 \]
 and its composition with the twist map on $\Sigma_n \times \Sigma_n$ are conjugates.   We are currently investigating if we can use this fact  to describe Adem relations abstractly in terms of Eulerian sequences avoiding explicit calculations of $\Hmr_{\star}^{\Gmr}(\Bmr_{\Gmr}\Sigma_n; \ull{\FF}_p)$. 
 
 Our theory, which is sensitive to $\Ncal_\infty$-ring structures, applies to a wide range of equivariant and nonequivariant cohomology theories. Nonequivariantly, it is applicable to any cohomology theory represented by an $\EE_\infty$-ring spectrum with finite  orientation order (as defined in \cite{BC}). This includes complex oriented theories such as  $\Hmr\ZZ/p^i$,$\Hmr\ZZ$, $\ku$ and Morava $\Emr$-theory, along with real $\Kmr$-theory, topological modular forms, Johnson-Wilson theories and $\EO$-theories.  While we do not know if every stable cohomology operation for an $\Ncal_\infty$-ring $\Rmr$ can be obtained from an $\Rmr$-Eulerian sequence using \Cref{main1}, this is the case for $\Hmr\ull{\FF}_p$ when $\Gmr = \sfe$ and $\Gmr = \Cmr_2$.  We therefore conjecture:  
 \begin{conj} \label{conj:genall} The collection of genuine stable $\uprho_\Gmr$-stable cohomology operations 
 \[  {\bf S}_{\Gmr, p} := \{ \Sfrak^{\upchi} : \upchi \in \bigsqcup_i \Ecal^{(p^i)}_{\Gmr, p}\} \]
 generate $ \Acal_{\Gmr, p}^{\star}$ for all finite group $\Gmr$ at all prime $p$. 
 \end{conj}
One can  formulate a stronger version of this conjecture. First, note that the classical Steenrod algebra $\Acal^{*}_p$ and the $\Cmr_2$-equivariant Steenrod algebra $\Acal_{\Cmr_2, p}^{\star}$ are multiplicatively generated over the coefficient ring by the stable cohomology operations arising from $\Ecal^{(p)}_{\Gmr, p}$.  This motivates the following question for future investigation: 
\begin{question} Does there exist an $n \in \NN$ such that \[  {\bf S}_{\Gmr, p}\langle n \rangle:= \{ \Sfrak^{\upchi} : \upchi \in \bigsqcup_{i=1}^{n} \Ecal^{(p^i)}_{\Gmr, p} \} \] generate the algebra $\Acal_{\Gmr, p}^{\star}$ for all finite group $\Gmr$? 
\end{question}

\begin{notn} \label{notn:main} Throughout this paper: 
\begin{itemize}
\item $\Gmr$ is a finite group, 
\item $\uprho_{\Gmr}$ is the regular representation of $\Gmr$ (we drop the subscript when the underlying group is clear from the context), 
\item $ \uptau_n$ and $\td{\uptau}_n$ denote the permutation and the standard representations of $\Sigma_n$ respectively, 
\item $\Ucal_{\Gmr}$ denotes the complete $\Gmr$-universe, 
\item $\Vmr$ denotes an orthogonal $\Gmr$-representation which contains the trivial sub-representation $\RR$,
\item $\Ucal_{\Gmr, \Vmr}$ denotes the sub $\Gmr$-universe generated by the $\Gmr$-representation $\Vmr$, 
 \item $\Sp_{\Gmr}$ be a  category of orthogonal $\Gmr$ spectra in the universe $\Ucal_{\Gmr}$, 
 \item $\upiota_\Kmr: \Sp_\Gmr \longrightarrow \Sp_\Kmr$ will denote the restriction functor for the subgroup $\Kmr$,
 \item $\Phi^\Kmr:  \Sp_\Gmr \longrightarrow \Sp_{\Wmr(\Kmr)}$ denote the geometric fixed point functor for the subgroup $\Kmr$, where $\Wmr(\Kmr)$ denotes the Weyl group of $\Kmr$ in $\Gmr$, 
  \item $\RO(\Gmr, \Vmr)$ denote the subring of $\RO(\Gmr)$ generated by $\Ucal_{\Gmr, \Vmr}$. 
\end{itemize}
\end{notn}
\begin{conv} Throughout this paper, all $\Gmr$-spaces will be assumed to have a  basepoint.  We denote the reduced Bredon homology and cohomology with coefficients in the Mackey functor $\ull{\Amr}$ by
\[ \Hmr_{\star}^\Gmr(- ; \ull{\Amr}) \text{ and } \Hmr^{\star}_\Gmr(- ; \ull{\Amr} ),  \]
respectively. Unreduced versions will be indicated by adding a disjoint basepoint.  
\end{conv}
\subsection*{Acknowledgement} This paper have benefitted from conversations with Samik Basu, Hood Chatham, Bert Guillou, Mike Hill, Tyler Lawson,  Ang Li, Guchuan Li, and Mike Mandell. 

This research is supported by NSF grant DMS-2305016. 
\subsection*{Organization of the paper} In \Cref{sec:power}, we introduce equivariant power operations. In \Cref{sec:shiftedpower}, we use orientations of certain equivariant bundles to define shifted power operations. 

In \Cref{sec:Euler}, we introduce the theory of Eulerian sequences and prove \Cref{main1}. We also introduce restrictions, geometric fixed-points , and modified geometric fixed-points of Eulerian sequences. We discuss Cartan formula for Eulerian sequences in \Cref{sec:Cartan}. 

In \Cref{sec:compose}, we introduce the notion of $\Vmr$-shifted homotopy $\Ncal_\infty$-rings and use it define composition of Eulerian sequences. 

In \Cref{sec:new}, we identify new $\Hmr\ull{\FF}_p$-Eulerian sequences and prove \Cref{main2}, \Cref{main3}, and \Cref{main4}.

\section{Power operations associated to $\Ncal_\infty$ ring spectra}\label{sec:power}

Classical Steenrod operations are constructed from power operations, which are defined using the $\EE_\infty$-structure of $\Hmr\FF_2$.  In equivariant homotopy theory, $\Ncal_\infty$-operads  generalize  their nonequivariant $\EE_\infty$-operad counterparts. First studied in \cite{Ninfty}, these operads differ from their nonequivariant counterparts in that there can be multiple distinct homotopy classes for a given finite group $\Gmr$ (see \cite{RubinNinftyOperads}, \cite{EquiCommutativeGutierrezWhite}, \cite{GenEquiOperadsBonventrePereira}). 

In this section, we first generalize the extended power construction (see \Cref{defn:extpower}) to equivariant settings. We  then use  the multiplicative $\Ncal_\infty$-structure of a ring spectrum $\Rmr$ to define equivariant generalizations of the classical power operations (see \Cref{defn:powerF}). We also study the interaction of equivariant power operations with restriction, fixed-points, geometric fixed-points, and modified geometric fixed-points. First, we recall some standard definitions from equivariant homotopy theory. 
\begin{defn} \label{defn:family}
 A collection of subgroups $\mathscr{F}$ of a group $\Gamma$ is called a {\bf family} if it is closed with respect to subgroups up to conjugation.
We will call a  family  $\mathscr{F}$ of $ \Gamma = \Gmr \times \Pi$  a  {\bf $\Gmr$-closed family} if  it contains all subgroups of the form  $\Kmr \times \{ {\sf e} \}$. 
\end{defn}
\begin{notn} \label{notn:univGspace} For every family $\mathscr{F}$ of $\Gamma$, let $\mr{E}\mathscr{F} $ denote the universal $\Gamma$-space  satisfying 
\[ 
(\mr{E}\mathscr{F})^{\Hmr} \simeq  \left\lbrace
\begin{array}{cl}
 \ast & \text{$\Hmr \in \mathscr{F}$} \\
 \emptyset & \text{otherwise}.
\end{array}
\right.
\]
When $\Fscr$ consists of all subgroups  $\Gmr \times \Pi$ such that its intersection with $1 \times \Pi$ is the trivial group, then we use $\Emr_\Gmr \Pi$ to denote $\Emr \Fscr$. 
\end{notn}

\subsection{Extended powers and power operations}
\begin{defn} \label{defn:extpower}
  Let $\Xmr$ be a $\Gmr$-space or a $\Gmr$-spectrum, and let $\Tmr$ be a finite $\Pi$-set. For a $(\mr{G} \times \Pi)$-family $\mathscr{F}$, the {\bf $(\mathscr{F}, \Tmr)$-th extended power of $\Xmr$} is defined as the $\Gmr$-space (or $\Gmr$-spectrum) given by the formula:
  \[
  \Dmr^{\mathscr{F}}_{\Tmr}(\Xmr) := (\Emr \mathscr{F})_+ \sma_{\Pi} (\Xmr^{\sma \Tmr})
  \]
  where the $\Gmr$-action is the diagonal action.
\end{defn}

\begin{rmk}
  \label{rmk:genuine-naive}
  When $\Xmr$ is a $\Gmr$-spectrum, the object $\Xmr^{\sma \Tmr}$ is regarded as a $(\Gmr \times \Pi)$-spectrum. The defining feature is that its underlying $\Gmr$-spectrum is genuine, while its $\Pi$-spectrum structure is naive. Specifically, the universe of $\Xmr^{\sma \Tmr}$ is generated by finite-dimensional orthogonal representations of the form $\alpha \otimes \epsilon$, where $\alpha$ is a $\Gmr$-representation and $\epsilon$ is a trivial representation of $\Pi$.
\end{rmk}
 Since the diagonal map 
\[ 
\begin{tikzcd}
\Delta:\mr{E} \mathscr{F}  \rar & \mr{E}(\mathscr{F} \times \mathscr{F}) \rar & \mr{E} \mathscr{F} \times \mr{E} \mathscr{F}
\end{tikzcd}
\]
is  $\mr{G} \times \Pi$-equivariant, we get the following lemma. 
 \begin{lem} \label{lem:delpair1} Given a $\Pi$-set $\Tmr$, there exists a   $\Gmr$-equivariant natural  map  
\begin{equation} \label{eqn:extendedoplax}
\begin{tikzcd}
\partialup_{\Tmr}^{\mathscr{F}}: \mr{D}^{\mathscr{F}}_{\mr{T}}(\Xmr \sma \Ymr) \rar & \mr{D}^{\mathscr{F}}_{\mr{T} }( \Xmr) \sma \mr{D}^{\mathscr{F}}_{\mr{T} }( \Ymr)
\end{tikzcd}
\end{equation}
for any pair of $\Gmr$-spaces (or $\Gmr$- spectra) $\Xmr$ and $\Ymr$ that satisfies the external associativity condition. This means that the diagram 
\begin{equation} 
\label{diag:extcom}
\begin{tikzcd}
\mr{D}^{\mathscr{F}}_{\mr{T}}(\Xmr \sma \Ymr \sma \Zmr) \rar \dar  & \mr{D}^{\mathscr{F}}_{\mr{T}}(\Xmr \sma \Ymr) \sma \mr{D}^{\mathscr{F}}_{\mr{T}}(\Zmr)  \dar  \\
\mr{D}^{\mathscr{F}}_{\mr{T}}(\Xmr \sma \Ymr)  \sma \mr{D}^{\mathscr{F}}_{\mr{T}}(\Zmr) \rar & \mr{D}^{\mathscr{F}}_{\mr{T}}(\Xmr) \sma \mr{D}^{\mathscr{F}}_{\mr{T}}( \Ymr)  \sma \mr{D}^{\mathscr{F}}_{\mr{T}}(\Zmr) 
\end{tikzcd}
\end{equation}
commutes for any triplet $(\Xmr, \Ymr, \Zmr)$. 
\end{lem}
\begin{proof}The map $\partialup_{\Tmr}^{\mathscr{F}}$, which is defined as  the composite map 
\[ 
\begin{tikzcd}
\mr{D}^{\mathscr{F}}_{\mr{T}}(\Xmr \sma \Ymr) := \Emr \mathscr{F}_+ \sma_{\Pi} (\Xmr \sma \Ymr)^{\sma \Tmr} \dar["\Delta \sma_\Pi \mr{Id}"] \\
(\Emr \mathscr{F} \times \Emr \mathscr{F})_+ \sma_{\Pi} (\Xmr \sma \Ymr)^{\sma \Tmr} \dar \\
  (\Emr \mathscr{F}_+ \sma_{\Pi} \Xmr ^{\sma \Tmr} ) \sma( \Emr \mathscr{F}_+ \sma_{\Pi} \Ymr^{\sma \Tmr} )=: \mr{D}^{\mathscr{F}}_{\mr{T} }( \Xmr) \sma \mr{D}^{\mathscr{F}}_{\mr{T} }( \Ymr),
\end{tikzcd}
\] 
 satisfies \eqref{diag:extcom} as the diagonal map $\Delta$ satisfies an external associativity condition.
\end{proof}

\begin{notn} \label{notn:Ofam}Suppose $\Ocal$ is an $\Ncal_{\infty}$-operad.  Let $\mathscr{F}_n(\Ocal)$ denote  the $\Gmr$-closed family  of $\mr{G} \times \Sigma_n$ 
such that $ \mr{E}{\mathscr{F}}_n(\Ocal)$ is equivalent to $\Ocal(n)$, the $n$-th space of $\Ocal$. 
\end{notn}
The collection $\{\mathscr{F}_n(\Ocal)\}_{n \in \NN}$ must satisfy certain compatibility criteria, which gives rise to the operadic structure maps of $\Ocal$:
\begin{equation} \label{map:structure} 
\begin{tikzcd}
\mu : \Ocal(k) \times (\Ocal(n_1) \times \cdots \times \Ocal(n_k))  \rar & \Ocal(n_1 + \dots + n_k), 
\end{tikzcd}
\end{equation}
where $\Ocal(i)$ denotes the $i$-th space of $\Ocal$. This compatibility is encoded using a symmetric monoidal coefficient system $\ull{\Ccal}(\Ocal)$, called the \textbf{indexing system} of $\Ocal$.
\begin{notn} \label{notn:D} Let   ${\sf n}$ denote the set $\{ 1, 2, \dots, n \}$ on which $\Sigma_n$ acts by permutation. For an  $\Ncal_{\infty}$ $\Gmr$-operad $\mathcal{O}$,   we will use the abbreviation $\Dmr_{\sf n}^{\Ocal}(\Xmr) := \Dmr^{\mathscr{F}_{n}(\Ocal)}_{ \sf n}(\Xmr)$.
\end{notn}
Let $\Rmr \in \Sp_{\Gmr}$ be a spectrum. We call $\Rmr$ an $\Ncal_\infty$-ring if it is an algebra over some $\Ncal_\infty \Gmr$-operad $\Ocal$. By definition, an $\Ocal$-algebra $\Rmr \in \Sp_{\Gmr}$ is equipped with a compatible family of $\Gmr$-equivariant structure maps:
\[ 
\begin{tikzcd}
\uptheta_n^\Rmr: \Dmr_{\sf n}^{\Ocal}(\Rmr) \rar & \Rmr 
\end{tikzcd}
\]
 for each  $n \in \NN$. 

Given a group homomorphism  $ \upkappa: \Pi \longrightarrow \Sigma_n,$ one obtains a $\Gmr \times \Pi$-family  $(1 \times \upkappa)^*(\mathscr{F}_n(\Ocal))$ by pulling back the $\Gmr \times \Sigma_n$-family $\mathscr{F}_n(\Ocal)$ along the map $1 \times \upkappa$.
 For any sub-family $\mathscr{F}$  of $(1 \times \upkappa)^*(\mathscr{F}_n(\Ocal))$, we define a composite map
\begin{equation}
\begin{tikzcd}
\uptheta_{ \mathscr{F}, \upkappa}^{ \Rmr} : \Dmr^{\mathscr{F}}_{\upkappa^*{\sf n}}(\Rmr) \rar & \Dmr_{\sf n}^{\Ocal}(\Rmr) \rar["\uptheta^\Rmr_n"] & \Rmr,
\end{tikzcd}
\end{equation}
whose initial map is induced by  the sequence of maps 
\[ 
\begin{tikzcd}
\mr{E} \mathscr{F} \rar & \mr{E} ((1 \times \upkappa)^*\mathscr{F}_n(\Ocal))  \rar &   \mr{E}  \mathscr{F}_n(\Ocal)
\end{tikzcd}
\]
combined with the $\Pi$-equivariant map  $\Rmr^{\sma \upkappa^*{\sf n}} \longrightarrow \Rmr^{\sma {\sf n}} $. 
\begin{defn} \label{defn:powerF} Let $\Ocal$ be an $\Ncal_\infty$ $\Gmr$-operad, $\upkappa: \Pi \to \Sigma_n$ denote a group homomorphism and $\mathscr{F}$ be a sub-family of the  $\mr{G} \times \Pi$-family $(1 \times \upkappa)^*(\mathscr{F}_n(\Ocal))$. Then the {\bf $\mathscr{F}$-th  power operation} of  an $\Ocal$-ring  $\Rmr$ is the  natural map  
\[ 
\begin{tikzcd}
\Pcal^{\mathscr{F}} : \Rmr^0(-) \rar & \Rmr^0( \Dmr^{\mathscr{F}}_{ \upkappa^*{\sf n}}(-) )
\end{tikzcd}
\]
which sends  $x: \Xmr \to \Rmr$ to the composite 
\begin{equation} \label{eq:power}
\begin{tikzcd}
\Pcal^{\mathscr{F}}(x): \Dmr^{\mathscr{F}}_{ \upkappa^*{\sf n}}(\Xmr) \ar[rr, " \Dmr^{\mathscr{F}}_{ \upkappa^*{\sf n}}( x)"] &&  \Dmr^{\mathscr{F}}_{\upkappa^*{\sf n}}(\Rmr) \ar[rr,"\uptheta_{ \mathscr{F}, \upkappa}^{ \Rmr} "] && \Rmr 
\end{tikzcd}
\end{equation}
in ${\sf Ho}(\Sp_{\Gmr}) $. 
\end{defn}

\subsection{Restrictions and fixed points of power operations} \label{subsec:Powerres} \ 

 Suppose $\mathscr{F}$ is a $\mr{G}$-closed family of $\Gmr \times \Pi$ and  $\Kmr$ a subgroup of  $\Gmr$.  The {\bf restriction} of  $\mathscr{F}$ to $\Kmr$, defined as 
\[ 
\upiota_\Kmr \mathscr{F} := \{ \mr{F} \cap (\Kmr \times \Pi): \mr{F} \in \mathscr{F}   \} 
\]
is an $\mr{K}$-closed family of $\Kmr \times \Pi$. The {\bf $\Kmr$-fixed point family} of $\mathscr{F}$ is given by 
\[ 
\mathscr{F}^\Kmr := \{ \left( \mr{F} \cap (\mr{N}(\mr{K})\times \Pi) \right)/(\mr{K} \times \{ 1\}):  \Kmr \times \{ 1 \}  \subset \Fmr \in \Fscr \}
\]
and is a $\mr{W}(\mr{K})$-closed family.  These constructions lead to the following equivalences: 
\begin{itemize}
\item[--] $\Emr (\upiota_\Kmr \mathscr{F}) \simeq \upiota_{\Kmr}(\Emr  \mathscr{F})$ as  $(\Kmr \times \Pi)$-spaces.
\item[--] $(\mr{E} \mathscr{F})^{ \Kmr} \simeq \mr{E}(\mathscr{F}^{\Kmr})$ as  $\mr{W}(\Kmr) \times \Pi$-spaces.
\end{itemize}
 Therefore, for an $\Ncal_\infty$ $\Gmr$-operad $\Ocal$, the collections 
\begin{itemize}
  \item[--] $ \upiota_\Kmr\mathcal{O} := \{\upiota_\Kmr\mathcal{O}(n)  \}_{n \in \NN} $ 
  \item[--]  $\mathcal{O}^{\mr{K}} := \{\mathcal{O}(n)^{\mr{K}} \}_{ n \in \NN}$
 \end{itemize}
 form an $\Ncal_{\infty}$ $\Kmr$-operad and an $\Ncal_{\infty}$ $\mr{W}(\Kmr)$-operad, respectively. 
 
 Suppose $\Rmr$ is an $\Ocal$-ring in $\Sp_\Gmr$. Then its restriction 
$\upiota_\Kmr\Rmr$ is an $\upiota_\Kmr\Ocal$-ring in $\Sp_{\Kmr}$ with structure maps
\[ 
\begin{tikzcd}
\uptheta_{n}^{\upiota_\Kmr \Rmr} : \Dmr_{\sf n}^{\upiota_\Kmr \Ocal}(\upiota_\Kmr\Rmr) \cong \upiota_\Kmr\Dmr_{\sf n}^{\Ocal}(\Rmr) \ar[rrr, "\upiota_\Kmr (\uptheta_n^\Rmr) "] &&& \upiota_\Kmr \Rmr,
\end{tikzcd}
\]
and its $\Kmr$-fixed points  $\Rmr^{\Kmr}$ is an $\Ocal^\Kmr$-ring in $\Sp_{\mr{W}(\Kmr)}$ with structure maps 
\[ 
\begin{tikzcd}
\uptheta_{n}^{ \Rmr^\Kmr} : \Dmr_{\sf n}^{ \Ocal^\Kmr}(\Rmr^\Kmr) \rar["\uplambda"] &  \Dmr_{\sf n}^{\Ocal}(\Rmr)^\Kmr \ar[rrr, "(\uptheta_n^\Rmr)^\Kmr "] &&& \Rmr^\Kmr,
\end{tikzcd}
\]
 where $n \in \NN$ and  $\uplambda$ is the map defined in  \eqref{map:lambda}. 
\begin{rmk} \label{rmk:lambda} For any space with an action of $\Kmr \times \Sigma_n$  its  $\Sigma_n$-orbits of $\Kmr$-fixed points is a subspace of $\Kmr$-fixed points of  $\Sigma_n$-orbits. Thus, there is a natural map of the form  
\begin{equation} \label{map:lambda}
\begin{tikzcd}
\uplambda: \Dmr_{\sf n}^{ \mathscr{F}^\Kmr}([-]^\Kmr) \rar & { [\Dmr_{\sf n}^{ \mathscr{F}}( -)]^\Kmr  } 
\end{tikzcd}
\end{equation} 
in  $\Top^{\Wmr(\Kmr)}_*$. This natural map extends to $\Sp_{\Gmr}$ because the $\Kmr$-fixed points functor is lax monoidal. 
\end{rmk}
\begin{notn} For $x \in \Rmr^0_\Gmr(\Xmr) \cong [\Xmr, \Rmr ]^\Gmr$, let
\[
\begin{tikzcd}
 \upiota_{\Kmr*}(x) : \upiota_{\Kmr}\Xmr \rar & \upiota_{\Kmr}\Rmr 
 \end{tikzcd} \  \in   \Rmr^0_\Kmr(\upiota_\Kmr\Xmr) = [\upiota_\Kmr \Xmr, \upiota_\Kmr \Rmr ]^\Kmr \]
 denote the restriction of $x$  to the subgroup $\Kmr$, and 
  \[
  \begin{tikzcd}
   x^\Kmr : \Xmr^\Kmr \rar & \Rmr^\Kmr 
   \end{tikzcd}
  \  \in   (\Rmr^{\Kmr})^{0}_{\Wmr(\Kmr)} ( \Xmr^\Kmr)
   \]  is the map induced by $x$ on $\Kmr$-fixed points. 
\end{notn}
\begin{assump} \label{assump:RF}  Let 
\begin{itemize}
\item $\Ocal$ be an $\Ncal_\infty$-operad, 
\item  $\Rmr$   an $\Ocal$-ring, 
\item $\upkappa: \Pi \to \Sigma_n$ a group homomorphism, and 
\item $\mathscr{F}$ a sub $(\Gmr \times \Pi)$-family of $(1 \times \upkappa)^*\mathscr{F}_n(\Ocal)$.
\end{itemize}
\end{assump}
\begin{lem} \label{lem:power-fix} Suppose $\Rmr$, $\upkappa$, $\Fscr$ are as in \Cref{assump:RF} and $\Kmr$ be a subgroup of $\Gmr$.
 Then for any $\Xmr \in \Sp_{\Gmr}$ and  $x \in \Rmr^0_\Gmr(\Xmr)$
  \begin{enumerate}
\item $
\upiota_{\Kmr*} (\Pcal^{\mathscr{F}}(x)) = \Pcal^{\upiota_\Kmr\mathscr{F}}( \upiota_{\Kmr*}(x)),
$
\item $\uplambda^* (\Pcal^{\mathscr{F}}(x)^{\Kmr}) = \Pcal^{\mathscr{F}^\Kmr}( x^{\Kmr}) $, 
\end{enumerate}
where $\uplambda$ is the natural map of \Cref{rmk:lambda}. 
\end{lem}
\begin{proof} The claim \emph{(1)} follows from the fact  that there is a natural equivalence 
 \[ 
 \begin{tikzcd}
 \Dmr^{\upiota_{\Kmr}\mathscr{F}}_{ \upkappa^*{\sf n}}( \upiota_{\Kmr}(-)) \rar["\simeq"] & \upiota_{\Kmr} \left(\Dmr^{\mathscr{F}}_{ \upkappa^*{\sf n}}(-) \right).
 \end{tikzcd}
 \] 
  Whereas,  \emph{(2)} follows from the naturality of $\uplambda$ (see \Cref{rmk:lambda}) in that we have a commutative diagram 
\begin{equation} \label{eqn:Dfix} 
\begin{tikzcd}
\Dmr^{\mathscr{F}^\Kmr}_{ \upkappa^*{\sf n}}(\Xmr^\Kmr) \dar[color = DarkBrown, "\uplambda"'] \ar[ " \Dmr^{\mathscr{F}^\Kmr}_{ \upkappa^*{\sf n}}( x^\Kmr)", color = DarkBlue, rr] && \dar["\uplambda"']  \Dmr^{\mathscr{F}^\Kmr}_{\upkappa^*{\sf n}}(\Rmr^\Kmr) \ar[color =DarkBlue, rrd,"\uptheta_{ \mathscr{F}^\Kmr, \upkappa}^{ \Rmr^\Kmr} "] \\
\Dmr^{\mathscr{F}}_{ \upkappa^*{\sf n}}(\Xmr)^\Kmr \ar[rr, "\Dmr^{\mathscr{F}}_{ \upkappa^*{\sf n}}( x)^\Kmr"', color = DarkBrown] &&  \Dmr^{\mathscr{F}}_{\upkappa^*{\sf n}}(\Rmr)^\Kmr \ar[color = DarkBrown, rr,"(\uptheta_{ \mathscr{F}, \upkappa}^{ \Rmr})^\Kmr "'] && \Rmr^\Kmr
\end{tikzcd}
\end{equation}
in which the composition of red arrows represents $\uplambda^* (\Pcal^{\mathscr{F}}(x)^{\Kmr}) $ and the composition of blue arrows represents $\Pcal^{\mathscr{F}^\Kmr}( x^{\Kmr}) $. 
\end{proof}

\subsection{Power operations and geometric fixed-points.} \label{subsec:Powerfix} \ 
 
In equivariant homotopy theory, the $\Kmr$-fixed points functor does not commute with $\Sigma_\Gmr^{\infty}: \Top^{\Gmr}_* \to \Sp_\Gmr$ and the error term is explained by the tom Dieck splitting. One of the summands of $(\Sigma_\Gmr^{\infty}\Xmr )^\Kmr$ is defined using the $\Kmr$-geometric fixed point functor 
\[ 
\begin{tikzcd}
\Phi^\Kmr : \Sp_{\Gmr} \rar & \Sp_{\Wmr(\Kmr)}, 
\end{tikzcd}
\]
which is given by  the formula (also see \Cref{notn:familytilde})
\[ 
\Phi^\Kmr (\Rmr) := \left(\widetilde{\Emr\mathscr{P}_{\Kmr}} \sma \upiota_{\Nmr(\Kmr)}\Rmr \right)^\Kmr,
\]
where $\mathscr{P}_{\Kmr}$ is the $\Nmr(\Kmr)$-family consisting of those subgroups which  do not contain $\Kmr$. 
The functor $\Phi^\Kmr$ is a symmetric monoidal functor and commutes with $\Sigma_\Gmr^{\infty}$,  i.e., \[ \Phi^\Kmr(\Sigma_\Gmr^{\infty}\Xmr) \simeq \Sigma_{\Wmr(\Kmr)}^{\infty}(\Xmr^\Kmr) \] for any $\Gmr$-space $\Xmr$.

\begin{notn} \label{notn:familytilde} Let  $\widetilde{\Emr\mathscr{F}}$ denote the cofiber of the $\Gmr$-equivariant map 
$
\Emr\mathscr{F}_+ \longrightarrow \Smr^0. 
$
\end{notn}

\begin{notn} For $x \in \Rmr^0_{\Gmr} (\Xmr) \cong [\Xmr, \Rmr]^\Gmr$, let 
\[ 
\begin{tikzcd}
\varphi^\Kmr(x) :  \Phi^\Kmr(\Xmr) \rar & \Phi^\Kmr(\Rmr)
\end{tikzcd}
\ \in  [ \Phi^\Kmr(\Xmr), \Phi^\Kmr(\Rmr) ]^{\Wmr(\Kmr)}
\]
denote the map induced by $x$ on the geometric fixed points with respect to $\Kmr$. 
\end{notn} 

Using the connecting map
$\Smr^0 \longrightarrow \widetilde{\Emr\mathscr{P}_\Kmr}$ one gets a lax monoidal natural transformation (see \cite[Appendix B]{Ninfty})
\begin{equation} \label{map:FtoG} 
\begin{tikzcd}
\upeta_{\Kmr} : (-)^{\Kmr} \rar & \Phi^\Kmr(-)
\end{tikzcd}
 \end{equation}
from the $\Kmr$-fixed point functor to the $\Kmr$-geometric fixed point functor. Consequently, for any family $\mathscr{F}$ of $\Gmr \times \Pi$ and $\Pi$-set $\Tmr$,  we have a  commutative diagram 
\begin{equation} \label{com:extended}
 \begin{tikzcd}
 \Dmr^{\mathscr{F}^\Kmr}_{\Tmr}((-)^\Kmr) \dar["\Dmr_{\sf n}^{\mathscr{F}^\Kmr}(\upeta_\Kmr)"'] \ar[rr, " \uplambda"]  && (\Dmr_{\Tmr}^{\mathscr{F}}(-))^\Kmr  \dar["\upeta_{\Kmr}"] \\
 \Dmr^{\mathscr{F}^\Kmr}_{\Tmr}(\Phi^\Kmr(-)) \ar[rr, " \widehat{\uplambda}"']  && \Phi^\Kmr(\Dmr_{\Tmr}^{\mathscr{F}}(-)), 
 \end{tikzcd}
\end{equation}
in $\Ho(\Sp_{\Wmr(\Kmr)})$, where the map $\widehat{\uplambda}$ is the map defined using  $\uplambda$ and the fact that  $\widetilde{\Emr\mathscr{P}_{\Kmr}} \sma \widetilde{\Emr\mathscr{P}_{\Kmr}}$ is equivalent to $ \widetilde{\Emr\mathscr{P}_{\Kmr}}$. 

When $\Rmr$ is an  $\Ocal$-ring in $\Sp_\Gmr$ then $\Phi^\Kmr(\Rmr)$ is an $\Ocal^{\Kmr}$-ring with structure maps 
 \[ 
 \begin{tikzcd}
 \uptheta_n^{\Phi^\Kmr(\Rmr)}: \Dmr^{\Ocal^\Kmr}_{\sf n}(\Phi^\Kmr(\Rmr)) \ar[r, " \widehat{\uplambda}"]  & \Phi^\Kmr(\Dmr_{\sf n}^{\Ocal}(\Rmr)) \ar[rr, "\varphi^\Kmr(\uptheta^\Rmr_n)"] && \Phi^\Kmr(\Rmr),
 \end{tikzcd}
 \]
 and  $\upeta_\Kmr$  of \eqref{map:FtoG} is a map of  $\Ocal^\Kmr$-rings. Thus, we get the following result. 
\begin{lem} \label{lem:power-geofix}  Suppose $\Rmr$, $\upkappa$, $\Fscr$ are as in \Cref{assump:RF} and $\Kmr$ be a subgroup of $\Gmr$. Then, 
for   $\Xmr \in \Sp_{\Gmr}$ and $x \in \Rmr^0_\Gmr(\Xmr)$
\begin{enumerate}
\item $\widehat{\uplambda}^{ *} \left(\varphi^\Kmr\left(\Pcal^{\mathscr{F}}\left(x\right)\right)\right) = \Pcal^{\mathscr{F}^\Kmr}\left( \varphi^\Kmr(x)\right) $
\item  $\Dmr_{\sf n}^{\mathscr{F}^\Kmr}(\upeta_\Kmr)_*\left(\Pcal^{\mathscr{F}^\Kmr}(x^\Kmr)\right) = \Pcal^{\mathscr{F}^\Kmr}(\upeta_{\Kmr*}(x^\Kmr)) = \Pcal^{\mathscr{F}^\Kmr}(\varphi^{\Kmr}(x))$. 
\end{enumerate}
\end{lem}
\begin{proof} It is easy to check the commutativity of the diagram
\[   \hspace{-30pt}
\begin{tikzcd}
\Rmr^{0}_\Gmr(\Xmr) \ar[dddd, bend right = 60, "\varphi^{\Kmr}( \ )"']\ar[dd, "(\ )^{\Kmr}"'] \ar[rrr, "\Pcal^{\Fscr}"]&&& \Rmr^{0}_{\Gmr}(\Dmr^{\Fscr}_{\upkappa^*{\sf n}}(\Xmr)) \ar[dd, "( \ )^{\Kmr}"] \ar[dddd, bend left = 70, "\varphi^{\Kmr}( \ )"]\\
 \\
 (\Rmr^\Kmr)^{0}_{\Wmr(\Kmr)}(\Xmr^\Kmr) \ar[dd, "\upeta_{\Kmr\ast}"'] \ar[rr, "\Pcal^{\Fscr^{\Kmr}}"]& & (\Rmr^\Kmr)^{0}_{\Wmr(\Kmr)}(\Dmr^{\Fscr^\Kmr}_{\upkappa^*{\sf n}}(\Xmr^\Kmr)) \ar[dd, "\Dmr_{\sf n}^{\mathscr{F}^\Kmr}(\upeta_\Kmr)_*"']
  \ar[r, "\uplambda^\ast"]& (\Rmr^\Kmr)^{0}_{\Wmr(\Kmr)}(\Dmr^{\Fscr}_{\upkappa^*{\sf n}}(\Xmr)^\Kmr) \ar[dd, "\upeta_{\Kmr*}"]\\
 \\
( \Phi^\Kmr\Rmr)^{0}_{\Wmr(\Kmr)}(\Phi^\Kmr(\Xmr)) \ar[rr, "\Pcal^{\Fscr^\Kmr}"'] && ( \Phi^\Kmr\Rmr)^{0}_{\Wmr(\Kmr)}(\Dmr^{\Fscr^\Kmr}_{\upkappa^*{\sf n}}(\Phi^{\Kmr}\Xmr)) \ar[r, "\widehat{\uplambda}^{\ast}"']& ( \Phi^\Kmr\Rmr)^{0}_{\Wmr(\Kmr)}(\Dmr^{\Fscr}_{\upkappa^*{\sf n}}(\Xmr)^\Kmr). 
\end{tikzcd}
\]
Then \emph{(1)} follows from the commutativity of the outer square and \emph{(2)} follows from the commutativity of the lower  left square. 
\end{proof}

\begin{notn}\label{notn:allFamily}
We call a $\Gmr$-operad $\Ocal$ an $\EE_\infty^\Gmr$-operad if its $n$-th space is equivalent to  $\mr{E}_{\Gmr}\Sigma_n$, the total space of the universal principal $\Gmr$-equivariant $\Sigma_n$-bundle. By definition  
 \[ \mr{E}_{\Gmr}\Sigma_n:= \Emr \All_n, \] 
 where $\All_n$ is the $\Gmr$-family consisting of all subgroups  of $\Gamma$ of $\Gmr \times \Sigma_n$ whose intersection with $\Sigma_n$ is trivial. 
\end{notn}

In this paper, we consider the special case when $\Rmr = \Hmr \ull{\FF}_p$, which is an $\EE_\infty^{\Gmr}$-ring in $\Sp_\Gmr$. For a subgroup $\Kmr \subset \Gmr$,  the $\Kmr$-fixed points of $\Hmr \ull{\FF}_p$ is $\Hmr \ull{\FF}_p$ at the Weyl group $\Wmr(\Kmr)$, however, the $\Kmr$-geometric fixed points need not be Eilenberg MacLane\footnote{When $\Gmr =\Cmr_2$ then $\Phi^{\Cmr_2}(\Hmr\ull{\FF}_2) \simeq \bigvee_{n \in \NN}\Sigma^n\Hmr\ull{\FF}_2 $ (see  \cite{HK}).  }. Nevertheless, there exists an $\EE_\infty^{\Wmr(\Kmr)}$-ring map 
\[ 
\begin{tikzcd}
\uppi: \Phi^\Kmr(\Hmr\ull{\FF}_p) \rar & \Hmr\ull{\FF}_p, 
\end{tikzcd}
\]
as the zeroth Postnikov tower is a lax monoidal functor. Consequently, the power operations (as in \Cref{defn:powerF}) commute with the 
 ``modified'' $\Kmr$-geometric fixed point functor 
\begin{equation} \label{eqn:modgeofix}
\begin{tikzcd} 
\widetilde{\varphi}^\Kmr :  \Hmr^\star_{\Gmr}( -; \ull{\FF}_2) \rar & \Hmr^\star_{\Wmr(\Kmr)}( \Phi^\Kmr( -) ; \ull{\FF}_2)
\end{tikzcd}
\end{equation}
which sends  a class $x$ to $\widetilde{\varphi}^\Kmr(x) :=  \uppi^*\varphi^\Kmr(x)$:
\begin{lem} \label{lem:powermodgeo}Suppose $\Rmr$, $\upkappa$, $\Fscr$ are as in \Cref{assump:RF} and $\Kmr$ be a subgroup of $\Gmr$. Then 
 \[ \widehat{\uplambda}^{ *} \left( \Pcal^{\mathscr{F}^\Kmr} (\widetilde{\varphi}^\Kmr(x)) \right)=\widetilde{\varphi}^\Kmr \left(\Pcal^{\mathscr{F}}(x)\right) \]
for any $x \in \Hmr^0_\Gmr(\Xmr; \ull{\FF}_p)$, where  $\Xmr \in \Sp_\Gmr$.  
\end{lem}
\section{Equivariant orientations and shifted power operations} \label{sec:shiftedpower}

In nonequivariant stable homotopy theory, shifted power operations are a feature of the $\HH_{\infty}^d$-ring structures introduced in \cite{Hinfty}. The main purpose of this section is twofold: first, to generalize this concept to the equivariant setting, and second, to extend \Cref{defn:powerF} to define power operations on classes in nonzero degrees. 
 As demonstrated in \cite[VII]{Hinfty}, an $\HH_{\infty}^d$-structure is equivalent to existence of a certain compatible family of orientations. To develop the equivariant analog of these results, we consider the following equivariant bundles:  
\begin{notn}\label{notn:equiBundle} Given  an $\Ncal_\infty$ $\Gmr$-operad $\Ocal$, a group homomorphism  $\upkappa: \Pi \to \Sigma_n$  and  a sub-family $\mathscr{F}$ of $(1 \times \upkappa)^*\mathscr{F}_n(\Ocal)$, define  the $\Gmr$-equivariant vector bundle:
\begin{equation} \label{bundle:TautVF}
\upgamma_{\Vmr}^{\mathscr{F}} := \begin{tikzcd}
(\Emr \mathscr{F}) \times_{\Pi} (\Vmr \otimes \upkappa^*\uptau_n) \dar \\
\Bmr \mathscr{F} := (\Emr \mathscr{F}) \times_{\Pi} {\bf 0} 
\end{tikzcd}
\end{equation}
where $\Vmr$ is a finite dimensional real $\Gmr$-representation  and $ \uptau_n$ is the permutation representation of $\Sigma_n$, i.e., the orthogonal $\Sigma_n$-representation generated by the set ${\sf n} = \{ 1, \dots, n \}$. 
\end{notn}
\begin{rmk} The $(\Fscr, \upkappa^*{\sf n})$-th extended power (as in \Cref{defn:extpower}) of the representation sphere $\Smr^{k\Vmr}$ is $\Gmr$-equivariantly homeomorphic to the Thom space of the bundle 
\[  \upgamma_{k\Vmr}^{\mathscr{F}}  \cong  \underbrace{\upgamma_{\Vmr}^{\mathscr{F}} \oplus \dots \oplus \upgamma_{\Vmr}^{\mathscr{F}}}_{\text{$k$-fold}},\] 
the $k$-fold direct sum of $\upgamma_{\Vmr}^{\mathscr{F}}$. In other words, 
\[ 
\Dmr^{\Fscr}_{\upkappa^*{\sf n}}(\Smr^{k \Vmr}) \cong \Th( \upgamma_{k\Vmr}^{\mathscr{F}} )
\]
for all $k \in \NN$. 
\end{rmk}
\begin{rmk} In this section, we examine the relationship between the shifted power operations across restriction and geometric fixed-point functors. Our comparison result,  \Cref{thm:Extendpower}, hinges on the fact that the restriction and the geometric fixed-point functors on the category of $\Gmr$-spectra are strictly monoidal. Consequently, our arguments do not compare these shifted power operations across categorical fixed-point functors, as they are not strictly monoidal functors.
\end{rmk}
\begin{assump} \label{assump:generic} Suppose $\Rmr$, $\upkappa$, $\Fscr$ are as in \Cref{assump:RF} such that $\upgamma_{\Vmr}^{\mathscr{F}}$ is $\Rmr$-orientable.  
\end{assump}

 If $ \upgamma_{k\Vmr}^{\mathscr{F}}$ is $\Rmr$-orientable in the sense of \cite[Definition 2.26]{BZ}, then an $\Rmr$-Thom class exists:
\begin{equation} \label{eqn:TCgammaF}
 {\bf u}_{k\Vmr} \in \Rmr^{nk \Vmr}({\sf Th}(\upgamma_{k\Vmr}^{\mathscr{F}})).
 \end{equation}
We utilize this class to extend \Cref{defn:powerF} in the following manner.
\begin{defn} \label{defn:extendpower} 
Under  \Cref{assump:generic},  define the {\bf $\mathscr{F}$-th  power operation} of  $\Rmr$  as the natural map  
\[ 
\begin{tikzcd}
\Pcal^{\mathscr{F}}_{ \Vmr} : \Rmr^{k \Vmr}(-) \rar & \Rmr^{n k \Vmr}( \Dmr^{\mathscr{F}}_{ \upkappa^*{\sf n}}(-) )
\end{tikzcd}
\]
which sends  $x: \Xmr \to \Sigma^{k \Vmr} \Rmr$ to the composite 
\begin{equation} \label{eq:powerext}
\begin{tikzcd}
 \Dmr^{\mathscr{F}}_{ \upkappa^*{\sf n}}(\Xmr) \ar[r, " \Dmr^{\mathscr{F}}_{ \upkappa^*{\sf n}}( x)"]  &  \Dmr^{\mathscr{F}}_{\upkappa^*{\sf n}}(\Sigma^{k \Vmr} \Rmr) \dar["\partialup_{\upkappa^*{\sf n} }"]  \\
& \Dmr^{\mathscr{F}}_{\upkappa^*\sf n}(\Smr^{k \Vmr} ) \sma \Dmr^{\mathscr{F}}_{\upkappa^*\sf n}(\Rmr ) \rar[cong]  &{\sf Th}(\upgamma_{k \Vmr}^{\mathscr{F}} ) \sma  \Dmr^{\mathscr{F}}_{\upkappa^*\sf n}(\Rmr ) \dar["{\bf u}_{k\Vmr} \sma \uptheta_{ \mathscr{F}, \upkappa}^{ \Rmr}"] \\
&& \Sigma^{nk \Vmr} \Rmr \sma \Rmr \rar["\upmu_\Rmr"'] & \Sigma^{nk \Vmr}\Rmr
 \end{tikzcd}
\end{equation}
in ${\sf Ho}(\Sp_{\Gmr}) $ for all $k \in \NN$. 
\end{defn}

\begin{thm} \label{thm:Extendpower} The $\mathscr{F}$-th power operation of \Cref{defn:extendpower} satisfies 
\begin{enumerate}
\item $\Pcal^{\upiota_\Kmr\mathscr{F}}_{\upiota_\Kmr\Vmr}(\upiota_{\Kmr}(x)) = \upiota_\Kmr(\Pcal^{\mathscr{F}}_{\Vmr}(x))$
\item $\Pcal^{\mathscr{F}^\Kmr}_{\Vmr^\Kmr}(\varphi^\Kmr(x)) = \widehat{\uplambda}^{*}(\varphi^\Kmr(\Pcal^{\mathscr{F}}_{\Vmr}(x)))  $
\end{enumerate}
for any subgroup $\Kmr \subset \Gmr$. 
\end{thm}
\begin{proof} Since  the restriction  of a $\R$-Thom class of $\upgamma_{\Vmr}^{\mathscr{F}}$ is an $\iota_\Kmr(\Rmr)$-Thom  class of $\upgamma_{\iota_\Kmr\Vmr }^{\iota_\Kmr \mathscr{F}}$
\[ {\bf u}_{\upiota_\Kmr\Vmr} = \upiota_{\Kmr*}({\bf u}_{ \Vmr}),\]
statement \emph{(1)} follows simply from applying the restriction functor. 

To prove statement \emph{(2)} we first observe that the geometric $\Kmr$-fixed points of an  $\R$-Thom class of $\upgamma_{\Vmr}^{\mathscr{F}}$ composed with 
\[ 
\begin{tikzcd} 
\widehat{\uplambda}:  \Sigma^{\infty}_{\Wmr(\Kmr)}\Dmr^{\mathscr{F}^\Kmr}_{\upkappa^*{\sf n}}( \Smr^{k \Vmr^\Kmr}) \dar[cong] \ar[rr] &&   \Phi^{\Kmr}\left(\Sigma^{\infty}_\Gmr \Dmr^{\mathscr{F}}_{\upkappa^*{\sf n}}( \Smr^{k \Vmr}) \right)  \dar[cong] \\
    \Sigma^{\infty}_{\Wmr(\Kmr)}\Th( \upgamma_{\Vmr^\Kmr}^{\mathscr{F}^\Kmr})  && \Phi^\Kmr(\Sigma^\infty_{\Gmr}\Th(\upgamma_{\Vmr}^{\mathscr{F}})) 
\end{tikzcd}
\] 
is an $\Phi^\Kmr(\Rmr)$-Thom  class of $\upgamma_{\Vmr^{\Kmr} }^{ \mathscr{F}^\Kmr}$. Then we have a homotopy commutative diagram 
 \[ 
\begin{tikzcd}
\Dmr^{\mathscr{F}^\Kmr}_{\upkappa^*{\sf n}}(\Phi^{\Kmr}(\Xmr)) \ar[DarkBrown, rrr, "\widehat{\uplambda}"] \dar[DarkBlue, "\Dmr^{\mathscr{F}^\Kmr}_{\upkappa^*{\sf n}}(\varphi^{\Kmr}(x))"'] &&& \Phi^{\Kmr}(\Dmr^{\mathscr{F}}_{\upkappa^*{\sf n}}(\Xmr)) \dar[DarkBrown,"\varphi^{\Kmr}(\Dmr^{\mathscr{F}}_{\upkappa^*{\sf n}}(x)) "] \\
\Dmr^{\mathscr{F}^\Kmr}_{\upkappa^*{\sf n}}( \Phi^{\Kmr}(\Sigma^{k \Vmr}\Rmr)) \dar[DarkBlue, " \partialup_{\upkappa^*{\sf n}}^{\Fscr^\Kmr}"'] \ar[rrr, "\widehat{\uplambda}"] &&& \Phi^{\Kmr}(\Dmr^{\mathscr{F}}_{\upkappa^*{\sf n}}(\Sigma^{k \Vmr}  \Rmr) ) \dar[DarkBrown, "\varphi^{\Kmr}(\partialup^{\Fscr}_{\upkappa^*{\sf n}})"]  \\
\Dmr^{\mathscr{F}^\Kmr}_{\upkappa^*{\sf n}}( \Smr^{k \Vmr^\Kmr}) \sma  \Dmr^{\mathscr{F}^\Kmr}_{\upkappa^*{\sf n}}( \Phi^{\Kmr}(\Rmr)) 
 \ar[rrr, "\widehat{\uplambda} \sma \widehat{\uplambda}"] \dar[DarkBlue, cong]
  &&& \Phi^{\Kmr}(\Dmr^{\mathscr{F}}_{\upkappa^*{\sf n}}(\Smr^{k \Vmr})) \sma   \Phi^{\Kmr}(\Dmr^{\mathscr{F}}_{\upkappa^*{\sf n}}( \Rmr) ) \dar[DarkBrown, cong] \\
\Th( \upgamma_{\Vmr^\Kmr}^{\mathscr{F}^\Kmr})  \sma  \Dmr^{\mathscr{F}^\Kmr}_{\upkappa^*{\sf n}}( \Phi^{\Kmr}(\Rmr)) 
  \dar[DarkBlue, "\widehat{\uplambda}^*\varphi^\Kmr({\bf u}_{ \Vmr}) \sma \uptheta_n^{\Phi^\Kmr(\Rmr)}"']  &&& \Phi^{\Kmr}(\Th( \upgamma_{\Vmr}^{\mathscr{F}}) ) \sma   \Phi^{\Kmr}(\Dmr^{\mathscr{F}}_{\upkappa^*{\sf n}}( \Rmr) ) \dar[DarkBrown, "\varphi^{\Kmr}({\bf u}_{ \Vmr} )\sma \varphi^\Kmr(\uptheta_n^{\Rmr})"]   \\
\Sigma^{k \Vmr^{\Kmr}}  \Phi^{\Kmr}(\Rmr) \sma \Phi^\Kmr(\Rmr) \ar[rrr, equal] \dar[DarkBlue, "\upmu_{\Phi^\Kmr(\Rmr)}"'] &&& \Sigma^{k \Vmr^\Kmr}\Phi^\Kmr(\Rmr) \sma \Phi^{\Kmr}(\Rmr) \dar[DarkBrown, "\upmu_{\Phi^\Kmr(\Rmr)}"] \\
\Sigma^{k \Vmr^{\Kmr}}  \Phi^{\Kmr}(\Rmr) \ar[  rrr, equal] &&& \Sigma^{k \Vmr^\Kmr}\Phi^\Kmr(\Rmr)  
\end{tikzcd}
\]
using the naturality of $\widehat{\uplambda}$, $\partialup_{\upkappa^*{\sf n}}^{\Fscr^\Kmr}$ and $\partialup_{\upkappa^*{\sf n}}^{\Fscr}$ (defined in \eqref{eqn:extendedoplax}), as well as the strong symmetric monoidal property of $\Phi^{\Kmr}$. Now observe that the composition of the blue arrows and the red arrows represent the left hand side  and the right hand side of \emph{(2)} respectively. Hence, the result. 
\end{proof}

 Consider a  $\Gmr$-space $\Xmr$ as a $\Gmr \times \Pi$-space with the trivial action of $\Pi$. Then the diagonal
 \[ 
 \begin{tikzcd}
 \Delta: \Xmr \rar & \Xmr^{\sma \Tmr}
 \end{tikzcd}
 \]
 is $(\Gmr \times \Pi)$-equivariant for any finite $\Pi$-set $\Tmr$, and it induces the following  map:
 \begin{equation} \label{eqn:delta}
 \begin{tikzcd}
 \updelta: \Bmr \mathscr{F} \times \Xmr \simeq \Emr \mathscr{F}_+ \sma_{\Pi} \Xmr \ar[rrr, "1_{\Emr\mathscr{F}_+} \sma_{\Pi}(\Delta) "] &&& \Emr \mathscr{F}_+ \sma_{\Pi} \Xmr^{\sma \Tmr}  \simeq \Dmr^{\mathscr{F}}_{\Tmr}(\Xmr). 
 \end{tikzcd}
 \end{equation}
 (specifically, under the conditions of \Cref{assump:generic}), we use the induced map $\updelta$ to obtain a class
  \[ \updelta^*\Pcal^\mathscr{F}_{k \Vmr}(x) \in \Rmr^{nk \Vmr}_{\Gmr}( \Bmr \mathscr{F} \times \Xmr)_+  \]  which is crucial to the construction of Steenrod operations. 
 
\subsection{Known examples of equivariant Steenrod operations}  \label{subsec:known}\ 

Steenrod operations have been constructed nonequivariantly for $\Hmr\FF_p$-cohomology for all primes \cite{SteenBook}, and for $\Hmr\ull{\FF}_2$-cohomology when $\Gmr = \Cmr_2$ \cite{HK, VRed}. 

We begin by discussing the case \(p=2\). In the following discussion, we simplify our notation by letting
\[ 
\begin{tikzcd}
\Pcal: \Hmr^{k\uprho_\Gmr}_{\Gmr}( - ; \ull{\FF}_2) \ar[rr] &&  \Hmr^{2k \uprho_\Gmr}_{\Gmr}( \Dmr^{\All_2}_{2}(-) ; \ull{\FF}_2) 
\end{tikzcd}
\]  
 denote the $\All_2$-th power operation for all $k \in \NN$. 
 
  When $\Gmr$ is trivial, the coefficient ring is the field  $\FF_2$, and we therefore  have a K\"unneth isomorphism. Consequently, for any $\Xmr \in \Top$
\begin{equation} \label{Kunnetiso}
  \Hmr^*( (\B\Sigma_2 \times  \Xmr)_+; \FF_2) \cong \Hmr^{*}(\Xmr) \llbracket{\bf e}_1 \rrbracket, 
  \end{equation}
where ${\bf e}_1 \in \mr{H}^1(\Bmr \Sigma_2; \FF_2)$ is the $\Hmr\FF_2$-Euler class of the tautological line bundle.  Under the identification \eqref{Kunnetiso}, we have the formula
\begin{equation} \label{eqn:classicalSq}
\updelta^* \Pcal(x) = \sum_{i = 0}^{k} \Sq^i(x) {\bf e}_1^{k - i}
\end{equation}
which defines the classical Steenrod operations. 

When $\Gmr = \Cmr_2$,  the coefficient ring 
$ \MM_2^{\Cmr_2} := \pi_{\star}^{\Cmr_2} (\Hmr\ull{\FF}_2)$
is not a field. Therefore,  a K\"unneth isomorphism  should not be expected to hold in general. However, Hu and Kriz \cite{HK} showed that
\begin{equation} \label{eqn:cohHK}
 \Hmr^{\star}_{\Cmr_2}( (\Bmr_{\Gmr}\Sigma_2)_+; \ull{\FF}_2 ) \cong \MM_2^{\Cmr_2}\llbracket{\bf y}, {\bf e}_{\uprho} \rrbracket / ( {\bf y}^2 = a {\bf y} + u {\bf e}_{\uprho}),   
 \end{equation}
where ${\bf e}_{\uprho}$ is the $\Hmr\ull{\FF}_2$-Euler class of a $\uprho$-dimensional $\Cmr_2$-equivariant vector bundle $\overline{\upgamma}$ as in \cite[pg 17]{BGL2} (also see \Cref{rmk:bundleex}), and $a$ and $u$ are specific elements in the coefficient ring. Importantly, $\Hmr^{\star}_{\Cmr_2}( \Bmr_{\Gmr}\Sigma_2; \ull{\FF}_2 )$ is free as over $\MM^{\Cmr_2}_2$, and there is a K\"unneth isomorphism.  Under this isomorphism, we have the formula:
\begin{equation} \label{eqn:CtwoSq}
\updelta^* \Pcal(x) = \sum_{i = 0}^{ k} \Sq^{2i}(x)  {\bf e}_{\uprho}^{k - i} +  \sum_{i = 0}^{k-1} \Sq^{2i+1}(x)  {\bf y}{\bf e}_{\uprho}^{k - i-1}
\end{equation}
for any $x$ in degree $k \uprho$. This formula defines the $\Cmr_2$-equivariant Steenrod operations \cite[$\mathsection$3]{BGL2}. 

\begin{rmk} \label{rmk:K\"unneth}
The above definition of $\Cmr_2$-equivariant Steenrod operations may not generalize to an arbitrary group $\Gmr$ because the K\"unneth map
\begin{equation} \label{eqn:K\"unneth}
\begin{tikzcd}
 \frak{K}:   \Rmr^{\star}_{\Gmr}(\Bmr \mathscr{F}_+) \otimes_{\pi_{\star}^{\Gmr}\Rmr}  \Rmr^\star_{\Gmr}(\Xmr_+)  \rar[] &  \Rmr^{\star}_\Gmr(\Bmr \mathscr{F}_+ \sma \Xmr_+) 
 \end{tikzcd}
 \end{equation}
 is not always an isomorphism\footnote{In the unpublished work \cite{NickGProj}, Nick Georgakopoulus showed that the $\RO(\Cmr_4)$-graded cohomology of $\Bmr_{\Cmr_4}\Sigma_2$ is not free over its coefficient ring. }. In \Cref{sec:Euler}, we will introduce the theory of Eulerian sequence to circumvent this failure of the K\"unneth map to be an isomorphism. 
\end{rmk}

In the classical case, i.e., when $\Gmr$ is the trivial group,  Epstein and Steenrod (see \cite[Chapter VII]{SteenBook})  use the inclusion of $\Pi = \Cmr_p$
\begin{equation} \label{eqn:CptoS}
\begin{tikzcd}
\upkappa: \Cmr_p \rar[hook] & \Sigma_p
\end{tikzcd}
\end{equation}
to identify odd primary Steenrod operations. In this case, they consider the power operation $\Pcal_{k(p-1)}^{\upkappa^*\All_p}$, which  will be denote by 
\[ 
\begin{tikzcd}
\Pcal: \Hmr^{k(p-1)}_{\Gmr}( - ; \ull{\FF}_p) \ar[rr] &&  \Hmr^{pk(p-1)}_{\Gmr}( \Dmr^{\upkappa^*\All_p}_{\upkappa^*{\sf p}}(-) ; \ull{\FF}_p) 
\end{tikzcd}
\]  
for all $k \in \NN$  to alleviate notation. Note that 
\[ \Hmr^{\ast}((\Bmr\Cmr_p)_+ ; \FF_p) \cong \Lambda_{\FF_p}({\bf y})\llbracket {\bf u} \rrbracket, \]
where $|{\bf y}| =1$ and $|{\bf u}| = 2$. Then they utilize the K\"unneth isomorphism to note that 
\begin{equation} \label{eqn:oddP} 
\nu(k(p-1))\,\updelta^* \Pcal(x) = \sum_{i = 0}^{ k} (-1)^{i}\Pmr^{i}(x)  {\bf u}^{(k - i)(p-1)} +  \sum_{i = 0}^{k-1} (-1)^i \beta\Pmr^{i}(x)  {\bf y}{\bf u}^{(k - i)(p-1)-1}  
\end{equation}
where $x \in \Hmr^{k(p-1)} (\Xmr_+; \FF_p)$, $\beta$ is the Bockstein homomorphism, and
\[\nu(q) = \left(\left(\frac{p-1}{2}\right)!\right)^{q}(-1)^{(p-1)(q^2+q)/4}.\]
The above equation can be used to define mod $p$ Steenrod operations.  
\begin{rmk} If $x$ is a cohomology class in a degree which is not a multiple of $p-1$ then one can use the suspension isomorphism 
\begin{equation} \label{eqn:suspiso}
\upsigma : \Hmr^{\ast}(\Xmr_+; \FF_p) \cong  \Hmr^{\ast+1}(\Sigma\Xmr_+; \FF_p) 
\end{equation}
 to define the mod $p$ Steenrod operations on $x$ 
\[
\beta^{\epsilon}\Pmr^i (x) = \upsigma^{-n} (\beta^{\epsilon}\Pmr^i (\upsigma^n(x))),
 \]
 where $n$ is a positive integer such that $n + |x|$ is a multiple of $p-1$. This is well-defined as Epstein and Steenrod have shown that the mod $p$ Steenrod operations are stable, i.e.,  the operations  $\Pmr^i$ and $\beta$ commutes with the suspension isomorphism. 
\end{rmk} 

\begin{rmk} 
It is possible to define odd primary Steenrod operations using the power operation 
\[ 
\begin{tikzcd}
\Pcal:= \Pcal_{k(p-1)}^{\All_p}: \Hmr^{k(p-1)}_{\Gmr}( - ; \ull{\FF}_2) \ar[rr] &&  \Hmr^{pk(p-1)}_{\Gmr}( \Dmr^{\All_p}_{{\sf p}}(-) ; \ull{\FF}_2) 
\end{tikzcd}
\]  
avoiding the group homomorphism $\upkappa$ of \eqref{eqn:CptoS}. To see this we first note that the map 
\[ 
\begin{tikzcd}
\Bmr\upkappa^*: \Hmr^{\ast}((\Bmr\Sigma_p)_+ ; \FF_p) \cong \Lambda_{\FF_p}({\bf v})\llbracket {\bf e} \rrbracket  \ar[rr] && \Hmr^{\ast}((\Bmr\Cmr_p)_+ ; \FF_p)
\end{tikzcd}
 \]
 sends ${\bf v} \mapsto {\bf y}{\bf u}^{p-2} $ and ${\bf e} \mapsto {\bf u}^{p-1}$,  and therefore, 
 \begin{equation} \label{eqn:oddPS} 
\nu(k(p-1))\,\updelta^* \Pcal(x) = \sum_{i = 0}^{ k} (-1)^{i}\Pmr^{i}(x)  {\bf e}^{(k - i)} +  \sum_{i = 0}^{k-1} (-1)^i \beta\Pmr^{i}(x)  {\bf v}{\bf e}^{k - i -1}  
\end{equation}
for any $x$ in degree $k(p-1)$. 
\end{rmk}

\section{ Eulerian sequences and stable cohomology operations} \label{sec:Euler}

In this section, we introduce the concept of Eulerian sequences. These sequences are designed to resolve the problem posed in \Cref{rmk:K\"unneth}---the definition of Steenrod operations in the absence of a K\"unneth isomorphism. We then prove the fundamental result of this paper, \Cref{thm:stable}, which establishes that a stable cohomology operation exists corresponding to every such sequence.

 To motivate this definition, we first observe that the known Steenrod operations (described in \Cref{subsec:known}) can be defined using the slant product (see \Cref{rmk:classicalsl} and \Cref{rmk:ctwoicalsl} below), thereby avoiding the need for a K\"unneth isomorphism
\begin{notn} \label{notn:slant}
For any $\Rmr \in \Sp_\Gmr$, define the {\bf $\RO(\Gmr)$-graded  slant product} as the pairing
\begin{equation} \label{eqn:slant}
\begin{tikzcd}
(\ )|_{(\ )}:   \Rmr^{\Wmr}_\Gmr(\Bmr\mathscr{F}_+ \sma \Xmr) \otimes \Rmr_{\Vmr}^\Gmr(\Bmr\mathscr{F}_+)  \rar & \Rmr^{\Vmr - \Wmr}_\Gmr(\Xmr),
\end{tikzcd}
\end{equation}
 where $x|_{{\sf b}}$ is the homotopy class of the composite 
\[ 
\begin{tikzcd}
x|_{{\sf b}}: \Smr^{\Vmr} \sma \Xmr \rar["{\sf b} \sma 1_{\Xmr}"] & \Rmr \sma \Bmr \mathscr{F}_+ \sma \Xmr  \rar["1_{\Rmr} \sma x"] & \Rmr \sma \Sigma^{\Wmr} \Rmr  \ar[rr, "\Sigma^{\Wmr}\upmu_\Rmr"] && \Sigma^{\Wmr} \Rmr 
\end{tikzcd}
\] 
for any $x \in \Rmr^{\Wmr}_\Gmr( \Bmr\mathscr{F}_+ \sma \Xmr)$ and ${\sf b} \in \Rmr_{\Vmr}^\Gmr(\Bmr\mathscr{F}_+) $.  
\end{notn}

\begin{rmk} \label{rmk:classicalsl} In the classical case, we note that the Steenrod squaring operations (as in \eqref{eqn:classicalSq}) can  be equivalently defined  as 
\begin{equation} \label{Sqslant}
 \Sq^{i}(x) := \updelta^* \Pcal(x)|_{{\bf b}_{k-i}} 
 \end{equation}
whenever $|x| = k$, where ${\bf b}_{k-i} \in \Hmr_{k-i}((\Bmr \Sigma_2)_+; \FF_2)$ is the class dual to ${\bf e}_1^{k-i} \in \Hmr^{k-i}((\Bmr \Sigma_2)_+; \FF_2) $.  Likewise, at an odd prime $p$, one may also define $\Pmr^{i}$ of \eqref{eqn:oddP} using the slant product:
\begin{align*}
    \begin{split}
        \Pmr^{i}(x) &:= (-1)^{i}\nu(k(p-1))\,\updelta^* \Pcal(x)|_{{\bf b}_{k-i}} \\
        \beta\Pmr^{i}(x) &:= (-1)^{i}\nu(k(p-1))\, \updelta^* \Pcal(x)|_{{\bf c}_{k-i}}
    \end{split}
\end{align*}
whenever $|x| = k(p-1)$, where ${\bf b}_{k-i}, {\bf c}_{k-i}   \in \Hmr_{\ast}((\Bmr\Sigma_p)_+; \FF_p)$  are elements which satisfy 
\[ \langle  {\bf u}^{(k-i)(p-1)}, {\bf b}_{k-i}\rangle = 1 = \langle  {\bf y}{\bf u}^{(k-i)(p-1)-1}, {\bf c}_{k-i}\rangle \]
where $\langle \ , \ \rangle$ is the Kronecker product \cite[III.9]{AdamsGeneralized}. 
\end{rmk}
\begin{rmk} \label{rmk:ctwoicalsl} Note that $\Hmr\ull{\FF}_2$-homology and $\Hmr\ull{\FF}_2$-cohomology of $\Bmr_{\Cmr_2}\Sigma_2$  are $\MM_2^{\Cmr_2}$-linear dual to each other, as they are free over the coefficient ring $\MM_2^{\Cmr_2}$. Let 
\[ {\bf b}_{k\uprho}, {\bf c}_{k\uprho + \sigma} \in \Hmr_{\star}^{\Cmr_2}((\Bmr_{\Cmr_2}\Sigma_2)_+; \ull{\FF}_2) \]
denote the elements $\MM^{\Cmr_2}_2$-linear dual to ${\bf e}_{\uprho}^{k}$ and ${\bf y}{\bf e}_{\uprho}^{k}$ (described in \eqref{eqn:cohHK}), respectively. Then the $\Cmr_2$-equivariant Steenrod operations of \eqref{eqn:CtwoSq} can be defined as  
\[ \Sq^{2i}(x) :=  \updelta^* \Pcal(x)|_{{\bf b}_{(k-i) \uprho}} \text{ and }  \Sq^{2i + 1}(x) := \updelta^* \Pcal(x)|_{{\bf c}_{(k-i)\uprho + \sigma}} \]
whenever $|x| = k \uprho$. 
\end{rmk}
\begin{notn} \label{notn:cap} For a ring spectrum $\Rmr \in \Sp_{\Gmr}$, the {\bf $\RO(\Gmr)$-graded cap product} 
\[ 
\begin{tikzcd}
-\frown- : \Rmr_{\Wmr}^\Gmr(\Xmr)  \times \Rmr^{\Vmr}_\Gmr(\Xmr)  \rar &  \Rmr_{ \Wmr- \Vmr}^\Gmr(\Xmr)
\end{tikzcd}
\]
is defined for any $\Gmr$-space $\Xmr$ by sending the pair $b \in \Rmr_{\Wmr}^\Gmr(\Xmr)$ and $e \in \Rmr^{\Vmr}_\Gmr(\Xmr)$ to the composite
\[ 
\begin{tikzcd}
\Smr^{\Wmr} \rar["b"] & \Xmr \sma \Rmr \ar[r, "\Delta \sma 1_\Rmr"] &   \Xmr \sma \Xmr \sma  \Rmr \ar[rr, "1_\Xmr \sma e \sma 1_{
\Rmr}"] && \Xmr \sma \Sigma^{\Vmr} \Rmr \sma \Rmr \rar["1_\Xmr \sma \upmu_\Rmr"] & \Sigma^{\Vmr}\Xmr \sma \Rmr,
\end{tikzcd}
\] 
where $\Delta$ is the diagonal map of $\Xmr$. 
\end{notn}

Note, in \eqref{Sqslant}, the definition of the $i$-th  Steenrod operation depends  on the degree of the class $x$. Thus one may ask why the Steenrod squaring operations are stable, i.e., 
\[ 
\upsigma(\updelta^* \Pcal(x)|_{{\bf b}_{k-i}}) = \upsigma(\Sq^i(x)) = \Sq^i(\upsigma(x)) = \updelta^* \Pcal(\upsigma(x))|_{{\bf b}_{k+1-i}}
\]
where $\upsigma$ is the suspension isomorphism of  \eqref{eqn:suspiso}.  We observe that this is a consequence of  the relation 
\[  {\bf b}_{k+1}  \frown {\bf e}_1 = {\bf b}_{k},   \]
where `$\frown$' denotes the cap product and ${\bf e}_1$ is the Euler class of the tautological line bundle over $\Bmr\Sigma_2 \simeq \RR\PP^{\infty}$. This leads us to the notion of Eulerian sequences. 

\subsection{The theory of Eulerian sequences} \

A $\Gmr$-equivariant Eulerian sequence is defined using an Euler class of certain $\Gmr$-equivariant vector bundles. To describe these bundles consider $\uptau_n$---the orthogonal permutation representation of $\Sigma_n$ (see \Cref{notn:equiBundle}), and let $\widetilde{\uptau}_n$ denote the orthogonal complement of  the $1$-dimensional trivial subrepresentation of $\uptau_n$ (spanned by the sum of all elements in ${\sf n}$). 
 
  For any finite orthogonal $\Gmr$-representation $\Vmr$, let $\widetilde{\upgamma}_{\Vmr}^{\mathscr{F}}$ denote the $(n-1) \Vmr$-dimensional  $\Gmr$-equivariant vector bundle
\begin{equation}  \label{redbundle}
\widetilde{\upgamma}_{\Vmr}^{\mathscr{F}} := \begin{tikzcd}
(\Emr \mathscr{F}) \times_{\Pi} (\Vmr \otimes \upkappa^*\widetilde{\uptau}_n) \dar \\
\Bmr \mathscr{F}. 
\end{tikzcd}
\end{equation}
Note that  $\widetilde{\upgamma}_{\Vmr}^{\mathscr{F}} \oplus \upepsilon_{\Vmr} \cong \upgamma_{\Vmr}^{\mathscr{F}}$, and consequently, 
$ \Sigma^{\Vmr} \Th(\widetilde{\upgamma}_{\Vmr}^{\mathscr{F}}) \simeq \Th(\upgamma_{\Vmr}^{\mathscr{F}}).  $

 \begin{rmk} \label{rmk:bundleex} Recall  the $\Gmr$-closed family $\All_n$ of $\Gmr \times\Sigma_n$ from \Cref{notn:allFamily}. 
  When $\Gmr$ is the trivial group, then $\widetilde{\upgamma}^{\All_2}_{\uprho}$ is the tautological line bundle over $\Bmr\Sigma_2$. When $\Gmr = \Cmr_2$,  $\widetilde{\upgamma}^{\All_2}_{\uprho}$ is the $\uprho$-dimensional bundle  $\overline{\upgamma}$ described in \cite{BGL2}, which is used in the construction of  the $\Cmr_2$-equivariant Steenrod operations. 
\end{rmk}
 \begin{notn} Let $\Rmr \in \Sp_{\Gmr}$ be a ring spectrum. 
  If $\widetilde{\upgamma}_{\uprho}^{\mathscr{F}}$ admits and is equipped with an $\Rmr$-orientation (in the sense of \cite{BZ}), then we let 
\[  \widetilde{\bf u}_{\Vmr} \in \Rmr^{(n-1)\Vmr}_\Gmr(\Th(\widetilde{\upgamma}_{\Vmr}^{\mathscr{F}}) )  \]
 denote the corresponding $\Rmr$-Thom class. Let 
 \[ 
 \begin{tikzcd}
 \upzeta: \Bmr\mathscr{F}_+ \rar & \Th(\widetilde{\upgamma}^{\mathscr{F}}_{\Vmr})
 \end{tikzcd}
 \] denote the zero section of the $\Gmr$-equivariant  vector bundle $\widetilde{\upgamma}^{\mathscr{F}}_{\Vmr}$. When 
 $\widetilde{\upgamma}^{\mathscr{F}}_{\Vmr}$ is $\Rmr$-orientable then its  {\bf $\Rmr$-Euler class} 
\begin{equation} \label{eqn:euler}
 \td{\bf e}_{\Vmr} := \upzeta^*( \widetilde{\bf u}_{\Vmr}) \in \Rmr^{ (n-1)\Vmr}_\Gmr(\Bmr \mathscr{F}_+)
 \end{equation}
is the pullback of its $\Rmr$-Thom class  along the zero section. 

 \end{notn}
 \begin{rmk} If $\widetilde{\upgamma}_{\uprho}^{\mathscr{F}}$ is $\Rmr$-orientable, then 
 $\upgamma^{\mathscr{F}}_{\Vmr} \cong \td{\upgamma}_{\Vmr}^{\mathscr{F}} \oplus \upepsilon_\Vmr $ is also $\Rmr$-orientable, and we set
  \[ {\bf u}_{\Vmr} := \upsigma_{\Vmr}(\widetilde{\bf u}_{\Vmr})\] 
  as the  $\Rmr$-Thom class of $\upgamma^{\mathscr{F}}_{\Vmr}$ in \eqref{eqn:TCgammaF}. 
 \end{rmk}
\begin{defn} \label{defn:Eulerian}  Suppose $\Rmr, \upkappa, \Ocal, \Fscr$ be as in \Cref{assump:generic}.  Then a {\bf $\Vmr$-stable $\Rmr$-Eulerian sequence}  is a sequence
\[ 
{\upchi} = ( {\sf x}_0, {\sf x}_1, \dots )
\]
 such that 
 \begin{itemize}
 \item ${\sf x}_{i +1} \frown \td{\bf e}_{\Vmr}  = {\sf x}_i$,
 \item $ {\sf x}_0 \frown \td{\bf e}_{\Vmr} = 0, $
 \end{itemize}
 where $\td{\bf e}_{\Vmr}$ is an $\Rmr$-Euler class of $\widetilde{\upgamma}_{\Vmr}^{\mathscr{F}} $. Since $\td{\bf e}_{\Vmr}$ has degree $(n-1)\Vmr$
 \[ \| \upchi \| := t(n-1) \Vmr - |{\sf x}_t|  \in \RO(\Gmr, \Vmr)  \]
is independent of $t$. We call $\| \upchi\|$ the {\bf degree of $\upchi$}, and $n$ the {\bf weight} of $\upchi$. 
 \end{defn}
 \begin{notn} For a $\Gmr$-representation $\Vmr$ and $\Rmr \in \Sp_\Gmr$, let 
\[ 
\begin{tikzcd}
\upsigma_{\Vmr}: \Rmr^{\star}_\Gmr(-) \rar["\cong"] & \Rmr^{\star + \Vmr}_\Gmr(\Sigma^\Vmr(-))
\end{tikzcd}
\]
denote the $\Vmr$-th suspension isomorphism. We will simply use $\upsigma$ when $\Vmr$ is the $1$-dimensional trivial representation. 
\end{notn}
 \begin{defn} \label{defn:EulerSteen}
 Given a $\Vmr$-stable $\Rmr$-Eulerian sequence $\chi:= ({\sf x}_0, {\sf x}_1, \dots )$ of weight $n$,   define the $\RO(\Gmr, \Vmr)$-graded {\bf $\upchi$-th  Steenrod operation} 
 \[ 
 \begin{tikzcd}
 \Sfrak^{\upchi}: \Rmr_{\Gmr}^{\star}(-)  \rar & \Rmr_{\Gmr}^{\star + \| \upchi \|}(-)
 \end{tikzcd}
 \]
  as follows:
   \begin{itemize}
  \item   If $x \in  \Rmr^{t\Vmr}_\Gmr(\Xmr)$ where $\Xmr$ is a $\Gmr$-space , $t \in \NN$, then let 
  \begin{equation} \label{eqn:Sqchi}
   \Sfrak^{\upchi}(x) := \updelta^*\Pcal^{\mathscr{F}}_{t\Vmr}(x)|_{{\sf x}_t}.
   \end{equation}
   \item If $x \in  \Rmr^{\Wmr}_\Gmr(\Xmr)$ for some $\Wmr \underset{\text{finite}}{\subset} \Ucal_{\Gmr, \Vmr}$, then choose $\Wmr'$ such that $\Wmr \oplus \Wmr' = t \Vmr$ for some $t \in \NN$ and define 
   \begin{equation} \label{eqn:Sqchi2}
   \Sfrak^{\upchi}(x) := \upsigma_{\Wmr'}^{-1} \Sfrak^{\upchi}(\upsigma_{\Wmr'}( x)).
   \end{equation}
   \end{itemize}
\end{defn}
Our next result, \Cref{thm:stable}, demonstrates that the operation $\Sfrak^{\upchi}$  commutes with the $\Vmr$-th suspension  $\upsigma_{\Vmr}$. This simultaneously implies two key points: 
\begin{itemize}
\item[-]  the equation \eqref{eqn:Sqchi2} is well-defined, i.e., independent of the choice of $\Wmr'$, 
\item[-]  the operation $\Sfrak^{\upchi}$ is stable in the universe $\Ucal_{\Gmr, \Vmr}$, and therefore extends to classes of degrees contained in $\RO(\Gmr, \Vmr)$. 
\end{itemize}

\begin{thm} \label{thm:stable} Let $\Xmr$ be a $\Gmr$-space and let $\Rmr, \upkappa, \Ocal, \Fscr$ be as in \Cref{assump:generic}. Given a  $\Vmr$-stable $\Rmr$-Eulerian sequence \[ \upchi = ({\sf x}_0, \sfx_1, \dots),\]
   the  $\upchi$-th  Steenrod operation defined in \eqref{eqn:Sqchi} satisfies the naturality condition with respect to suspension:
   \[  \Sfrak^{\upchi}(\upsigma_{\Vmr}(x)) =   \upsigma_{\Vmr} ( \Sfrak^{\upchi}(x)) \]
   for any $x \in \Rmr^{t \Vmr}_\Gmr(\Xmr)$ for any $t \in \NN$. 
   \end{thm}
 \begin{proof}

 Consider the diagram \eqref{master} in which  the map
 \begin{enumerate}[ (1) \ \ ]
 \item is  the composition 
 \[  
 \begin{tikzcd}
 \Smr^{|{\sf x}_{t+1}|} \sma \Sigma^{\Vmr} \Xmr \ar[rrr, "\upsigma_{(n-1)\Vmr}({\sf x}_{t} \sma 1_{\Sigma\Vmr})"] &&& \Sigma^{(n-1) \Vmr} \Bmr\mathscr{F}_+ \sma \Rmr \sma  \Sigma \Xmr \dar[cong]  \\ &&& \Sigma^{n \Vmr} \Bmr\mathscr{F}_+ \sma \Xmr \sma \Rmr, 
 \end{tikzcd}
 \]
 \item is  the composition
  \[ 
  \begin{tikzcd}
 \Smr^{|{\sf x}_{t+1}|} \ar[rr,"  {\sf x}_{t +1} \sma 1_{\Sigma^{\Vmr}\Xmr}"] && \Bmr \mathscr{F}_+ \sma \Rmr \sma \Sigma^{\Vmr}\Xmr  \rar[cong] & \Bmr \mathscr{F}_+ \sma \Sigma^{\Vmr}\Xmr \sma  \Rmr,
  \end{tikzcd}
  \] 
 \item is  induced by the diagonal map $\Delta: \Bmr\mathscr{F} \longrightarrow \Bmr\mathscr{F} \times \Bmr\mathscr{F}$, 
 \item is   $1_{\Bmr\mathscr{F}_+} \sma \td{\bf e}_{\Vmr} \sma 1_{\Sigma^{\Vmr}\Xmr \sma \Rmr }$, 
 \item is  induced by the multiplication of $\Rmr$, 
 \item  is the natural map  $\updelta$ of \eqref{eqn:delta}, 
 \item  is  the composition 
 \[ 
   \ \ \ \ \ \ \begin{tikzcd}
& \Bmr \mathcal{F}_+ \sma \Bmr \mathcal{F}_+ \sma \Sigma^{\Vmr}\Xmr \sma \Rmr \dar[cong] \\
& (\Bmr \mathcal{F}_+  \sma \Smr^{\Vmr})\sma (\Bmr \mathcal{F}_+ \sma \Xmr) \sma \Rmr \ar[rrr, "\updelta \sma \updelta \sma 1_{\Rmr}"] &&& \Dmr^{\mathscr{F}}_{\upkappa^*{\sf n}}(\Smr^{\Vmr}) \sma \Dmr^{\mathscr{F}}_{\upkappa^*{\sf n}}(\Xmr) \sma \Rmr, 
 \end{tikzcd}
 \]
 \item is  the composition 
 \[ 
 \hspace{30pt}  \begin{tikzcd}
 \Bmr\mathscr{F}_+ \sma \Sigma^{(n-1)\Vmr} \Rmr \sma \Sigma^{\Vmr} \Xmr \sma \Rmr \dar[cong] \\
 \Sigma^{n\Vmr} \Rmr \sma  (\Bmr\mathscr{F}_+ \sma  \Xmr )\sma \Rmr \ar[rrr, "\upsigma_{n \Vmr}(1_{\Rmr} \sma \updelta \sma 1_{\Rmr})"] &&& \Sigma^{n\Vmr} \Rmr \sma  \Dmr^{\mathscr{F}}_{\upkappa^*{\sf n}} (  \Xmr) \sma \Rmr,
 \end{tikzcd}
 \]
 \item is $\upsigma_{n \Vmr} (\updelta \sma 1_{\Rmr})$, 
 \item is  $\partialup_{\upkappa^*{\sf n}}^{\mathscr{F}}  \sma 1_\Rmr$,
 \item is  $ \sigma_{\Vmr}(\widetilde{\bf u}_{\Vmr}) \sma 1_{\Dmr^{\mathscr{F}}_{\upkappa^*{\sf n}}(\Xmr) } \sma 1_{\Rmr}$,
 \item is   induced by the multiplication of $\Rmr$, 
 \item is  $\Dmr^{\mathscr{F}}_{\upkappa^*{\sf n}}(\upsigma_{\Vmr}(x)) \sma 1_{\Rmr}$, 
 \item  is  $1_{\Dmr^{\mathscr{F}}_{\upkappa^*{\sf n}}(\Smr^{\Vmr})} \sma  \Dmr^{\mathscr{F}}_{\upkappa^*{\sf n}}(x) \sma 1_{\Rmr}$,
 \item is  $\upsigma_{n\Vmr}(1_{\Rmr} \sma \Dmr^{\mathscr{F}}_{\upkappa^*{\sf n}}(x) \sma 1_{\Rmr} ) $,
 \item is $\upsigma_{n\Vmr}(\Dmr^{\mathscr{F}}_{\upkappa^*{\sf n}}(x) \sma 1_{\Rmr} ) $,
 \item is   $\partialup^{\mathscr{F}}_{\upkappa^*{\sf n}} \sma 1_{\Rmr}$, 
 \item is  $ \sigma_{\Vmr}(\widetilde{\bf u}_{\Vmr}) \sma 1_{\Dmr^{\mathscr{F}}_{\upkappa^*{\sf n}}(\Sigma^{t \Vmr} \Rmr) } \sma 1_{\Rmr}$,
 \item is   induced by the multiplication of $\Rmr$, 
 \item  is   $\partialup^{\mathscr{F}}_{\upkappa^*{\sf n}} \sma 1_{\Rmr}$,
 \item  is $1_{\Dmr^{\mathscr{F}}_{\upkappa^*{\sf n}} (\Smr^{\Vmr})} \sma  \partialup^{\mathscr{F}}_{\upkappa^*{\sf n}} \sma  1_{\Rmr} $, 
 \item  is $\upsigma_{n \Vmr} ( 1_\Rmr \sma \partial^{\mathscr{F}}_{\upkappa^*{\sf n}} \sma 1_{\Rmr})$, 
 \item  is $\upsigma_{n \Vmr} (  \partial^{\mathscr{F}}_{\upkappa^*{\sf n}} \sma 1_{\Rmr})$,
 \item  is $ \partialup^{\mathscr{F}}_{\upkappa^*{\sf n}} \sma 1_{\Dmr^{\mathscr{F}}_{\upkappa^*{\sf n}}} \sma 1_{\Rmr} $, 
 \item  is  $ \sigma_{\Vmr}(\widetilde{\bf u}_{\Vmr}) \sma 1_{\Dmr^{\mathscr{F}}_{\upkappa^*{\sf n}}(\Smr^{t \Vmr}) } \sma 1_{\Dmr^{\mathscr{F}}_{\upkappa^*{\sf n}}( \Rmr) }  \sma 1_{\Rmr}$,
 \item is   induced by the multiplication of $\Rmr$, 
 \item is $ \upsigma_{(t +1) \Vmr}(\widetilde{\bf u}_{(t+1)\Vmr}) \sma \uptheta_{ \mathscr{F}, \upkappa}^{ \Rmr} \sma 1_{\Rmr},$
 \item is $ \upsigma_{\Vmr}(\widetilde{\bf u}_{\Vmr}) \sma  \upsigma_{t \Vmr}(\widetilde{\bf u}_{t \Vmr}) \sma \uptheta_{ \mathscr{F}, \upkappa}^{ \Rmr} \sma 1_{\Rmr},$
 \item is $\upsigma_{n\Vmr} (1_{\Rmr} \sma \upsigma_{t\Vmr}(\widetilde{\bf u}_{t \Vmr}) \sma \uptheta_{ \mathscr{F}, \upkappa}^{ \Rmr} \sma 1_{\Rmr}  )$, 
 \item is $ \upsigma_{n \Vmr}(\upsigma_{t\Vmr}(\widetilde{\bf u}_{t \Vmr}) \sma \uptheta_{ \mathscr{F}, \upkappa}^{ \Rmr} \sma 1_{\Rmr})$, 
 \item is   induced by the multiplication of $\Rmr$, 
 \item is   induced by the multiplication of $\Rmr$. 
 \end{enumerate}
We will now argue that \eqref{master} commutes up to homotopy by showing that squares \textbf{(S1)} through \textbf{(S16)} commutes up to homotopy. 

The square \textbf{(S1)} commutes because ${\sf x}_{t+1} \frown \td{\bf e}_{\Vmr} = {\sf x}_t$. The square  \textbf{(S2)} commutes as the definition of the map $\partialup^{\mathscr{F}}_{\upkappa^*{\sf n}}$ involves the diagonal map of $\Bmr \mathscr{F}$. The square \textbf{(S3)} commutes because the map \[ \updelta: \Bmr\mathscr{F}_+ \sma \Smr^{\Vmr} \longrightarrow \Dmr^{\mathscr{F}}_{\upkappa^*{\sf n}}(\Smr^{\Vmr}) 
\simeq \Sigma^{\Vmr} \Th(\widetilde{\upgamma}_{\Vmr})\]  is equivalent to
$\upsigma_{\Vmr}(\upzeta)$, where $\upzeta$ is the zero section of
$\widetilde{\upgamma}_{\Vmr}$, and the fact that $\td{\bf e}_{\Vmr} =
\upzeta^{*}(\widetilde{\bf u}_{\Vmr})$.

\
 \begin{equation} \label{master}
  \rotatebox{90}{
    \begin{tikzcd}[ampersand replacement=\&, nodes={font=\scriptsize}]
 \Smr^{|{\sf x}_{t+1}|} \sma  \Sigma^{\Vmr} \Xmr \ar[ddd, color = DarkBlue,   "(2)"'] \ar[rrrddd, color = DarkBrown, bend left, "(1)"] \ar[rrrdd, phantom, "\text{\large \bf (S1)}"] \\\\
\& \& \& \text{$ $} \\
  \Bmr\mathscr{F}_+ \sma \Sigma^{\Vmr}\Xmr  \sma  \Rmr \rar["(3)"] \ar[dd,color = DarkBlue, "(6)"']  \ar[rdd, phantom, "\text{ \bf (S2)}"] \&
    \Bmr\mathscr{F}_+ \sma \Bmr\mathscr{F}_+ \sma \Sigma^{\Vmr}\Xmr  \sma   \Rmr \rar["(4)"] \ar[dd, "(7)"']  \ar[rdd, phantom, "\text{\bf  (S3)}"] \&
       \Bmr\mathscr{F}_+ \sma \Sigma^{(n-1)\Vmr} \Rmr \sma \Sigma^{\Vmr}\Xmr  \sma   \Rmr \rar["(5)"] \ar[dd, "(8)"']  \ar[rdd, phantom, "\text{\bf (S4)}"]  \&
        \Sigma^{n\Vmr} \Bmr\mathscr{F}_+  \sma \Xmr  \sma    \Rmr \ar[dd,color = DarkBrown, "(9)"]   \\ \\
\Dmr^{\mathscr{F}}_{\upkappa^*{\sf n}}(\Sigma^{\Vmr}\Xmr) \sma \Rmr \ar[dd, color = DarkBlue, "(13)"'] \rar["(10)"'] \ar[rdd, phantom, "\text{\bf (S5)}"]  \&
 \Dmr^{\mathscr{F}}_{\upkappa^*{\sf n}}( \Smr^{\Vmr}) \sma \Dmr^{\mathscr{F}}_{\upkappa^*{\sf n}}( \Xmr) \sma \Rmr \rar["(11)"'] \ar[dd, "(14)"'] \ar[rdd, phantom, "\text{\bf (S6)}"]  \&
  \Sigma^{n\Vmr} \Rmr \sma \Dmr^{\mathscr{F}}_{\upkappa^*{\sf n}}( \Xmr) \sma \Rmr \rar["(12)"'] \ar[dd, "(15)"']  \ar[rdd, phantom, "\text{\bf (S7)}"] \&
   \Sigma^{n \Vmr} \Dmr^{\mathscr{F}}_{\upkappa^*{\sf n}}( \Xmr) \sma \Rmr   \ar[dd, color = DarkBrown, "(16)"] \\ \\
\Dmr^{\mathscr{F}}_{\upkappa^*{\sf n}}(\Sigma^{(t+1)\Vmr}\Rmr) \sma \Rmr  \ar[dd, color = DarkBlue, "(20)"'] \rar["(17)"'] \ar[rdd, phantom, "\text{\bf (S8)}"]  \&
 \Dmr^{\mathscr{F}}_{\upkappa^*{\sf n}}( \Smr^{\Vmr}) \sma  \Dmr^{\mathscr{F}}_{\upkappa^*{\sf n}}(\Sigma^{t\Vmr}\Rmr) \sma \Rmr \rar["(18)"'] \ar[dd, "(21)"'] \ar[rdd, phantom, "\text{\bf (S9)}"]  \& 
 \Sigma^{n \Vmr} \Rmr \sma \Dmr^{\mathscr{F}}_{\upkappa^*{\sf n}}(\Sigma^{t\Vmr}\Rmr) \sma \Rmr \rar["(19)"'] \ar[dd, "(22)"'] \ar[rdd, phantom, "\text{\bf (S10)}"]  \&
  \Sigma^{n \Vmr} \Dmr^{\mathscr{F}}_{\upkappa^*{\sf n}}(\Sigma^{t\Vmr}\Rmr) \sma \Rmr  \ar[dd, color = DarkBrown,"(23)"]      \\ \\
\Dmr^{\mathscr{F}}_{\upkappa^*{\sf n}}(\Smr^{(t+1)\Vmr}) \sma  \Dmr^{\mathscr{F}}_{\upkappa^*{\sf n}}(\Rmr) \sma \Rmr \ar[dd, color = DarkBlue, "(27)"']  \rar["(24)"']  \ar[rdd, phantom, "\text{\bf (S11)}"]\&
 \Dmr^{\mathscr{F}}_{\upkappa^*{\sf n}}(\Smr^{\Vmr}) \sma  \Dmr^{\mathscr{F}}_{\upkappa^*{\sf n}}(\Smr^{t\Vmr})  \sma  \Dmr^{\mathscr{F}}_{\upkappa^*{\sf n}}(\Rmr) \sma \Rmr \rar["(25)"'] \ar[dd, "(28)"']  \ar[rdd, phantom, "\text{\bf (S12)}"] \&
  \Sigma^{n\Vmr}\Rmr \sma  \Dmr^{\mathscr{F}}_{\upkappa^*{\sf n}}(\Smr^{t\Vmr})  \sma  \Dmr^{\mathscr{F}}_{\upkappa^*{\sf n}}(\Rmr) \sma \Rmr  \rar["(26)"'] \ar[dd, "(29)"']  \ar[rdd, phantom, "\text{\bf (S13)}"]\&
    \Sigma^{n\Vmr} \Dmr^{\mathscr{F}}_{\upkappa^*{\sf n}}(\Smr^{t\Vmr})  \sma  \Dmr^{\mathscr{F}}_{\upkappa^*{\sf n}}(\Rmr) \sma \Rmr  \ar[dd,color = DarkBrown, "(30)"]  \\ \\
\Sigma^{ n (t +1) \Vmr} \Rmr \sma \Rmr \sma \Rmr \ar[dd, color = DarkBlue, "(34)"']   \ar[rdd, phantom, "\text{\bf (S14)}"] \&
 \Sigma^{n \Vmr} \Rmr \sma \Sigma^{nt\Vmr} \Rmr \sma \Rmr \sma \Rmr \rar[equal] \lar["(31)"] \ar[dd, "(34)"'] \ar[rdd, phantom, "\text{\bf (S15)}"]  \&
  \Sigma^{n \Vmr} \Rmr \sma \Sigma^{nt\Vmr} \Rmr \sma \Rmr \sma \Rmr  \rar["(32)"'] \ar[dd, "(35)"']  \ar[rdd, phantom, "\text{\bf (S16)}"]\&
   \Sigma^{n \Vmr} ( \Sigma^{nt\Vmr} \Rmr \sma \Rmr \sma \Rmr ) \ar[dd, color = DarkBrown, "(36)"]  \\ \\
\Sigma^{ n (t +1) \Vmr}\Rmr \rar[equal] \& \Sigma^{ n (t +1) \Vmr} \Rmr \rar[equal]  \& \Sigma^{ n (t +1) \Vmr} \Rmr \rar [equal] \& \Sigma^{ n (t +1) \Vmr} \Rmr 
\end{tikzcd}  
}
\end{equation}

The squares \textbf{(S5)}, \textbf{(S8)}, \textbf{(9)}, \textbf{(10)} commutes from the naturality of  $\partialup^{\mathscr{F}}_{\upkappa^*{\sf n}}$. The  square \textbf{(S11)} commutes because  $\td{\bf u}_{(t+1)\Vmr} =\upomega^*( \widetilde{\bf u}_{\Vmr} \times \widetilde{\bf u}_{t \Vmr} )$ (see \eqref{eqn:cross}) as
\begin{equation} \label{eqn:Thomclassadd}
\begin{tikzcd}
\Th(\widetilde{\upgamma}_{(t+1)\Vmr}^\Fscr) \dar["\upomega"'] \ar[rrr, "\widetilde{\bf u}_{(t+1) \Vmr}"] &&& \Sigma^{(t+1)(n-1)\Vmr}
 \Rmr    \dar[equal]  \\
 \Th(\widetilde{\upgamma}_{\Vmr}^\Fscr) \sma \Th(\widetilde{\upgamma}_{t\Vmr}^\Fscr) \ar[rr, " \widetilde{\bf u}_{\Vmr} \sma \widetilde{\bf u}_{t \Vmr} "] \ar[rrr, bend right, " \widetilde{\bf u}_{\Vmr} \times \widetilde{\bf u}_{t \Vmr} "']&& \Sigma^{(n-1)\Vmr}\Rmr \sma \Sigma^{t(n-1)\Vmr} \Rmr \rar["\upmu_{\Rmr}"]  & \Sigma^{(t+1)(n-1)\Vmr}\Rmr, 
 \end{tikzcd}
 \end{equation}
 where $\upomega$ is the natural map from the Thom space of the Whitney sum to Thom space of the product of two vector bundles. The remaining squares, i.e. \textbf{(S4)},  \textbf{(S4)},  \textbf{(S6)},  \textbf{(S7)},  \textbf{(S12)}, \textbf{(S13)}, \textbf{(S14)}, \textbf{(S15)}, \textbf{(S16)} commutes trivially.

Therefore,  the class represented by the composition of blue arrows in \eqref{master}
\[ [(34) \circ (27) \circ (20) \circ (13) \circ (6) \circ (2)] = \Sfrak^{\upchi}(\upsigma_{\Vmr}(x))  \in \Rmr^{t \Vmr + \| \upchi \|}_\Gmr(\Xmr) \]
must equal the class represented by  the composition of red arrows 
\[ [(36) \circ (30) \circ (23) \circ (16) \circ (9) \circ  (1) ]=\upsigma_{\Vmr}( \Sfrak^{\upchi}(x))  \in \Rmr^{t \Vmr + \| \upchi\|}_\Gmr(\Xmr),   \]
and hence,  the result. 
 \end{proof}  
\begin{rmk} \label{rmk:shift}
Observe that if $\upchi$ is a $\Vmr$-stable $\Rmr$-Eulerian sequence then its {\bf $k$-th shift}
\[ 
\upchi[k]= (\overbrace{0, \dots, 0}^{k\text{-fold}}, \sfx_0, \sfx_1, \sfx_2, \dots )
\]
is also a $\Vmr$-stable $\Rmr$-Eulerian sequence. Note that 
\[ 
\| \upchi[k] \| = \| \upchi \| + k (n-1) \Vmr  
\]
if $\upchi$ has weight $n$. 
\end{rmk}
\begin{ex}\label{ex:classicalEulerianSteen} In the classical case $\Hmr_*((\Bmr\Sigma_2)_+; \FF_2) = \FF_2 \{ {\bf b}_0, {\bf b}_1, {\bf b}_2, \dots  \}$, where 
${\bf b}_i$ is the element dual to ${\bf e}_1^i$ as in \Cref{rmk:classicalsl}. Then 
\[ 
\upbeta= ({\bf b}_0, {\bf b}_1, \dots )
\]
and its shifts accounts for all $\Hmr\FF_2$-stable Eulerian sequences of weight $2$. It follows form \Cref{rmk:classicalsl} that 
$
\Sfrak^{\upbeta[k]} = \Sq^{k}
$
for all $k \in \NN$. 
\end{ex}
\begin{ex} When $\Gmr = \Cmr_2$, let 
\[{\bf b}_{i\uprho}, {\bf c}_{i\uprho+\sigma}\in \Hmr_\star^{\Cmr_2}((\Bmr_{\Cmr_2}\Sigma_{2})_+;\ull{\FF}_2)\]
be the generators discussed in \Cref{rmk:ctwoicalsl}. Then
\begin{align*}
    \begin{split}
        \upbeta &= ({\bf b}_0, {\bf b}_{\uprho},  {\bf b}_{2\uprho},  \dots)\\
    \upzeta &= (0, {\bf c}_{\sigma}, {\bf c}_{\uprho + \sigma}, \dots) 
    \end{split}
\end{align*}
and their shifts are $\uprho$-stable Eulerian sequence of weight $2$ such that
\[\Sfrak^{\upbeta[k]} = \Sq^{2k}\quad\text{and }\quad \Sfrak^{\upzeta[k]} = \Sq^{2k+1}\]
for all $k\in \NN$.
\end{ex}
\begin{rmk} \label{rmk:equivEuler} Given an $\Vmr$-stable $\Rmr$-Eulerian sequence $\upchi = (\sfx_0, \sfx_1, \sfx_2, \dots )$ of weight $n$, we get a $k\Vmr$-stable  $\Rmr$-Eulerian sequence
\[ \mathfrak{t}_k(\upchi)= ( \sfx_0, \sfx_k, \sfx_{2k}, \dots ) \]
of weight $n$. This is because 
\begin{itemize}
\item $\td{\bf e}_{k\Vmr} = \underbrace{\td{\bf e}_{\Vmr} \smile \dots \smile \td{\bf e}_{\Vmr}}_{k}$ is an $\Rmr$-Euler class of $\td{\upgamma}^{\mathscr{F}}_{k\Vmr} = (\widetilde{\upgamma}^{\mathscr{F}}_{\Vmr})^{\oplus k}$, and, 
\item ${\sf x}_{(i+1) k}  \frown \td{\bf e}_{k\Vmr}  =  {\sf x}_{ik} $.
\end{itemize}
It follows immediately from \Cref{thm:stable} that $\Sfrak^{\upchi} = \Sfrak^{\mathfrak{t}_k\upchi} $.
\end{rmk}

\subsection{Restrictions and geometric fixed points of Eulerian sequences} \

Suppose $\Kmr \subset \Gmr$. The cap product (as in \Cref{notn:cap}) is well-behaved with restrictions and fixed points: 
\begin{itemize}
\item  $\upiota_{\Kmr\ast} (x \frown e) = \upiota_{\Kmr\ast} (x) \frown  \upiota_{\Kmr\ast}(e)  $
\item $\varphi^{\Kmr} (x \frown e) = \varphi^{\Kmr} (x) \frown \varphi^{\Kmr} (e) $. 
\end{itemize}
\begin{defn}\label{defn:restrEulerianSeq} Suppose $\upchi =(\sfx_0, \sfx_1, \sfx_2, \dots )$, where $\sfx_i \in \Rmr_{\star}^\Gmr( \Bmr \Fscr_+)$, is a $\Vmr$-stable $\Rmr$-Eulerian sequence.  Then  $\upiota_{\Kmr*} \sfx_i \in \Rmr_{\star}^\Gmr( \Bmr \upiota_\Kmr\Fscr_+)$, and
 \[ 
 \upiota_\Kmr (\upchi) := ( \upiota_{\Kmr*} \sfx_0, \upiota_{\Kmr*} \sfx_1,  \upiota_{\Kmr*} \sfx_2, \dots  )
 \]
is a $\upiota_{\Kmr}\Vmr$-stable $\upiota_\Kmr\Rmr$-Eulerian sequence, which we will call the {\bf restriction of $\upchi$ to the subgroup $\Kmr$}. 
\end{defn}
Defining the geometric fixed points of an Eulerian sequence is somewhat subtle as the natural map 
\[ 
\begin{tikzcd}
\wh{\uplambda}:  \Bmr(\Fscr^{\Kmr})_+ \rar & (\Bmr\Fscr)^\Kmr_+ 
\end{tikzcd}
\] 
is not generally an equivalence. However, $ \Bmr(\Fscr^{\Kmr})$  is a path component of $\Bmr\Fscr^\Kmr$ \cite[Theorem 10]{LM} (also see \cite[Theorem 2.13]{BZ}). Therefore,  there exists a natural collapse map 
\begin{equation} \label{eqn:collapse}
 {\sf c}:(\Bmr\Fscr)^\Kmr_+ \longrightarrow  \Bmr(\Fscr^{\Kmr})_+ 
\end{equation}
which is identity on $\Bmr(\Fscr^{\Kmr})$ and maps the elements from other components to the disjoint basepoint. Moreover,  an argument similar to \cite[Lemma 2.18]{BZ} reveals that the pullback of the the $\Kmr$-fixed point of the bundle $\td{\upgamma}_{\Vmr}^{\mathscr{F}}$ along $\wh{\uplambda}$ is isomorphic to  $\td{\upgamma}_{\Vmr^\Kmr}^{\mathscr{F}^\Kmr}$:
\[ \wh{\uplambda}^* \left( (\td{\upgamma}_{\Vmr}^{\mathscr{F}})^{\Kmr} \right) \cong  \td{\upgamma}_{\Vmr^\Kmr}^{\mathscr{F}^\Kmr}\]
Consequently,  $\widehat{\uplambda}^*(\td{\bf e}_\Vmr) $ (where $\td{\bf e}_\Vmr$ as defined  in \eqref{eqn:euler}) is an $\Phi^\Kmr(\Rmr)$-Euler class of $\td{\upgamma}_{\Vmr}^{\mathscr{F}^\Kmr}$. This motivates the following definition: 
\begin{defn} \label{defn:geofixEulerSeq} Suppose $\upchi =(\sfx_0, \sfx_1, \sfx_2, \dots )$, where $\sfx_i \in \Rmr_{\star}^\Gmr( \Bmr \Fscr_+),$ is  a $\Vmr$-stable $\Rmr$-Eulerian sequence.   Then  ${\sf c}_* \varphi^{\Kmr}( \sfx_i) \in \Rmr_{\star}^{\Wmr(\Kmr)}( \Bmr \Fscr_+^\Kmr)$ and 
 \[ 
 \varphi^{\Kmr} (\upchi) := ( {\sf c}_* \varphi^{\Kmr}(\sfx_0), {\sf c}_* \varphi^{\Kmr}( \sfx_1),  {\sf c}_* \varphi^{\Kmr}( \sfx_2), \dots  )
 \]
is a $\Vmr^{\Kmr}$-stable $\Phi^{\Kmr}(\Rmr)$-Eulerian sequence, which we will call the {\bf $\Kmr$-geometric fixed points of $\upchi$}.
\end{defn}

The fact that $\varphi^{\Kmr} (\upchi)$ in \Cref{defn:geofixEulerSeq} is a $\Vmr^{\Kmr}$-stable $\Phi^{\Kmr}(\Rmr)$-Eulerian sequence follows from:
\begin{lem}\label{lem:geoFixedPointsES} Suppose $x \in \Rmr_\star^\Gmr(\Bmr \Fscr_+)$ and $e \in \Rmr^\star_\Gmr(\Bmr \Fscr_+)$, then 
\[ 
{\sf c}_* \varphi^\Kmr(x) \frown \wh{\uplambda}^\ast(e) = {\sf c}_*\varphi^\Kmr(x \frown e). 
\]
\end{lem} 
\begin{proof} The result follows from the commutative diagram:
\[ 
\begin{tikzcd}
 \Smr^{\Vmr^\Kmr} \dar[DarkBrown, "{\sf c}_* \varphi^\Kmr(x)"'] \ar[DarkBlue, bend left, rrd, "\varphi^\Kmr(x)"]  \\
\Bmr (\Fscr^{\Kmr})_+ \sma \Phi^\Kmr(\Rmr) \ar[rr, "\wh{\uplambda} \sma 1"]  \dar[DarkBrown, "\Delta \sma 1"'] && (\Bmr \Fscr^\Kmr)_+ \sma \Phi^\Kmr(\Rmr) \dar[DarkBlue, "\Delta \sma 1"]  \\
\Bmr (\Fscr^{\Kmr})_+ \sma \Bmr (\Fscr^{\Kmr})_+ \sma \Phi^\Kmr(\Rmr) \ar[rr, "\wh{\uplambda} \sma \wh{\uplambda} \sma 1 "] \dar[DarkBrown, "1 \sma \wh{\uplambda}^*(e) \sma 1"'] && (\Bmr \Fscr^\Kmr)_+ \sma (\Bmr \Fscr^\Kmr)_+ \sma \Phi^\Kmr(\Rmr) \dar[DarkBlue, "1 \sma e \sma 1 "] \\
\Bmr (\Fscr^{\Kmr})_+ \sma \Sigma^{\Wmr^\Kmr}\Phi^\Kmr(\Rmr)  \sma \Phi^\Kmr(\Rmr) \ar[rr, "\wh{\uplambda} \sma 1 \sma 1"'] \ar[DarkBrown, bend right, ddrr, "1 \sma \upmu_\Rmr"']  && (\Bmr \Fscr^\Kmr)_+ \sma \Phi^\Kmr(\Sigma^{\Wmr}\Rmr)  \sma \Phi^\Kmr(\Rmr) \ar[DarkBlue, dd, "{\sf c} \sma \upmu_\Rmr"] \\ \\
 && \Bmr (\Fscr^{\Kmr})_+ \sma \Sigma^{\Wmr^\Kmr}\Phi^\Kmr(\Rmr)
\end{tikzcd}
\]
where the composition of the red arrows is the left hand side and the composition of the blue arrows is the right hand side. 
\end{proof}

\begin{thm}\label{thm:opsRestrAndGeoFixedPts} Suppose $\upchi = (\sfx_0, \sfx_1, \dots )$ in an $\Vmr$-stable  $\Rmr$-Eulerian sequence. Then  
\begin{enumerate}
\item $\upiota_{\Kmr*}\Sfrak^{\upchi}(-) = \Sfrak^{\upiota_{\Kmr}(\upchi)}(\upiota_{\Kmr*} (-))  $
\item $\varphi^\Kmr (\Sfrak^{\upchi}(- )) = \Sfrak^{\varphi^\Kmr(\upchi)}(\varphi^{\Kmr}(- )) $
\end{enumerate}
for all $\RO(\Gmr, \Vmr)$-graded cohomology classes. 
\end{thm}
\begin{proof} By \Cref{thm:stable}, it is enough to show that  {\it (1)} and {\it (2)} hold for any $\Rmr$-cohomology class $x \in \Rmr_{\Gmr}^{t\Vmr}(\Xmr)$, where $t \in \NN$.  

From   \Cref{thm:Extendpower}{\it (1)} and the definition of slant product (see \Cref{notn:slant})
\begin{eqnarray*}
\upiota_{\Kmr*} \Sfrak^{\upchi}(x) &=& \upiota_{\Kmr*} (\updelta^*\Pcal^{\mathscr{F}}_{t\Vmr}(x)|_{{\sf x}_t}) \\
&=& \updelta^* \upiota_{\Kmr*}(\Pcal^{\mathscr{F}}_{t\Vmr}( x) )|_{\upiota_{\Kmr*}({\sf x}_t)} \\
&=& \updelta^* \Pcal^{\upiota_{\Kmr}\mathscr{F}}_{t\upiota_{\Kmr*}\Vmr}( \upiota_\Kmr(x) )|_{\upiota_{\Kmr*}({\sf x}_t)} \\
&=& \Sfrak^{\upiota_{\Kmr}(\upchi)}(\upiota_{\Kmr*} (x )). 
\end{eqnarray*}
To prove {\it (2)}, we first observe that there is a commutative diagram
\[ 
\begin{tikzcd}
 \Bmr(\Fscr^{\Kmr})_+ \sma \Phi^{\Kmr}(\Xmr) \ar[rrr, "\updelta"] \dar["\wh{\uplambda} \sma 1"'] &&& \Dmr^{\Fscr^\Kmr}_{\upkappa^*{\sf n}}(\Phi^{\Kmr}(\Xmr))  \dar["\widehat{\uplambda}"]\\
 \Phi^{\Kmr}(\Bmr \Fscr_+ \sma \Xmr)   \cong  (\Bmr \Fscr)^{\Kmr}_+\sma \Phi^{\Kmr}(\Xmr)  \ar[rrr, "\Phi^{\Kmr}(\updelta)"'] &&&  \Phi^\Kmr(\Dmr^{\Fscr}_{\upkappa^*{\sf n}}(\Xmr)).
\end{tikzcd}
\] 
Combining this with  \Cref{thm:Extendpower}{\it (2)} and  the fact that ${\sf c} \circ \wh{\uplambda}$ is identity, we get 
\begin{eqnarray*}
\Sfrak^{\varphi^\Kmr(\upchi)}(\varphi^{\Kmr}(x )) &=& \updelta^*\Pcal^{\mathscr{F}^\Kmr}_{t\Vmr^\Kmr}(\varphi^\Kmr(x))|_{{\sf c}_*\varphi^\Kmr({\sf x}_t)} \\
&=& \updelta^* \wh{\uplambda}^* \varphi^{\Kmr}( \Pcal^{\Fscr}_{t\Vmr}(x))|_{{\sf c}_*\varphi^\Kmr({\sf x}_t)} \\ 
&=& (\wh{\uplambda} \sma 1)^* \Phi^{\Kmr}(\updelta)^* \varphi^{\Kmr}( \Pcal^{\Fscr}_{t\Vmr}(x)) |_{{\sf c}_*\varphi^\Kmr({\sf x}_t)} \\
&= & \varphi^{\Kmr}\left( \updelta^* \Pcal^{\Fscr}_{t\Vmr}(x) |_{{\sf x}_t}\right) \\
&=& \varphi^\Kmr (\Sfrak^{\upchi}({\sf x}_t ))
\end{eqnarray*}
as desired. 
\end{proof} 
\subsection{Modified geometric fixed-point of $\Hmr\ull{\FF}_p$-Eulerian Sequence} \

In \Cref{sec:new}, we demonstrate that $\widetilde{\upgamma}^{\All_n}_{\uprho_\Gmr}$ (as defined in \Cref{notn:equiBundle}) is $\H\ull{\FF}_p$-orientable when $|\Gmr|$ is even. Furthermore, we show that $2\widetilde{\upgamma}^{\All_n}_{\uprho_\Gmr}$ is always $\H\ull{\FF}_p$-orientable for all finite groups $\Gmr$. We then proceed to identify the $\uprho_\Gmr$-stable $\Hmr\ull{\FF}_p$-Eulerian sequence. The primary purpose of this subsection is to lay the foundation for analyzing the modified geometric fixed-points of the corresponding cohomology operations.

Recall that the modified geometric fixed-point map  of \eqref{eqn:modgeofix} is the composition: 
\[ 
\td{\varphi}^\Kmr (-) := \uppi_*\varphi^{\Kmr}(-)\]
where 
\[ \begin{tikzcd}
\uppi: \Phi^\Kmr(\Hmr\ull{\FF}_p) \rar & \Hmr\ull{\FF}_p, 
\end{tikzcd}
\]
is the $\EE_\infty^{\Wmr(\Kmr)}$-ring map from the $\Kmr$-geometric fixed-point of the $\Gmr$-Eilenberg MacLane spectrum $\Hmr\ull{\FF}_p$ to the $\Wmr(\Kmr)$-Eilenberg MacLane  spectrum $\Hmr\ull{\FF}_p$, induced by the zeroth Postnikov approximation. 

 \begin{defn} \label{defn:modgeoES}For a $k\uprho_\Gmr$-stable  $\Hmr\ull{\FF}_p$-Eulerian sequence $\upchi = (\sfx_0, \sfx_1, \sfx_2, \dots )$  such that ${\sf x}_i \in \Hmr_{\star}^{\Gmr}(\Bmr_\Gmr \Sigma_n ; \ull{\FF}_p)$, we define  its modified geometric fixed-point as 
\begin{equation} \label{eqn:modifiedES}
\td{\varphi}^\Kmr(\upchi) := \uppi_*\varphi^\Kmr(\upchi)= ( \uppi_*{\sf c}_* \varphi^{\Kmr}(\sfx_0), \uppi_*{\sf c}_* \varphi^{\Kmr}( \sfx_1),  \uppi_*{\sf c}_* \varphi^{\Kmr}( \sfx_2), \dots  ). 
  \end{equation}
  \end{defn}
  \begin{rmk}   Since $\sfc_*$ and $\uppi_*$ commute, we may rewrite \eqref{eqn:modifiedES}
    \[ \td{\varphi}^\Kmr(\upchi) = ( {\sf c}_* \td{\varphi}^{\Kmr}(\sfx_0), {\sf c}_* \td{\varphi}^{\Kmr}( \sfx_1),  {\sf c}_* \td{\varphi}^{\Kmr}( \sfx_2), \dots  ).\]
  Note $\td{\varphi}^\Kmr$ sends the $\H\ull{\FF}_p$-Euler class of $ k \widetilde{\upgamma}^{\All_n}_{\uprho_\Gmr} $, denote it by $\td{\bf e}_{\Gmr, p}$,  to $\td{\bf e}_{\Wmr(\Kmr), p}$, the $\H\ull{\FF}_p$-Euler class of  $ k \widetilde{\upgamma}^{\All_n}_{\uprho_{\Wmr(\Kmr)}} $. 
This along with the fact that 
\[ \td{\varphi}^\Kmr(x \frown e) = \td{\varphi}^\Kmr(x) \frown \td{\varphi}^\Kmr(e) \]
implies that $\td{\varphi}^\Kmr(\upchi)$ is a $k\uprho_{\Wmr(\Kmr)}$-stable  $\Hmr\ull{\FF}_p$-Eulerian sequence. 
  \end{rmk}
Combining  \Cref{thm:opsRestrAndGeoFixedPts} with the fact that $\uppi$ is an $\EE_\infty^{\Wmr(\Kmr)}$-ring map, we conclude:
\begin{thm} \label{thm:modgeoES} Any $k\uprho_\Gmr$-stable $\Hmr\ull{\FF}_p$-Eulerian sequence $\upchi = (\sfx_1, \sfx_2, \dots )$ satisfies:
\begin{equation} \label{Sqmgeofix}
\td{\varphi}^\Kmr(\Sfrak^{\upchi}(-)) = \Sfrak^{\td{\varphi}^\Kmr(\upchi)}( \td{\varphi}(-))
\end{equation}
for any subgroup $\Kmr$ of $\Gmr$.
\end{thm}
\begin{proof} Applying $\uppi_*$ to \Cref{thm:opsRestrAndGeoFixedPts} \emph{(2)}, we get 
\begin{eqnarray*}
\td{\varphi}^\Kmr (\Sfrak^{\upchi}(-)) &=& \uppi_* \left( \varphi^\Kmr (\Sfrak^{\upchi}(-)) \right) \\
&=& \uppi_*\left(  \Sfrak^{\varphi^\Kmr(\upchi)}(\varphi^{\Kmr}(- ))  \right) \\
&=&  \uppi_*\left( \updelta^*\Pcal^{\mathscr{F}^\Kmr}_{t\Vmr^\Kmr}(\varphi^\Kmr(-))|_{{\sf c}_*\varphi^\Kmr({\sf x}_t)}\right) \\
&=&   \updelta^*\Pcal^{\mathscr{F}^\Kmr}_{t\Vmr^\Kmr}(\uppi_*\varphi^\Kmr(-)) |_{\uppi_* \left({\sf c}_*\varphi^\Kmr({\sf x}_t) \right)} \\
&=& \updelta^*\Pcal^{\mathscr{F}^\Kmr}_{t\Vmr^\Kmr}(\td{\varphi}^\Kmr(-)) |_{ {\sf c}_*\td{\varphi}^\Kmr({\sf x}_t) } \\
&=& \Sfrak^{\td{\varphi}^\Kmr(\upchi)}( \td{\varphi}(-))
\end{eqnarray*}
as desired. 
\end{proof}
  \section{Generalized Cartan formula} \label{sec:Cartan}
  
The classical Cartan formula encodes the relationship between Steenrod operations and the external product. The standard formulation, however, relies on the K\"unneth isomorphism, which does not hold for most equivariant cohomology theories satisfying \Cref{assump:generic}. In this section, we overcome this challenge by introducing the diagonal of an Eulerian sequence and formulating a generalization of the Cartan formula (see \Cref{thm:GenCartan}). 

Recall that any multiplicative $\Rmr$-cohomology theory admits a natural external product pairing:
\[ 
\begin{tikzcd}
(-)\times(-): \Rmr_{\Gmr}^{\star}(\Xmr) \times \Rmr_{\Gmr}^{\star}(\Ymr) \rar & \Rmr_{\Gmr}^{\star}(\Xmr \sma \Ymr)
\end{tikzcd}
\]
This pairing sends elements $x\in \Rmr_{\Gmr}^{\star}(\Xmr)$ and $y\in \Rmr_{\Gmr}^{\star}(\Ymr)$ to the composite map
\begin{equation} \label{eqn:cross}
\begin{tikzcd} 
x \times y: \Xmr \sma \Ymr \rar["x \sma y"] & \Sigma^{|x|} \Rmr \sma \Sigma^{|y|} \Rmr \rar["\upmu_\Rmr"] & \Sigma^{|x| + |y|} \Rmr. 
\end{tikzcd}
\end{equation}
The first step toward establishing the Cartan formula is to examine the following relationship between power operations and the external product:
\begin{thm} \label{power:cross} Suppose $\Rmr, \upkappa, \Fscr$ satisfy the condition of \Cref{assump:generic}.  For elements $x\in \Rmr_{\Gmr}^{i_1 \Vmr}(\Xmr)$ and $y\in \Rmr_{\Gmr}^{i_2 \Vmr}(\Ymr)$ with $i_1,i_2 \in \NN$ the  formula 
\begin{equation} \label{eq:Pcartan}
\Pcal^{\Fscr}_{i \Vmr}(x \times y)  = (\partialup_{\upkappa^*{\sf n}}^{\Fscr})^* \left( \Pcal^{\Fscr}_{i_1 \Vmr}(x ) \times \Pcal^{\Fscr}_{i_2 \Vmr}(y) \right),
\end{equation}
where $i= i_1 + i_2$, relates power operations with external products. 
\end{thm} 
\begin{proof} For simplicity, let $\uptau_k^\Rmr $ denote the composite 
\[ 
\begin{tikzcd}
 \uptau_i^\Rmr: \Dmr^{\mathscr{F}}_{ \upkappa^*{\sf n}}(\Sigma^{i \Vmr}\Rmr) \rar["\partialup_{\upkappa^*{\sf n}}^\Fscr"] &  \Dmr^{\mathscr{F}}_{ \upkappa^*{\sf n}}(\Smr^{i \Vmr}) \sma  \Dmr^{\mathscr{F}}_{ \upkappa^*{\sf n}}(\Rmr) \ar[rr, "{\bf u}_{ i\Vmr} \sma \uptheta_{n}^{ \Rmr}"] && \Sigma^{in\Vmr} \Rmr \sma \Rmr \rar["\upmu_\Rmr"] & \Sigma^{in\Vmr} \Rmr, 
\end{tikzcd}
\]
where ${\bf u}_{i\Vmr}$ be an $\Rmr$-Thom class of  $\upgamma^{\mathscr{F}}_{i\Vmr}$ (as in \eqref{eqn:TCgammaF}). These $\Rmr$-Thom classes are compatible as $i$ varies, in the sense that the diagram 
\begin{equation} \label{eq:thomcampatible}
\begin{tikzcd}
\Th(\upgamma^{\mathscr{F}}_{i\Vmr}) \rar["\partialup_{\upkappa^*{\sf n}}"] \ar[bend right , rrrd, "{\bf u}_{i \Vmr}"'] &  \Th(\upgamma^{\mathscr{F}}_{i_1\Vmr}) \sma \Th(\upgamma^{\mathscr{F}}_{i_2\Vmr}) \ar[rr,"{\bf u}_{i_1\Vmr} \sma {\bf u}_{i_2\Vmr}"] && \Sigma^{i_1 \Vmr }  \Rmr  \sma   \Sigma^{i_2 \Vmr }  \Rmr \dar["\upmu_{\Rmr}"] \\
&&& \Sigma^{i \Vmr} \Rmr
\end{tikzcd}
\end{equation}
commutes whenever $i = i_1 + i_2$. Consequently, we have a commutative diagram
\[ 
\begin{tikzcd}
 \Dmr^{\mathscr{F}}_{ \upkappa^*{\sf n}}(\Xmr\sma \Ymr) \rar["\partialup_{\upkappa^*{\sf n}}", color = DarkBrown] \ar[dd,color = DarkBlue,  " \Dmr^{\mathscr{F}}_{ \upkappa^*{\sf n}}(x \sma y)"']  & \Dmr^{\mathscr{F}}_{ \upkappa^*{\sf n}}(\Xmr) \sma  \Dmr^{\mathscr{F}}_{ \upkappa^*{\sf n}}(\Ymr) \ar[dd, color = DarkBrown, "\Dmr^{\mathscr{F}}_{ \upkappa^*{\sf n}}(x) \sma \Dmr^{\mathscr{F}}_{ \upkappa^*{\sf n}}(y) "] \\ \\
  \Dmr^{\mathscr{F}}_{ \upkappa^*{\sf n}}(\Sigma^{i_1\Vmr}\Rmr \sma \Sigma^{i_2\Vmr}\Rmr ) \rar["\partialup_{\upkappa^*{\sf n}}"'] \ar[dd, color = DarkBlue, "\Dmr^{\mathscr{F}}_{ \upkappa^*{\sf n}}(\upmu_\Rmr) "'] & \Dmr^{\mathscr{F}}_{ \upkappa^*{\sf n}}(\Sigma^{i_1\Vmr}\Rmr) \sma \Dmr^{\mathscr{F}}_{ \upkappa^*{\sf n}}(\Sigma^{i_2\Vmr}\Rmr ) \dar[color = DarkBrown,"\uptau_{i_1} \sma \uptau_{i_2}"]  \\
  & \Sigma^{i_1n\Vmr} \Rmr \sma \Sigma^{i_2n\Vmr} \Rmr \dar[color = DarkBrown, "\upmu_\Rmr"] \\ 
  \Dmr^{\mathscr{F}}_{ \upkappa^*{\sf n}}(\Sigma^{i \Vmr}\Rmr)  \rar[color = DarkBlue, "\uptau^{\Rmr}_i"'] & \Sigma^{in\Vmr} \Rmr 
 \end{tikzcd},
\]
where the top square commutes because of naturality $\partialup_{\upkappa^*{\sf n}}^\Fscr$ and bottom square commutes because of \eqref{eq:thomcampatible}. In the above diagram, the composition of the blue arrows is the left hand side and composition of the red arrows is the right hand side of \eqref{eq:Pcartan}, hence the result. 
\end{proof}
Now notice, there is a commutative diagram 
\begin{equation} \label{eqn:diagonal}
\begin{tikzcd}
\Bmr \Fscr_+ \sma \Xmr_1 \sma \dots \sma \Xmr_k \ar[rrr, "\Delta_k \sma 1_{\Xmr_1 \sma\dots \sma  \Xmr_k}"] \ar[dd, "\updelta"'] &&& \overbrace{\Bmr \Fscr_+ \sma \dots \sma \Bmr \Fscr_+}^{k\text{-fold}} \sma  \Xmr_1 \sma \dots \sma \Xmr_k \dar[cong] \dar[bend left = 90, dd, "\updelta^{(k)}"] \\
&&& \Bmr \Fscr_+ \sma   \Xmr_1 \sma \dots  \sma \Bmr \Fscr_+ \sma  \Xmr_k  \dar["\updelta \sma \dots \sma \updelta"] \\
\Dmr^{\mathscr{F}}_{ \upkappa^*{\sf n}}(\Xmr_1 \sma \dots \sma \Xmr_k) \ar[rrr,"\partialup_{\upkappa^*{\sf n}}^\Fscr"'] &&& \Dmr^{\mathscr{F}}_{ \upkappa^*{\sf n}}(\Xmr_1) \sma \dots \sma \Dmr^{\mathscr{F}}_{ \upkappa^*{\sf n}}(\Xmr_k),
\end{tikzcd}
\end{equation}
where  
\[ 
\begin{tikzcd}
\Delta_k : \Bmr \Fscr  \rar &  \overbrace{\Bmr \Fscr \times \dots \times \Bmr \Fscr}^{k\text{-fold}}
\end{tikzcd}
\]
is the $k$-fold diagonal map.   

 Since the pullback of $\td{\upgamma}_{\Vmr}^{\Fscr} \times \dots \times \td{\upgamma}_{\Vmr}^{\Fscr}$ along $\Delta_k$ is $k \td{\upgamma}_{\Vmr}^{\Fscr}$, we make the following definition.  
 \begin{defn}
For a $\Vmr$-stable $\Rmr$-Eulerian sequence $\upchi = (\sfx_0, \sfx_1, \sfx_2, \dots )$,  the sequence
\[ \Delta_k (\upchi) := ( \Delta_{k\ast} \sfx_0, \Delta_{k\ast} \sfx_k, \Delta_{k\ast} \sfx_{2k}, \dots   ) \]
is defined as the {\bf diagonal of $\upchi$}. 
\end{defn}
This sequence exhibits Eulerian like properties  with respect to  the $\Rmr$-Euler class $\td{\bf e}_{\Vmr}^{\times k}$ of the bundle  $(\td{\upgamma}_{\Vmr}^{\Fscr})^{\times k}$. Indeed, the naturality of cap product implies:
\begin{eqnarray*} 
(\Delta_{k\ast} \sfx_{ik}) \frown \td{\bf e}_{\Vmr}^{\times k} &=& \Delta_{k*} \left( \sfx_{ik} \frown \Delta_{k}^{\ast} (\td{\bf e}_{\Vmr}^{\times k})\right) \\
&=&  \Delta_{k*} \left( \sfx_{ik} \frown \td{\bf e}_{k \Vmr} \right) \\
&=&  \Delta_{k*} \sfx_{(i-1)k}. 
\end{eqnarray*}
However, it is not an Eulerian sequence according to \Cref{defn:Eulerian}, as $(\td{\upgamma}_{\Vmr}^{\Fscr})^{\times k}$ is not of the form \eqref{bundle:TautVF}. This is because the $\Sigma_n \times \cdots \times \Sigma_n$-representation $\td{\uptau}_n \times \cdots \times \td{\uptau}_n$ is not a pullback of $\td{\uptau}_{d}$ for any $d \in \NN$.   

We call a sequence  $\wh{\upchi} = (\wh{\sfx}_0, \wh{\sfx}_1, \dots)$ a {\bf pseudo $\Rmr$-Eulerian sequence} if:
\begin{itemize}
\item  $\wh{\sfx}_i \in \Rmr^{\Gmr}_{\star}(\overbrace{\Bmr \Fscr_+ \sma \dots \sma\Bmr \Fscr_+}^{k\text{-fold}})$
\item  $\wh{\sfx}_{i+1} \frown \td{\bf e}_{\Vmr}^{\times k} = \wh{\sfx}_i $ 
\item $\wh{\sfx}_0 \frown \td{\bf e}_{\Vmr}^{\times k}   = 0 $. 
\end{itemize}
Given a pseudo $\Rmr$-Eulerian sequence  $\wh{\upchi} = (\wh{\sfx}_0, \wh{\sfx}_1, \dots)$,  define the cohomology operation 
\[ 
\begin{tikzcd}
\Sfrak^{\wh{\upchi}} : \Rmr^{t_1\Vmr}_\Gmr (\Xmr_1 ) \times \dots \times  \Rmr^{t_k\Vmr}_\Gmr (\Xmr_k) \rar & \Rmr^{t\Vmr}_\Gmr (\Xmr_1 \sma \dots \sma \Xmr_k).
\end{tikzcd}
\]
 by setting 
\begin{equation}
\Sfrak^{\wh{\upchi}}(x_1, \dots, x_n) = (\updelta^{(k)})^* \left(\Pcal_{t_1 \Vmr}^{\Fscr}(x_1) \times \dots \times \Pcal_{t_1 \Vmr}^{\Fscr}(x_k)\right)|_{\wh{\sfx}_t},
\end{equation} 
where $t = t_1 + \dots + t_k$,.  A diagram chase similar to \eqref{master} shows that above  operation is stable, i.e., 
\[ 
\Sfrak^{\wh{\upchi}} \left( \upsigma_{\Vmr}(x_1), \dots, \upsigma_{\Vmr}(x_k) \right) = \upsigma_{k\Vmr}\left( \Sfrak^{\wh{\upchi}}(x_1, \dots, x_k) \right)
\]
and extends to $\RO(\Gmr, \Vmr)$-graded cohomology classes. 

For any $\Rmr$-Eulerian sequence $\upchi = (\sfx_0, \sfx_1, \dots )$,  $\Delta_k(\upchi)$ is a pseudo $\Rmr$-Eulerian sequence and the corresponding stable $\Rmr$-cohmology operation satisfies the following relation: 
\begin{thm}[{\bf Generalized Cartan formula}] \label{thm:GenCartan} For any  $\Rmr$-Eulerian sequence $\upchi$
\[ 
\Sfrak^{\upchi}(x_1 \times \dots \times x_k) = \Sfrak^{\Delta_k(\upchi)}(x_1, \dots, x_k), 
\]
where $x_j$, for each $1 \leq j \leq k$, is an $\RO(\Gmr, \Vmr)$-cohomology class. 
\end{thm}
\begin{proof} It is enough to establish the result when $x_i \in \Rmr^{t_i \Vmr}_\Gmr(\Bmr\Fscr_+)$ for some $t_i \in \NN$ for each $j \in \{ 1, \dots, k \} $. In the diagram 
\[ 
\begin{tikzcd}
\Smr^{|\sfx_{ik}|} \sma \left( \bigwedge_{j =1}^k \Xmr_j \right) \dar[DarkBlue, "\sfx_{\sfi k} "'] \ar[DarkBrown, rrd, bend left, "\Delta_{k \ast}\sfx_{\sfi k}"] \ar[rrd, phantom, "\text{\bf (I)}"]  \\
\Bmr \Fscr_+ \sma \left( \bigwedge_{j =1}^k \Xmr_j \right)  \sma \Rmr  \ar[rr, "\Delta_k \sma 1"] \dar[DarkBlue, "\updelta \sma 1_{\Rmr}"'] \ar[rrd, phantom, "\text{\bf (II)}"] && \Bmr \Fscr^{\times k}_+ \sma \left( \bigwedge_{j =1}^k \Xmr_j \right)  \sma \Rmr \dar[DarkBrown, "\updelta^{(k)} \sma 1_{\Rmr}"]  \\
\Dmr_{\upkappa^*\sfn}^{\Fscr}\left( \bigwedge_{j =1}^k \Xmr_j \right) \sma \Rmr  \ar[rr, "\partialup_{\upkappa^*{\sf n}}^\Fscr \sma 1_\Rmr"'] \dar[DarkBlue, "\Pcal_{t \Vmr}^\Fscr(\sfx_1 \times \dots \times \sfx_k) \sma 1_\Rmr"'] \ar[rrdd, phantom, "\text{\bf (III)}"] && \left( \bigwedge_{j =1}^k \Dmr_{\upkappa^*\sfn}^{\Fscr}(\Xmr_j) \right) \sma \Rmr \dar[DarkBrown, "\left(\bigwedge\Pcal_{t_i \Vmr}^{\Fscr}(x_j)\right) \sma 1_{\Rmr} " ] \\
\Sigma^{t \Vmr} \Rmr \sma \Rmr \ar[DarkBlue, rrd, bend right, "\upmu_\Rmr"'] &&  \left( \bigwedge_{j =1}^k \Sigma^{t_j \Vmr} \Rmr  \right) \sma \Rmr  \dar[DarkBrown, "\upmu_\Rmr"] \\
&& \Sigma^{t\Vmr}\Rmr,
\end{tikzcd}
\]
it is easy to see {\bf (I)} and {\bf (II)} commutes, and {\bf (III)} commutes because of \Cref{power:cross}. Since the composition of blue arrows is the left hand side and the composition of the red arrows is the right hand side, the result follows.  
\end{proof}
In the following remark, we show that \Cref{thm:GenCartan} recovers the standard Cartan formula in the classical setting where the K\"unneth isomorphism holds.
\begin{rmk}   Recall from \Cref{ex:classicalEulerianSteen} that $\Sq^{k} = \Sfrak^{\upbeta[k]}$ where 
\[
\upbeta[k] = (\overbrace{0, \dots, 0}^{k\text{-fold}}, {\bf b}_0, {\bf b}_1, \dots )
\]
The diagonal map $\Delta_2: \Hmr_{\ast}(\Bmr\Sigma_{2+}; \FF_2) \to \Hmr_{\ast}\left((\Bmr \Sigma_2 \times \Bmr \Sigma_2)_+; \FF_2\right)$ sends ${\bf b}_l$ to 
\[ \Delta_2({\bf b}_l) = {\bf b}_l \otimes {\bf b}_0 +  {\bf b}_{l-1} \otimes {\bf b}_1 + \dots + {\bf b}_0 \otimes {\bf b}_l.\] 
Suppose, $x \in \Hmr^{i}(\Xmr; \FF_2)$ and  $y \in \Hmr^{j}(\Ymr; \FF_2)$ such that $i+j$ is an even number greater than $k$. Set $l = i+j -k$.  By \Cref{thm:GenCartan} 
\begin{eqnarray*}
\Sq^{k}(x \otimes y) &=& (\updelta \sma \updelta)^*(\Pcal_i(x) \otimes \Pcal_j(y))|_{\Delta_2({\bf b}_l)} \\ 
&=& \left(\updelta^*(\Pcal_i(x)) \otimes \updelta^*(\Pcal_j(x)) \right)|_{\sum_{l' +l'' = l } {\bf b}_{l'} \otimes  {\bf b}_{l''} } \\
&=&  \sum_{l' + l'' = l} \updelta^*(\Pcal_i(x))|_{{\bf b}_{l'}} \otimes \updelta^*(\Pcal_j(y))|_{ {\bf b}_{l''} } \\
&=& \sum_{l' + l'' = l} \Sq^{i - l'}(x) \otimes \Sq^{j - l''}(y)
\end{eqnarray*}
which is equivalent to the classical Cartan formula after implementing the unstable conditions. 
\end{rmk}

\section{Homotopy $\Ncal_\infty$ rings and the composition law} \label{sec:compose}

The main goal of this section is to establish a product law for $\Rmr$-Eulerian sequences in a way that emulates the composition of their corresponding stable cohomology operations. We note that this product law does  not require the strict commutativity among the structure maps of an $\Ncal_\infty$-ring $\Rmr$; instead we need these diagrams to commute up to a homotopy. However, we require certain compatibility among the 
$\Rmr$-Euler classes of \eqref{eqn:euler} across all $n \in \NN$. In the classical case, this requirement is precisely the $\mathbb{H}_\infty^d$-ring structures introduced in \cite{Hinfty}. We  generalize this to equivariant settings to introduce shifted homotopy $\Ncal_\infty$-rings before establishing the composition law. We  also identify the shift degree of the homotopy $\EE_\infty^{\Gmr}$-ring structures for Eilenberg MacLane spectra. 

Suppose $\Ocal$ is an $\Ncal_\infty$ $\Gmr$-operad. Let $\Fscr_n$ denote the $\Gmr$-closed family of $(\Gmr \times \Sigma_n)$ (see \Cref{defn:family}) such that the $n$-th space $\Ocal(n)$ is equivalent to the universal space $\Emr\Fscr_n$. Since $\Fscr_n$ is $\Gmr$-closed,  the $\Gmr$-fixed point space of $\Ocal(n)$ is contractible. Consequently, there is  a  contractible choice of $\Gmr$-equivariant maps 
\begin{equation} \label{map:iota}
\begin{tikzcd}
\iota_n: \ast \rar[hook] & \Ocal(n).
\end{tikzcd}
\end{equation}
This gives rise to a diagram of $\Gmr$-equivariant maps
\begin{equation} \label{diag:compatibletau}
\begin{tikzcd}
\ast \times (\ast \times \dots \times \ast ) \ar[rrr, "\iota_{n} \times (\iota_{i_1} \times \dots \times \iota_{i_k}) "] \dar[equal] &&& \Ocal(n) \times \Ocal(i_1) \times \dots \times \Ocal(i_n) \dar  \\
\ast \ar[rrr, "\iota_{i_1 + \dots + i_n}"'] &&& \Ocal(i_1 + \dots + i_n)
\end{tikzcd}
\end{equation}
which commutes up to a $\Gmr$-equivariant homotopy for all  $n, i_1, \dots, i_n \in \NN$. 
\begin{rmk} When $n =1$, we  choose $\iota_1$ such that it maps $\ast$ to the  distinguished element $1 \in \Ocal(1)$. Since,  $\Gmr \times \Sigma_1$ has exactly one $\Gmr$-closed family, $\iota_1$ of \eqref{map:iota} is an equivalence. 
\end{rmk}
By combining $\iota_n$ of \eqref{map:iota} with the structure map \eqref{map:structure}  of the operad $\Ocal$, we obtain a  map 
\[ 
\begin{tikzcd} 
\Ocal(i_1) \times  \dots \times \Ocal(i_n) \rar & \Ocal(i_1 + \dots + i_n).
\end{tikzcd}
\] 
This gives rise to a  natural $\Gmr$-equivariant map 
\[
\begin{tikzcd}
\alpha_{i_1, \dots, i_k}: \Dmr_{{\sf i}_1}^{\Ocal}(\Xmr) \sma \dots \sma \Dmr_{{\sf i}_n}^{\Ocal}(\Xmr) \rar & \Dmr_{i_1 + \dots + i_n}^{\Ocal}(\Xmr) 
\end{tikzcd}
 \]
which satisfies the following property: 
\begin{lem}  \label{lem:alpha} For all $i, j, k \in \NN$, 
\[ 
\alpha_{i + j, k} \circ (\alpha_{i, j } \sma 1) \simeq \alpha_{i, j, k} \simeq \alpha_{i, j + k} \circ( 1 \sma \alpha_{j, k} ).
\]
\end{lem} 
\begin{proof} To abbreviate notations, let $\Ocal(i_1, \dots, i_n):= \Ocal(i_1) \times \dots \times  \Ocal(i_n) $. Consider the commutative diagram: 
\[ 
\begin{tikzcd}
\ast \times (\ast \times \ast) \times \Ocal(i,j,k) \dar \rar[equal] & \ast \times \Ocal(i,j, k) \dar  & \lar[equal]\ast \times (\ast \times \ast) \times \Ocal(i,j,k)   \dar  \\ 
\Ocal(2) \times \Ocal(2, 1) \times \Ocal(i,j, k)  \dar  \rar & \Ocal(3) \times \Ocal(i,j,k) \dar  & \lar  \Ocal(2) \times \Ocal(1, 2) \times \Ocal(i,j, k) \dar  \\
\Ocal(2) \times \Ocal(i+j, k) \rar & \Ocal(i+j +k)  & \lar   \Ocal(2) \times \Ocal(i+j, k)
\end{tikzcd}
\]  
In this diagram, the lower squares commute by the usual compatibility of the evident structure maps of $\Ocal$. The upper squares commute up to  homotopy due to  \eqref{diag:compatibletau}.  The result follows immediately from the homotopy commutativity of the entire diagram above.
\end{proof}

If $\Rmr$ is an $\Ocal$-ring, one can readily verify the existence of the following diagrams, which commute in the homotopy category $\Ho(\Sp_\Gmr)$ for all $i,j \in \NN$:
\begin{equation} \label{Hinfty1}
\begin{tikzcd}
\Dmr_{\sf i}^{\Ocal}(\Rmr) \sma \Dmr_{\sf j}^{\Ocal}(\Rmr) \ar[rr, "\alpha_{\sfi,\sfj}"] \ar[dd] && \Dmr_{\sf i +j}^{\Ocal}(\Rmr)  \ar[dd] \\ \\
\Rmr \sma \Rmr \rar & \Dmr_2^\Ocal(\Rmr) \rar & \Rmr  
\end{tikzcd}
\end{equation} 
\begin{equation} \label{Hinfty2}
\begin{tikzcd}
 \Dmr^{\Ocal}_{\sfi} \Dmr^{\Ocal}_{\sfj}(\Rmr) \dar["\Dmr^{\Ocal}_{\sf i}(\uptheta_{\sf j}^\Rmr)"'] \rar["\upbeta_{\sf i, j}"] & \Dmr^{\Ocal}_{\sfi\sfj}(\Rmr) \dar["\uptheta^{\Rmr}_{\sf ij}"] \\
\Dmr^{\Ocal}_{\sfi}(\Rmr) \rar["\uptheta^{\Rmr}_{\sf i}"] & \Rmr
\end{tikzcd}
\end{equation}
where the map $\upbeta_{i,j}$ is induced by the inclusion map $ \Sigma_j \wr \Sigma_i := \Sigma_j^{\times i} \rtimes \Sigma_i \hookrightarrow \Sigma_{ij}$. 
\begin{rmk} As a consequence of \Cref{lem:alpha}, the diagram of  \eqref{Hinfty1} implies the homotopy commutativity of the general diagram 
\[
\begin{tikzcd}
\Dmr_{{\bf i}_1}^{\Ocal}(\Rmr) \sma \dots \sma \Dmr_{{\sf i}_n}^{\Ocal}(\Rmr) \ar[rr, "\alpha_{\sfi_1, \dots, \sfi_n}"] \ar[dd, "\uptheta_{\sfi_1}^\Rmr \sma \dots \sma \uptheta_{\sfi_n}^\Rmr"'] && \Dmr_{\sf i_1 + \dots + i_n }^{\Ocal}(\Rmr)  \ar[dd, "\uptheta_{\sf i_1 + \dots + i_n}^{\Rmr}"] \\
 \\
\Rmr^{\sma n} \rar["\upalpha_{1, \dots, 1}"'] & \Dmr_\sfn^\Ocal(\Rmr) \rar["\uptheta_{n}^{\Rmr}"'] & \Rmr. 
\end{tikzcd} 
\]
The details are left to the reader to verify.
\end{rmk}
However, the commutativity of diagrams \eqref{Hinfty1} and \eqref{Hinfty2} does not guarantee an $\Ocal$-ring structure on $\Rmr$. For example, in the nonequivariant case  \cite{Noel} gave an example of an $\HH_\infty$ rings which is not $\EE_\infty$. 

\begin{defn} \label{defn:EquivHinfty} Let $\Ocal$ be an $\Ncal_\infty$-operad. An $\Rmr \in \Sp_\Gmr$ is a {\bf homotopy $\Ocal$-ring} (or simply an $\Ocal^{\sf h}$-ring) if for all $n \in \NN$, there exists a map 
\[ 
\begin{tikzcd}
\uptheta_n^\Rmr: \Dmr^{\Ocal}_{\sfn}(\Rmr) \rar & \Rmr
\end{tikzcd}
\]
 such that 
\begin{enumerate}
\item $\uptheta_{1}^\Rmr = 1_\Rmr$, and
\item  the diagrams in \eqref{Hinfty1} and \eqref{Hinfty2} commutes  for all $i,j \in \NN$ 
\end{enumerate}
in $\Ho(\Sp_\Gmr)$.
\end{defn}
\begin{rmk} When $\Gmr$ is the trivial group and $\Ocal$ is an $\Einfty$-operad then \Cref{defn:EquivHinfty} recovers the classical definition of $\HH_\infty$ ring spectra \cite[I.3.1]{Hinfty}. 
\end{rmk}
\begin{rmk} It is easy to check that the unit map $e= \uptheta_0^{\Rmr}: \SS \to \Rmr$ along with the composition 
\[ 
\begin{tikzcd}
\upmu^\Rmr: \Rmr \sma \Rmr \rar["\alpha_{1,1}"] & \Dmr_{2}^{\Ocal}(\Rmr) \rar["\uptheta_2^{\Rmr}"] & \Rmr
\end{tikzcd}
\]
gives $\Rmr$ the structure of a homotopy commutative $\Gmr$-ring, i.e. 
the  diagrams  
\[ 
\begin{tikzcd}
\SS \sma \Rmr \rar["e \sma 1"] \ar[dr, equal]&\Rmr \sma \Rmr \dar["\upmu^\Rmr"]& \Rmr \sma \SS \lar["1 \sma e"'] \ar[dl, equal]\\
& \Rmr
\end{tikzcd} \ \ 
\begin{tikzcd}
\Rmr \sma \Rmr \sma \Rmr \rar["1 \sma \upmu^\Rmr"] \dar[ "\upmu^\Rmr \sma 1 "'] & \Rmr \sma \Rmr \dar["\upmu^\Rmr"] \\
\Rmr \sma \Rmr \rar["\upmu^\Rmr"'] & \Rmr 
\end{tikzcd} \ \
\begin{tikzcd}
\Rmr \sma \Rmr \ar[dd, "\tau"] \ar[dr, "\upmu^\Rmr"] \\
& \Rmr \\
\Rmr \sma \Rmr \ar[ur, "\upmu^\Rmr"']
\end{tikzcd}
\]
in which $\tau$ is the shuffle map, commutes up to $\Gmr$-equivariant homotopies. 
\end{rmk}

\subsection{Shifted $\Ncal_\infty$-rings}  \ 
 
The definition of $\Fscr$-th power operation (as in \Cref{defn:powerF}) applies to any $\Ocal^{\sfh}$-ring $\Rmr$. Consequently, the results of \Cref{subsec:Powerres} and \Cref{subsec:Powerfix} generalize directly to $\Ocal^{\sfh}$-ring spectra. To extend the notion of shifted power operations (as in \Cref{sec:shiftedpower}) to a broader class of equivariant ring spectra, we introduce the equivariant analog of $\HH_\infty^{\sfd}$-rings \cite[I.4.1]{Hinfty}:
\begin{defn} \label{defn:HinftyV} Suppose $\Ocal$ is an $\Ncal_\infty$ $\Gmr$-operad and $\Vmr$ is a finite $\Gmr$-representation. We say that  $\Rmr$ is a {\bf $\Vmr$-shifted homotopy $\Ocal$-ring}  (or simply an $\Ocal^{\sfh}_{[\Vmr]}$-ring) if there are maps 
\begin{equation} \label{map:stHV} 
\begin{tikzcd}
\uptheta_{\sfi, \sft}^{\Rmr} : \Dmr_{\sfi}^{\Ocal}(\Sigma^{t \Vmr} \Rmr ) \rar & \Sigma^{i t \Vmr} \Rmr
\end{tikzcd}
\end{equation}
for each  $\sfi, \sft \in \NN$,  such that $\uptheta_{1, \sft}^\Rmr$ is equivalent to the identity map $1_{\Sigma^{\sft\Vmr}\Rmr}$,    and the diagrams
 \begin{equation}
\begin{tikzcd} \label{HinftyV1}
\Dmr_{\sfi}^{\Ocal}(\Sigma^{t\Vmr}\Rmr) \sma \Dmr_{\sfj}^{\Ocal}(\Sigma^{t\Vmr}\Rmr)  \ar[rr, "\alpha_{\sfi,\sfj}"] \ar[dd,"\uptheta_{\sfi, \sft}^\Rmr \sma \uptheta_{\sfj, \sft}^\Rmr"'] && \Dmr_{\sf i +j}^{\Ocal}(\Sigma^{t \Vmr}\Rmr)  \ar[dd, "\uptheta_{\sfi + \sfj, \sft}^\Rmr"] \\
\\
\Sigma^{it\Vmr}\Rmr \sma \Sigma^{jt\Vmr} \Rmr \ar[rr, "\upmu"'] & & \Sigma^{(i + j)t}\Rmr 
\end{tikzcd}
\end{equation}
\begin{equation} \label{HinftyV2}
\begin{tikzcd}
 \Dmr^{\Ocal}_{\sfi} \left( \Dmr^{\Ocal}_{\sfj}(\Sigma^{t\Vmr}\Rmr) \right) \dar["\Dmr^{\Ocal}_{\sfi}(\uptheta_{\sfj, \sft}^\Rmr)"'] \rar["\upbeta_{\sfi, \sfj}"] & \Dmr^{\Ocal}_{\sfi \sfj}(\Sigma^{t \Vmr}\Rmr) \dar["\uptheta^{\Rmr}_{\sfi \sfj, \sft}"] \\
\Dmr^{\Ocal}_{\sfi}(\Sigma^{jt\Vmr}\Rmr) \rar["\uptheta^{\Rmr}_{\sfi, \sfj\sft}"'] & \Sigma^{i j t \Vmr} \Rmr
\end{tikzcd}
\end{equation}
\begin{equation} \label{HinftyV3}
\begin{tikzcd}
\Dmr_{\sfi}^{\Ocal}(\Sigma^{s\Vmr}\Rmr \sma \Sigma^{t\Vmr}\Rmr) \ar[rr,"\Dmr_{\sfi}^\Ocal(\upmu)"] \dar && \Dmr_{\sfi}^\Ocal(\Sigma^{(s+ t)\Vmr}\Rmr) \ar[dd, "\uptheta^{\Rmr}_{\sfi, \sfs + \sft}"]\\
\Dmr_{\sfi}^{\Ocal}(\Sigma^{s\Vmr}\Rmr) \sma \Dmr^{\Ocal}_{\sfi}(\Sigma^{t\Vmr}\Rmr)\dar["\uptheta^\Rmr_{\sfi, \sfs} \sma \uptheta^\Rmr_{\sfi, \sft}"']  \\
\Sigma^{is  \Vmr} \Rmr \sma \Sigma^{i t \Vmr} \Rmr  \ar[rr, "\upmu^\Rmr"'] && \Sigma^{i(s+t)\Vmr} \Rmr
\end{tikzcd}
\end{equation}
commute for all $i, j ,s, t \in \NN$ in $\Ho(\Sp_\Gmr)$. 
\end{defn}
\begin{rmk} \label{notn:Hinfty-d} When $\Ocal$ is a nonequivariant $\EE_\infty$-operad and $\Vmr = \RR^{\sfd}$ then \Cref{defn:HinftyV} recovers the classical definition of $\HH_\infty^{\sfd}$-rings \cite[I.4.1]{Hinfty}. 
\end{rmk}
\begin{rmk} When $\sfi =0$, we have $\Dmr_{\sfi}^{\Ocal}(\Sigma^{t\Vmr} \Rmr ) \simeq \SS$ and the map $\uptheta_{\sfi, \sft}^\Rmr$ of \eqref{map:stHV} is the same map for all $t \in \NN$. This map serves as the unit map of $\Rmr$ which we will denoted by $\uptheta_{0}^{\Rmr}$.
\end{rmk}
\begin{rmk} The diagrams \eqref{HinftyV1} and \eqref{HinftyV2} restricted to $t =0$ implies that any $\Ocal^{\sfh}_{[\Vmr]}$-ring spectrum is automatically an $\Ocal^{\sfh}$-ring spectrum. 
\end{rmk}
\begin{notn}
 For brevity, we set $\upgamma^{(i)}_{\Vmr} \coloneqq \upgamma_{\Vmr}^{\Fscr_i(\Ocal)}$,  $\td{\upgamma}^{(i)}_{\Vmr} \coloneqq \td{\upgamma}_{\Vmr}^{\Fscr_i(\Ocal)}$ and $\partialup_{\sfi} = \partialup_{\sfi}^{\Fscr_i(\Ocal)}$ with $\Fscr_i(\Ocal)$ as given in \Cref{notn:Ofam}.
\end{notn}

Using the unit map of $\Rmr$, we get a map 
\begin{equation} \label{HinftyThom}
\begin{tikzcd}
{\bf u}_{t \Vmr}^{(i)}: \Th(\upgamma_{t\Vmr}^{(i)}) \simeq \Dmr_{\sfi}^{\Ocal}(\Smr^{t \Vmr}) \rar & \Dmr_{\sfi}^{\Ocal}(\Sigma^{t \Vmr}\Rmr) \rar & \Sigma^{i t \Vmr}\Rmr 
\end{tikzcd}
\end{equation}
which serves as an $\Rmr$-Thom class of the bundle $\upgamma_{t\Vmr}^{\Fscr_i}$. Since $\upgamma_{t\Vmr}^{(i)} \cong \td{\upgamma}_{t\Vmr}^{(i)} \oplus \upepsilon_{t\Vmr}$, we also get an $\Rmr$-Thom class 
\[ \td{\bf u}_{t \Vmr}^{(i)} \in \Rmr^{(\sfi -1)t \Vmr}\Th(\td{\upgamma}_{t\Vmr}^{(i)}). \] 
 We arrange these Thom classes so that 
\begin{equation} \label{eqn:tdu}
 {\bf u}_{t \Vmr}^{(i)} = \upsigma_{t \Vmr}(\td{\bf u}_{t \Vmr}^{(i)})
 \end{equation}
  for all $i \geq 2$ and $t \in \NN$. 

The diagrams in \eqref{HinftyV1}, \eqref{HinftyV2} and \eqref{HinftyV3} forces certain relationships between these $\Rmr$-Thom classes defined in \eqref{HinftyThom}. These relations are nothing but the equivariant analogs of those in \cite[VII]{Hinfty}. 
\begin{prop} The family  $\{ {\bf u}_{t \Vmr}^{(i)}: i, \sft \in \NN \}$ of $\Rmr$-Thom classes satisfy 
\begin{equation} \label{system1}
 {\bf u}_{(s +t) \Vmr}^{(i)} = \partialup_{\sfi}^*({\bf u}_{s \Vmr}^{(i)} \times {\bf u}_{t \Vmr}^{(i)}),  
 \end{equation}
where $\partialup_{\sfi}^{\Fscr_i}$ is the natural map defined in \Cref{eqn:extendedoplax}. 
\end{prop}
\begin{proof} The result follows from the diagram 
\[ 
\begin{tikzcd}
\Dmr_{\sfi}^{\Ocal}(\Smr^{(s+ t)\Vmr}) \ar[r, "{(\bf A)}"] \dar["\partialup_{\sfi}"'] & \Dmr_{\sfi}^{\Ocal}(\Sigma^{s\Vmr}\Rmr \sma \Sigma^{t\Vmr}\Rmr) \ar[rr,"\Dmr_{\sfi}^\Ocal(\upmu)"] \dar["\partialup_{\sfi}"'] && \Dmr_{\sfi}^\Ocal(\Sigma^{(s+ t)\Vmr}\Rmr) \ar[d, "\uptheta^{\Rmr}_{\sfi, \sfs + \sft}"]\\
\Dmr_{\sfi}^{\Ocal}(\Sigma^{s\Vmr}) \sma \Dmr^{\Ocal}_{\sfi}(\Sigma^{t\Vmr}) \rar["{(\bf B)}"'] & \Dmr_{\sfi}^{\Ocal}(\Sigma^{s\Vmr}\Rmr) \sma \Dmr^{\Ocal}_{\sfi}(\Sigma^{t\Vmr}\Rmr)\ar[rr, "\upmu^\Rmr  (\uptheta^\Rmr_{\sfi, \sfs} \sma \uptheta^\Rmr_{\sfi, \sft})"']   && \Sigma^{(s + t)i\Vmr} \Rmr
\end{tikzcd}
\]
where ${\bf (A)} =\Dmr_{\sfi}^{\Ocal}(e \sma e)$ and ${\bf (B)} = \Dmr_{\sfi}^{\Ocal}(e ) \sma \Dmr_{\sfi}^{\Ocal}(e )$, 
which commutes due to \eqref{HinftyV3} and the naturality of $\partialup_{\sfi}^{\Fscr_i}$. 
\end{proof}
\begin{cor} The family  $\{ \td{\bf u}_{t \Vmr}^{(i)}: i, t \in \NN \}$ of $\Rmr$-Thom classes satisfy 
\[ \td{\bf u}_{(s +t) \Vmr}^{(i)} = \partialup_{\sfi}^*(\td{\bf u}_{s \Vmr}^{(i)} \times \td{\bf u}_{t \Vmr}^{(i)}),  \]
where $\partialup_{\sfi}^{\Fscr_i}$ is the natural map defined in \Cref{eqn:extendedoplax}.
\end{cor}
Likewise, using the diagram \eqref{HinftyV1} and naturality of $\alpha_{i,j}$ we get: 
\begin{prop} The family  $\{ {\bf u}_{t \Vmr}^{(i)}: i, t \in \NN \}$ of $\Rmr$-Thom classes satisfy 
\begin{equation} \label{system2}
\alpha_{i,j}^*{\bf u}_{t \Vmr}^{(i +j)} = {\bf u}_{t \Vmr}^{(i)} \times {\bf u}_{t \Vmr}^{(j)} ,  
\end{equation}
for all $i,j, \sft \in \NN$.
\end{prop}
\begin{cor} The family  $\{ \td{\bf u}_{t \Vmr}^{(i)}: i, \sft \in \NN \}$ of $\Rmr$-Thom classes satisfy 
\[ \alpha_{i,j}^*\td{\bf u}_{t \Vmr}^{(i +j)} = \upsigma_{t \Vmr}(\td{\bf u}_{t \Vmr}^{(i)} \times \td{\bf u}_{t \Vmr}^{(j)}) ,  \]
for all $i,j, t \in \NN$.
\end{cor}
Now we discuss the consequence of \eqref{HinftyV2} which will be crucial in defining composition of Eulerian sequences. Using unit map of $\Rmr$ and  \eqref{HinftyV2}, we get the diagram 
\[ 
\begin{tikzcd}
 \Dmr^{\Ocal}_{\sfi} \left( \Dmr^{\Ocal}_{\sfj}(\Smr^{t\Vmr}) \right) \dar["\Dmr^{\Ocal}_{\sfi}({\bf u}_{t \Vmr}^{(j)})"'] \ar[rr, "\upbeta_{\sfi, \sfj}"] && \Dmr^{\Ocal}_{\sfi \sfj}(\Smr^{t \Vmr}) \ar[dd,"{\bf u}_{t \Vmr}^{(ij)}"] \\
\Dmr^{\Ocal}_{\sfi}(\Sigma^{jt\Vmr}\Rmr) \dar["\partialup_{\sfi}"']\ar[bend left, rrd, "\uptheta^{\Rmr}_{\sfi, \sfj\sft}"] \\
\Dmr^{\Ocal}_{\sfi}(\Smr^{j t\Vmr}) \sma \Dmr^{\Ocal}_{\sfi} (\Rmr) \ar[r, "{\bf u}_{jt\Vmr}^{(i)} \sm \uptheta^\Rmr_{i}"'] & \Sigma^{i j t \Vmr} \Rmr  \sma \Rmr \ar[r, "\upmu^\Rmr"'] & \Sigma^{i j t \Vmr} \Rmr 
\end{tikzcd}
\] 
which implies: 
\begin{prop} \label{prop:Thomsystem3} The system  $\{ \td{\bf u}_{\sft \Vmr}^{(i)}: i, \sft \in \NN \}$ of $\Rmr$-Thom classes satisfy
\begin{equation} \label{system3}
\upmu^\Rmr  \left({\bf u}_{jt\Vmr}^{(i)} \sm \uptheta^\Rmr_{i}\right) \circ \partialup_\sfi \circ \Dmr^{\Ocal}_{\sfi}({\bf u}_{\sft \Vmr}^{(j)}) = {\bf u}_{\sft \Vmr}^{(ij)} \circ \upbeta_{i, j}
\end{equation}
for all $i, j, t \in \NN$. 
\end{prop}
\begin{cor} The family  $\{ \td{\bf u}_{\sft \Vmr}^{(i)}: i, \sft \in \NN \}$ of $\Rmr$-Thom classes satisfy 
\[ 
\upmu^\Rmr  \left(\upsigma_{jt\Vmr}(\td{\bf u}_{jt\Vmr}^{(i)}) \sm \uptheta^\Rmr_{i}\right) \circ \partialup_\sfi \circ \Dmr^{\Ocal}_{\sfi}(\upsigma_{t\Vmr}(\td{\bf u}_{t \Vmr}^{(j)})) = \upsigma_{t\Vmr}(\td{\bf u}_{t \Vmr}^{(ij)}) \circ \upbeta_{i, j}
 \]
for all $i,j, t \in \NN$.
\end{cor}

Conversely:
\begin{prop}\label{prop:consistentOrientationsToShiftedRing}
    If $\Rmr$ is an $\Ocal^\sfh$-ring spectrum such that there exists a family of $\Rmr$-Thom classes $ \{ {\bf u}_{\Vmr}^{(i)}: i\in \NN\}$
   satisfying \eqref{system2} and \eqref{system3}, then $\Rmr$ is an  $\Ocal^{\sfh}_{[\Vmr]}$-ring spectrum.  
\end{prop}
\begin{rmk} Note that ${\bf u}_{t\Vmr}^{(i)}$, for $t \geq 2$ can be defined inductively from ${\bf u}_{\Vmr}^{(i)}$ using the formula in \eqref{system1}. This will satisfy both \eqref{system2} and \eqref{system1} as long as the initial family  $ \{ {\bf u}_{\Vmr}^{(i)}: i\in \NN\}$ does. This consistency allows us to  omit   the $\Rmr$-Thom classes for $t \geq 2$ in the family of $\Rmr$-Thom classes considered in \Cref{prop:consistentOrientationsToShiftedRing}.
\end{rmk}
The proof of \Cref{prop:consistentOrientationsToShiftedRing} is identical to the arguments provided in \cite[VII.5,VII.6]{Hinfty} (in particular see \cite[VII, Proposition 6.2]{Hinfty}) which covers the nonequivariant case. Hence,  we leave the details to the readers. Also note: 
\begin{lem} \label{lem:mapshiftVring}
Suppose $\Rmr$ is an $\Ocal_{[\Vmr]}^{\sfh}$-ring and suppose have a map of $\Ocal^{\sfh}$-rings \[ f:\Rmr \longrightarrow \Smr,\] then $\Smr$ is also an $\Ocal_{[\Vmr]}^{\sfh}$-ring.  
\end{lem} 
\begin{proof} By assumption, we have a family $ \{ {\bf u}_{\Vmr}^{(i)}: i\in \NN\}$ of $\Rmr$-Thom classes   satisfying \eqref{system2} and \eqref{system3}. Since $f$ is a homotopy $\Ocal$-ring map, the family
\[ \{ f_*{\bf u}_{\Vmr}^{(i)}: i\in \NN\} \]
of $\Smr$-Thom classes also satisfies \eqref{system2} and \eqref{system3}. Therefore, $\Smr$ is  an $\Ocal_{[\Vmr]}^{\sfh}$-ring.
\end{proof}

\subsection{Norms and geometric fixed points of $\Vmr$-shifted $\EE_\infty^{\Gmr}$-rings}\

In \cite{Hinfty}, the authors demonstrate that $\Hmr\FF_p$ is an $\HH_{\infty}^d$-ring, where $d = 1$ when $p=2$ and $d =2$ when $p$ is odd. 
The primary goal of this subsection is to leverage this result by employing the norm (introduced in \cite{HHR}) and geometric fixed-point functors to identify $\Vmr$ for which  $\Hmr\ull{\FF}_p$ admits a $\Vmr$-shifted homotopy $\EE_\infty^{\Gmr}$-ring structure. 

Suppose $\Kmr$ is a subgroup of $\Gmr$. Then for a pointed $\Kmr$-space $\Xmr$ its norm is nothing but the $\Gmr$-space $\Nmr_\Kmr^\Gmr(\Xmr):= \Map_{\Hmr}(\Gmr_+,  \Xmr)$. In \cite{HHR}, the authors introduced a spectrum level norm 
\[
\begin{tikzcd}
\Nmr_\Kmr^{\Gmr}:\Sp_\Kmr\rar & \Sp_\Gmr,
\end{tikzcd}
\]
which is  
\begin{itemize}
     \item[---] strong symmetric monoidal, 
    \item[---] commutes with sifted colimits, 
    \item[---] and $\Nmr_\Kmr^{\Gmr}(\Sigma^\infty_\Kmr\Xmr)\simeq \Sigma^\infty_\Gmr(\Nmr^{
    \Gmr}_\Kmr(\Xmr)).$ 
\end{itemize}
In particular,  $\Nmr_\Kmr^\Gmr(\Sigma^{\Vmr}\SS_{\Kmr}) \simeq \Sigma^{\Ind \Vmr} \SS_\Gmr$, where $\Ind \Vmr := \RR[\Gmr] \otimes_{\RR[\Hmr]} \Vmr$ for any orthogonal $\Kmr$-representation $\Vmr$. 

The smash product serves as the coproduct in the category of $\mathbb{E}_\infty^{\Gmr}$ rings. 
The norm $\Nmr_\Kmr^{\Gmr}$ is constructed inductively from  the smash product. Since $\Nmr_\Kmr^{\Gmr}$ is strong symmetric monoidal, it lifts to a functor from the category of (homotopy) $\mathbb{E}_\infty^{\Kmr}$-rings to the homotopy category of (homotopy) $\mathbb{E}_\infty^{\Gmr}$-rings. We will now show that:
\begin{prop}\label{prop:normOfShiftedRing}
    If $\Rmr$ is a $\Vmr$-shifted homotopy $\EE_\infty^{\Kmr}$-ring spectrum, then $\Nmr_\Kmr^{\Gmr}\Rmr$ is an $\Ind \Vmr$-shifted homotopy $\EE_\infty^{\Gmr}$-ring spectrum. 
\end{prop}
 \begin{notn}  In this subsection, we simplify the notation of \eqref{notn:equiBundle} by writing  $\upgamma_{\Vmr}^{(i)}$ as an abbreviation for $\upgamma_{\Vmr}^{\All_i}$, where $\All_i$ denotes the $\Gmr \times \Sigma_i$ family defined in \Cref{notn:allFamily}.  The notation implicitly indicates that $\upgamma_{\Vmr}^{(i)}$ is a $\Gmr$-equivariant bundle, a property that follows from $\Vmr$ being a $\Gmr$-representation.
   \end{notn}

Suppose $\Rmr$ is a $\Vmr$-shifted homotopy $\EE_\infty^{\Kmr}$-ring for some orthogonal $\Kmr$-representation $\Vmr$. Then there exist a $\Rmr$-Thom class ${\bf \sfu}_{\Vmr}^{(i)}$ of the bundle $\upgamma_{\Vmr}^{(i)}$ such that the collection $ \{ {\bf u}_{\Vmr}^{(i)}: i\in \NN\}$
   satisfies \eqref{system2} and \eqref{system3}.  It is easy to see 
\[\Nmr_{\Kmr}^{\Gmr}\Th(\upgamma_\Vmr^{(i)})\cong \Th(\Nmr^{\Gmr}_\Kmr(\upgamma_{\Vmr}^{(i)})).\]
An $\Rmr$-Thom class, restricted to ($\Gmr$-orbit of) a point of the base space is a unit in the homotopy group of $\Rmr$. Since norms are  strong symmetric monoidal they preserve units. Thus:
\begin{lem}
    The class $\Nmr_\Kmr^{\Gmr}{\bf \sfu}_\Vmr^{(i)}$ is an $\Nmr_\Kmr^{\Gmr}\Rmr$-Thom class for $\Nmr_\Kmr^\Gmr( \upgamma_\Vmr^{(i)}))$. 
\end{lem}

The non-basepointed space-level restriction-coinduction adjunction yields a natural unit map
\[ 
\begin{tikzcd}
\eta:\Xmr \rar & \Map_{\Kmr}(\Gmr, \Xmr)
\end{tikzcd}
\]
for any $\Gmr$-space $\Xmr$. Because $\Gmr$ is a finite group, we have a $\Gmr$-equivariant homeomorphism 
\[
\Map_{\Kmr}(\Gmr, \Xmr)_+ \cong \Nmr_{\Kmr}^\Gmr(\Xmr_+).
\] 
Using this homeomorphism, we identify $\eta^*\Nmr_{\Kmr}^\Gmr(\upgamma_{\Vmr}^{(i)}) \cong \upgamma_{\Ind \Vmr}^{(i)}.$
 Thus,
\[ {\bf u}_{\Ind\Vmr}^{(i)} := \eta^*\Nmr_\Kmr^{\Gmr} {\bf u}_\Vmr^{(i)} \] is an $\Nmr_\sfe^{\Gmr}\Rmr$-Thom class for $\upgamma_{\Ind \Vmr}^{(i)}$.

\bigskip
\begin{proof}[\bf Proof of \Cref{prop:normOfShiftedRing}]
 Since $\{ {\bf u}_{\Vmr}^{(i)}: i \in \NN \}$ satisfy the conditions of \Cref{prop:consistentOrientationsToShiftedRing},  the family 
 \[ \{ {\bf u}_{\Ind \Vmr}^{(i)}: i \in \NN \} \]
 also satisfies those conditions (see \Cref{lem:familyInd} below). This immediately implies the result. 
 \end{proof}
 
 \bigskip
\begin{lem} \label{lem:familyInd}
    The family $\{{\bf \sfu}_{\Ind \Vmr}^{(i)}: i\in \NN\}$ of $\Nmr_\Kmr^\Gmr(\Rmr)$ satisfies \eqref{system2} and \eqref{system3}.
    \end{lem}

\begin{proof}
Consider the diagram:
    \[ \hspace{-10pt}
    \begin{tikzcd}
        \dar[dd, "{\bf \sfu}_{\Ind \Vmr}^{(i)} \wedge {\bf \sfu}_{\Ind \Vmr}^{(j)}"', bend right = 90] \Dmr_{\sfi}^{\EE_\infty^{\Gmr}}(\Smr^{n\Ind\Vmr}) \wedge \Dmr_{\sfj}^{\EE_\infty^{\Gmr}}(\Smr^{n\Ind\Vmr}) \ar[rr, "\alpha_{\sfi, \sfj}"] \ar[d, ""'] && \Dmr_{\sfi + \sfj}^{\EE_\infty^{\Gmr}}(\Smr^{n\Ind\Vmr}) \ar[d, ""] \dar[dd, "{\bf \sfu}_{\Ind \Vmr}^{(i)}", bend left = 90]\\
        \Nmr_{\Kmr}^\Gmr(\Dmr_{\sfi}^{\EE_\infty^\Kmr}\Smr^{n\Vmr}) \wedge  \Nmr_\Kmr^\Gmr(\Dmr_{\sfj}^{\EE_\infty^\Kmr}\Smr^{n\Vmr}) \ar[rr, "\Nmr_{\Kmr}^\Gmr(\alpha_{\sfi,\sfj})"] \ar[d, "\Nmr_\Kmr^{\Gmr}{\bf \sfu}_\Vmr^{(i)}\wedge \Nmr_\Kmr^{\Gmr}{\bf \sfu}_\Vmr^{(j)}"'] && \Nmr_\Kmr^\Gmr\Dmr_{\sfi+\sfj}^{\EE_\infty^\Kmr}\Smr^{n\Vmr} \ar[d, "\Nmr_\Kmr^{\Gmr}{\bf \sfu}_\Vmr^{(i+j)}"] \\
        \Sigma^{i n\Ind\Vmr} \Nmr_\Kmr^\Gmr \Rmr \wedge \Sigma^{j n\Ind\Vmr}\Nmr_\Kmr^{\Gmr}\Rmr \ar[rr, "\Nmr_\Gmr^\Gmr(\upmu_\Rmr)"] && \Sigma^{(i+ j)n \Ind \Vmr}\Nmr_\Kmr^{\Gmr}\Rmr 
    \end{tikzcd} 
    \]
   where the top vertical arrows are induced from the unit map $\eta$ for $\Emr_\Gmr \Sigma_{(-)}$.  Thus, the top square commutes due to the naturality of $\eta$. The bottom square commutes because it is the norm of \eqref{system2} for the family $\{ \sfu^{(i)}_\Vmr: i \in \NN\}$. Consequently,  the entire diagram commutes  which shows that  the family  $\{{\bf \sfu}_{\Ind \Vmr}^{(i)}: i\in \NN\}$ also satisfies the condition specified in \eqref{system2}.
    
  Next, we consider the diagram:
    \[
        \begin{tikzcd}
            &\Dmr_{\sfi}^{\EE_\infty^\Gmr}\Dmr_{\sfj}^{\EE_\infty^\Gmr} \Smr^{n\Ind \Vmr} \ar[r, "\upbeta_{\sfi, \sfj}"] \ar[d, ""] \ar[dl, "\Dmr_{\sfi}{\bf \sfu}_{\Ind\Vmr}^{(j)}"', bend right = 20] & \Dmr_{\sfi\sfj}^{\EE_\infty^{\Gmr}} \Smr^{n\Ind \Vmr} \ar[d, ""] \ar[dd, "{\bf \sfu}_{\Ind\Vmr}^{(ij)}", bend left = 90]\\
            \Dmr_{\sfi}^{\EE_\infty^{\Gmr}} \Sigma^{j n\Ind \Vmr}\Nmr_\Kmr^{\Gmr}\Rmr \ar[dr, "", bend right = 20] & \Nmr_\Kmr^{\Gmr}\Dmr_\sfi^{\EE_\infty^{\Kmr}}\Dmr_\sfj^{\EE_\infty^{\Kmr}}\Smr^{n\Vmr} \ar[r, "\upbeta_{\sfi, \sfj}"] \ar[d, "\Dmr_\sfi{\bf \sfu}_{\Vmr}^{(j)}"] & \Nmr_\Kmr^\Gmr\Dmr_{\sfi\sfj}^{\Emr_\infty^{\Kmr}} \Smr^{n\Vmr} \ar[d, "{\bf \sfu}_{\Vmr}^{(ij)}"]\\
            & \Nmr_{\Kmr}^\Gmr \Dmr_{\sfi}^{\EE_\infty^{\Kmr}} \Sigma^{j n} \Rmr \ar[r, "\Nmr_{\Kmr}^{\Gmr}(\uptheta_{\sfi, \sfj}^{\Rmr})"] & \Sigma^{i j n \Ind\Vmr} \Nmr_\Kmr^\Gmr \Rmr
        \end{tikzcd}
    \]
  where the unlabelled arrows in the top square are induced from the unit map $\eta$ for $\Emr_\Gmr \Sigma_{(-)}$.  It is straightforward to check that this diagram commutes, and therefore the family $\{{\bf \sfu}_{\Ind \Vmr}^{(i)}: i\in \NN\}$ satisfies \eqref{system3}. 
\end{proof}

Rephrasing the result  \cite[Proposition I.4.5]{Hinfty} in our language we conclude that $\Hmr\FF_p$ is an $\RR^d$-shifted homotopy $\EE_\infty$-ring, where $d=1$ when $p=2$ and $d =2$ when $p$ is odd. Using the counit of the norm-restriction adjunction, we have an $\EE_\infty^\Gmr$-ring map
\[ 
\begin{tikzcd}
\Nmr_{\sfe}^{\Gmr} \Hmr\FF_p \rar & \Hmr\ull{\FF}_p.
\end{tikzcd}
\]
Therefore, from \Cref{lem:mapshiftVring} we get: 

\begin{thm}
    The spectrum $\Hmr\ull{\FF}_p$ admits a functorial $d\uprho$-shifted homotopy $\EE_\infty^{\Gmr}$-ring structure, where $d=1$ when $p=2$ and $d =2$ when $p$ is odd. 
\end{thm}

\begin{prop} \label{prop:GfixV}
    If $\Rmr$ is a $\Vmr$-shifted homotopy $\EE_\infty^{\Gmr}$ ring, then $\Phi^{\Kmr}\Rmr$ is a $\Vmr^{\Kmr}$-shifted homotopy $\EE_\infty^{\Kmr}$-ring. 
\end{prop}

\begin{proof}
    Let $\{ {\bf \sfu}_{\Vmr}^{(i)}: i \in \NN \}$ denote the family $\Rmr$-Thom classes satisfying  \eqref{system2} and \eqref{system3} that determine  the $\Vmr$-shifted homotopy ring structure of $\Rmr$. Define ${\bf \sfu}_{\Vmr^\Kmr}^{(i)}$ to be the composite
    \[
  {\bf \sfu}_{\Vmr^\Kmr}^{(i)} : \begin{tikzcd}
        \Dmr_\sfi^{\EE_\infty^{\Kmr}}(\Smr^{\Vmr^{\Kmr}}) \ar[r, "\wh{\uplambda}"] & \Phi^{\Kmr}\Dmr_\sfi^{\EE_\infty^{\Gmr}}(\Smr^{\Vmr}) \ar[rr, "\varphi^\Kmr({\bf \sfu}_{\Vmr}^{(i)})"] && \Phi^{\Kmr}\Sigma^{\sfi\Vmr} \Rmr \ar[r, cong] & \Sigma^{\sfi\Vmr^{\Kmr}}\Phi^{\Kmr}\Rmr
    \end{tikzcd},
    \]
    where $\wh{\uplambda}$ is the natural map  defined in \eqref{com:extended}. It is easy to see ${\bf \sfu}_{\Vmr^\Kmr}^{(i)}$ is a $\Phi^{\Kmr}\Rmr$-Thom class of $\upgamma_{\Vmr^\Kmr}^{(i)}$. A straightforward diagram chase shows that the family $\{   {\bf \sfu}_{\Vmr^\Kmr}^{(i)}: i \in \NN \}$
 also satisfies \eqref{system2} and \eqref{system3}. Hence, the result.  
  \end{proof}
\begin{thm}
    When $p$ is odd, $\Hmr\ull{\FF}_p$ does not admit a $\uprho$-shifted homotopy $\EE_\infty^\Gmr$-ring structure. 
\end{thm}
\begin{proof} Assume, for the sake of contradiction, $\Hmr\ull{\FF}_p$ admits such a  structure. Then, by \Cref{prop:GfixV}, its geometric fixed points $\Phi^\Gmr \Hmr\ull{\FF}_p$ would admit an $\RR$-shifted homotopy $\EE_\infty$-ring structure. Since the zeroth Postnikov approximation provides an $\EE_\infty$-ring map
\[ 
\begin{tikzcd}
 \Phi^\Gmr \Hmr\ull{\FF}_p \rar & \Hmr\FF_p
 \end{tikzcd} \]
which implies that $\Hmr\FF_p$ itself admits an $\RR$-shifted homotopy $\EE_\infty$-ring structure (by \Cref{lem:mapshiftVring}). This, however, contradicts \cite[Proposition I.4.5]{Hinfty}.
\end{proof}

\subsection{Composition of Eulerian sequences} \ 

\begin{defn} \label{defn:intES} We say that $\upchi = (\sfx_0, \sfx_1, \dots)$ 
is an {\bf integral $\Vmr$-stable $\Rmr$-Eulerian sequence} if $\chi$ satisfies \Cref{defn:Eulerian} and  $|\sfx_i| \in \RO(\Gmr, \Vmr)$ is non-virtual for all $i \in \NN$. 
\end{defn}

Let $\Ecal^{(n)}_{\Rmr, \Vmr}$ denote the collection of integral $\Vmr$-stable $\Rmr$-Eulerian sequence of weight $n$ for a $\Vmr$-shifted homotopy $\Ocal$-ring $\Rmr$, where $\Ocal$ is an $\Ncal_\infty$-operad. Our goal is to define a strictly associative pairing
\[ 
\begin{tikzcd}
\odot: \Ecal^{(n)}_{\Rmr, \Vmr} \times \Ecal^{(m)}_{\Rmr, \Vmr} \rar & \Ecal^{(nm)}_{\Rmr, \Vmr}
\end{tikzcd}
\]
and prove \Cref{thm:composeES}. 

Since $\Rmr$ is an $\Ocal^{\sfh}_{[\Vmr]}$-ring, $\td{\upgamma}_{\Vmr}^{(i)}$ admits an $\Rmr$-Thom class $\td{\bf u}_{\Vmr}^{(i)}$ for all $i \in \NN$ such that the family  $\{ {\bf u}_\Vmr^{(i)}: i \in \NN\}$ satisfies the conditions of \eqref{system2} and \eqref{system3}. Let $\td{\bf e}_\Vmr^{(i)}$ denote the corresponding Euler class. 

\begin{notn} We  introduce a slight abuse of notation for a non-basepointed $\Gmr$-space $\Xmr$, defining its $\Fscr$-th extended power as:
\[ \Dmr^{\Fscr}_{\sfn}(\Xmr) := \Emr\Fscr \times_{\Sigma_n} \Xmr^{\times \sfn} \]
This notation is consistent with the basepointed case, as there is a canonical homeomorphism $\Dmr^{\Fscr}_{\sfn}(\Xmr_+) \cong \Dmr^{\Fscr}_{\sfn}(\Xmr)_+$. 
\end{notn}

We now consider the standard inclusion of the wreath product 
\[ \begin{tikzcd}
i: \Sigma_m\wr \Sigma_n \cong \Sigma_m^{\times n} \rtimes \Sigma_n \rar[hook] & \Sigma_{mn}.
\end{tikzcd}
\]
The pullback of $\uptau_{mn}$ along this inclusion is isomorphic to the representation \[ \uptau_m^{\times \sfn} := \Map(\sfn, \uptau_m), \] where $\Sigma_n$ acts on $\sfn= \{ 1, \dots, n\}$ by permutation. The action of  the element  $((s_1, \dots, s_n), \sigma^{-1}) \in \Sigma_{m} \wr \Sigma_n$  on $\uptau_{m}^{\times \sfn}$ is given by the  formula  
\[  
((s_1, \dots, s_n), \sigma) \cdot (v_1, \dots ,v_{n}) = (s_{\sigma(1)} \cdot v_{\sigma(1)}, \dots,  s_{\sigma(n)} \cdot v_{\sigma(n)} ).
\] 
Thus, the pullback of $\td{\uptau}_{mn}$ along $i$ decomposes into a direct sum of the form  
\[ 
i^* \td{\uptau}_{mn} \cong \wh{\uptau}_n \oplus \td{\uptau}_m^{\times \sfn}, 
\] 
  where $\wh{\uptau}_{n}$ is the pullback of $\td{\uptau}_n$ along the quotient map $\Sigma_m \wr \Sigma_n \twoheadrightarrow \Sigma_n$.
\begin{notn} Let $\Bscr_n$ denote the $\Gmr$-space $\Bmr \Fscr_n(\Ocal) = \Ocal(n) \times_{\Sigma_n} \ast$.
\end{notn}  
The map $i$ also results in the map 
\begin{equation} \label{map:Biota}
\begin{tikzcd}
\iota: \Dmr^{\Ocal}_{\sfn}(\Bscr_{m}) \cong \Ocal(n) \times_{\Sigma_n} (\Bscr_{m})^{\times n} \rar &   \Bscr_{mn}. 
 \end{tikzcd}
 \end{equation}
The pullback of  $\upgamma_{\Vmr}^{(mn)}$ along $\iota$ is the bundle 
 \begin{equation*}
  \begin{tikzcd}
\left(\Ocal(n)\times (\Ocal(m)^{\times n}\right)\times_{\Sigma_m\wr \Sigma_n}\left((\Vmr\otimes(   \uptau_m^{\times \sfn} )\right) \dar \\
  \Dmr^{\Ocal}_{\sfn}(\Bscr_{m}). 
    \end{tikzcd}
\end{equation*}
The above bundle is isomorphic to $\Dmr_\sfn^\Ocal(\upgamma^{(m)}_\Vmr)$, which is obtained by applying $\Dmr_{\sfn}^{\Ocal}(-)$ to both the total space and the base space of $\upgamma^{(m)}_\Vmr$. Likewise, 
\begin{equation*}
 \iota^*\td{\upgamma}^{(mn)}_\Vmr \cong   \begin{tikzcd}
\left(\Ocal(n)\times (\Ocal(m)^{\times n}\right)\times_{\Sigma_m\wr \Sigma_n}\left((\Vmr\otimes(  \wh{\uptau}_n \oplus \td{\uptau}_m^{\times \sfn} )\right) \dar \\
  \Dmr^{\Ocal}_{\sfn}(\Bscr_{m}). 
    \end{tikzcd}  
\end{equation*}
is isomorphic to the direct sum $q^*\td{\upgamma}_{\Vmr}^{(n)}  \oplus \Dmr_{\sfn}(\td{\upgamma}_{\Vmr}^{(m)})$, where 
\[ 
\begin{tikzcd}
q: \Dmr^{\Ocal}_{\sfn}(\Bscr_{m}) \rar &\Bscr_{\sfn}
\end{tikzcd}
\]
is map induced by the quotient $\Sigma_m \wr \Sigma_n \twoheadrightarrow \Sigma_n$. 
 This enables us to express the $\Rmr$-Euler class of $ \iota^*\td{\upgamma}^{(mn)}_\Vmr$ in terms of the $\Rmr$-Euler class of $\td{\upgamma}_{\Vmr}^{(m)}$ and $\td{\upgamma}_{\Vmr}^{(n)}$. 

Now fix two $\Rmr$-Eulerian sequence 
\[\upchi^{(1)} = (\sfx_0, \sfx_1, \sfx_2, \dots ) \in \Ecal_{\Rmr, \Vmr}^{(n)} \quad\text{and}\quad \upchi^{(2)}= (\sfy_0, \sfy_1, \sfy_2, \dots) \in \Ecal_{\Rmr, \Vmr}^{(m)}\]
of weight $m$ and $n$ respectively. 
\begin{notn} For convenience, we introduce the notations: 
    \begin{itemize}
        \item ${\bf e}_x:=\td{\bf e}^{(n)}_{\Vmr} \in \Rmr^{(n-1)\Vmr}\left((\Bscr_n)_+\right)$.
        \item ${\bf e}_y:=\td{\bf e}_{\Vmr}^{(m)} \in \Rmr^{(m-1)\Vmr}\left((\Bscr_m)_+\right)$. 
    \end{itemize}
\end{notn}
Let \[ {\bf e}_{xy} \in \Rmr^{(mn-1)\Vmr}(\Dmr_{\sfn}^{\Ocal}(\Bscr_{m})_+ ) \] denote the class represented by the following composite map:  
\begin{equation*}
 \begin{tikzcd}
   {\bf e}_{xy}:   \Dmr^{\Ocal}_{\sfn}( \Bscr_{m})_+
     \ar[r,"{\Dmr^{\Ocal}_{{{\sf n}}}({\bf e}_y)}"] &
     \Dmr^{\Ocal}_{\sfn}(\Sigma^{(m-1)\Vmr} \Rmr) \ar[d, "\partialup_{\sfn}"]\\
     &
       \Dmr^{\Ocal}_{\sfn}(\Smr^{(m-1)\Vmr}) \sm  \Dmr^{\Ocal}_{\sfn}(\Rmr ) \dar["1 \sm \uptheta_{n}^\Rmr "] \\
       & \Dmr^{\Ocal}_{\sfn}(\Smr^{(m-1)\Vmr}) \sm \Rmr 
       \ar[simeq, d,"\hspace{40pt}\Tmr_{m-1}"]  \\
       & \Sigma^{(m-1)n\Vmr}  \Bscr_{n+} \sm \Rmr \ar[r,"{\bf e}_{x}"]
       &\Sigma^{(nm-1)\Vmr} \Rmr
  \end{tikzcd}
\end{equation*}
Here $\Tmr_{k}$ is the $\Rmr$-Thom isomorphism 
\begin{equation} \label{eqn:Tm}
\Tmr_{k}: \Dmr^{\Ocal}_{\sfn}(\Smr^{k\Vmr}) \sm \Rmr \simeq \Sigma^{kn\Vmr}  \Bscr_{n+} \sm \Rmr
\end{equation}
 defined using the $\Rmr$-Thom class ${\bf u}_{k\Vmr}^{(n)}$ (via the Thom diagonal) and the identification $\Dmr^{\Ocal}_{\sfn}(\Smr^{k \Vmr}) \cong \Th(k \upgamma^{(mn)}_\Vmr)$. Note that the $\Rmr$-Euler class of the bundle $\Dmr_\sfn^{\Ocal}(\upgamma^{(m)}_\Vmr)$  and $q^{\ast}\td{\upgamma}_\Vmr^{(n)}$ are  given by the following composites, respectively: 
 \[ \begin{tikzcd}
  \Dmr^{\Ocal}_{\sfn}( \Bscr_{m})_+
     \ar[r,"{\Dmr^{\Ocal}_{{{\sf n}}}({\bf e}_y)}"] &
     \Dmr^{\Ocal}_{\sfn}(\Sigma^{(m-1)\Vmr} \Rmr) \ar[rr, "\uptheta_{n, m-1}^{\Rmr}"] && \Sigma^{n(m-1) \Vmr} \Rmr \\   
 \end{tikzcd} \]
 \[ 
 \begin{tikzcd}
  \Dmr^{\Ocal}_{\sfn}( \Bscr_{m})_+ \rar["q"] & \Bscr_{n +} \rar["{\bf e}_x"] & \Sigma^{(n-1) \Vmr} \Rmr  
 \end{tikzcd}
 \]
 The class ${\bf e}_{xy}$ is their cup product. This  is implicit in the diagram in the following lemma: 
 
 \begin{lem} \label{lem:EulerPB}  The class ${\bf e}_{xy}$ is the pullback of the $\Rmr$-Euler class  $ \td{\bf e}_{\Vmr}^{(mn)}$ of the bundle $\td{\upgamma}_{\Vmr}^{(mn)}$ along the map $\iota$ define in \eqref{map:Biota}. 
\end{lem}

\begin{proof} Let  $\td{\zeta}^{(k)}_{t \Vmr}$ and $ \zeta^{(k)}_{t\Vmr}$  denote the zero sections of $t\td{\upgamma}_{ \Vmr}^{(k)}$ and $t\upgamma_{ \Vmr}^{(k)}$, respectively.  Using  the  identification $\Dmr_{\sfk}^\Ocal(\Smr^{t \Vmr}) \cong \Th(t\upgamma_{ \Vmr}^{(k)})$, we construct the following  diagram: 
\[ 
\begin{tikzcd}
\Sigma^\Vmr \Dmr_\sfn^{\Ocal}(\Bscr_{m+}) \ar[, DarkBlue, d, "\upsigma_{\Vmr}\Dmr_\sfn^\Ocal(\td{\upzeta}_\Vmr^{(m)})"'] \ar[, DarkBrown, rrr, "\iota"] &&&\Sigma^\Vmr \Bscr_{mn} \ar[d, DarkBrown, "\upsigma_{\Vmr}\left(\td{\zeta}^{(mn)}_\Vmr \right)"']\\
 \Sigma^{\Vmr}\Dmr_\sfn^{\Ocal}(\Th(\td{\upgamma}_{\Vmr}^{(m)})) \ar[DarkBlue, d,"\upsigma_{\Vmr}\Dmr_\sfn^\Ocal\left(\td{\bf u}^{(m)}_\Vmr\right)"' ] \ar[rr, "\Delta_1"] && \Dmr_\sfn^{\Ocal}(\Dmr_{\sfm}^{\Ocal}(\Smr^{\Vmr})) \ar[r, "\upbeta_{n,m}"] \ar[d, "\Dmr_{\sfn}^\Ocal \left({\bf u}^{(m)}_\Vmr\right)"]& \Dmr^\Ocal_{\sfm \sfn}(\Smr^\Vmr) \ar[DarkBrown, ddd, "{\bf u}_{\Vmr}^{(i)}"] \\ 
\Sigma^{\Vmr} \Dmr_{\sfn}^\Ocal(\Sigma^{(m-1)\Vmr}\Rmr) \ar[rr, "\Delta_2"] \ar[DarkBlue, d, "\Theta_1"']&& \Dmr_\sfn^\Ocal(\Sigma^{m\Vmr}\Rmr) \ar[d, "\Theta_2"] \\
\Sigma^{\Vmr} \Dmr_{\sfn}^\Ocal(\Smr^{(m-1)\Vmr})\sm \Rmr \ar[rr, "\Delta_3"] \ar[DarkBlue, d, "\Dmr"'] && \Dmr_\sfn^\Ocal(\Smr^{m\Vmr}) \sm  \Rmr \ar[d, "\partialup_{\sfn}"] \ar[rd, bend left, "\Umr_m"] \\
\Sigma^{\Vmr} \Bscr_{n} \sm \Dmr_{\sfn}^\Ocal(\Smr^{(m-1)\Vmr})\sm \Rmr  \ar[DarkBlue, rr, " \upsigma_{\Vmr}(\td{\upzeta}^{(n)}_\Vmr) \sma 1"'] && \Dmr_\sfn^\Ocal(\Smr^{\Vmr}) \sm \Dmr_\sfn^\Ocal(\Smr^{(\sfm-1)\Vmr}) \sm  \Rmr \ar[DarkBlue, r, "\Umr_{1,m-1}"'] & \Sigma^{mn}\Rmr
\end{tikzcd}
\]
The constituent maps are defined as follows: $\Dmr$ is induced by the Thom diagonal map for $\Th((m-1)\upgamma_{ \Vmr}^{(n)}) $,  $\Theta_i$ uses $\partialup_{\sfn}$ and $\uptheta^\Rmr_\sfn: \Dmr_\sfn(\Rmr) \to \Rmr$, $\Umr_m$ uses the $\Rmr$-Thom class ${\bf u}_{m \Vmr}^{(n)}$, $\Umr_{1,m-1}$ uses ${\bf u}_{\Vmr}^{n} \times {\bf u}_{(m-1)\Vmr}^{n}$, and  $\Delta_i$ are all induced by the diagonal map 
\[
\begin{tikzcd}
\Delta_{\Vmr}: \Smr^{\Vmr} \rar & \Smr^{\Vmr} \sm \cdots \sm \Smr^{\Vmr} \cong \Smr^{n \Vmr}.
\end{tikzcd}
\] 
The naturality of these diagonal maps, along with  \eqref{eqn:Thomclassadd}, and \Cref{prop:Thomsystem3}, shows that the diagram commutes. 

From, \eqref{eqn:tdu}, and the property that the zero section composed with the Thom class yields the Euler class, we conclude that compositions of blue arrows and red arrows are ${\bf e}_{xy}$ and $\iota_*\td{\bf e}_{\Vmr}^{(mn)}$, respectively. This establishes the result. 
\end{proof}

As a necessary precursor to defining the product $\upchi^{(1)} \odot \upchi^{(2)}$, we first define the class
\[\sfx_{mk}\circ \sfy_{k}\in \Rmr_{|\sfx_{km}|+|\sfy_{k}|}^\Gmr(\Dmr^{\Ocal}_{{{\sf n}}}(\Bscr_{m})_+)\]
as the following composition: 
\begin{equation}  \label{eqn:precomp}
\begin{tikzcd}
\Smr^{|\sfx_{mk}| + |\sfy_{k}|} \ar["\upsigma_{|\sfy_k|}(\sfx_{mk})",rr] && \Sigma^{|\sfy_k|} \Bscr_{n+} \sma \Rmr \ar[d, "\Delta_k' \sm 1_\Rmr "] \\
&&\Dmr_\sfn^{\Ocal}(\Smr^{|\sfy_k|}) \sm \Rmr \ar[d, "\Dmr_\sfn^\Ocal(\sfy_k) \sm 1_\Rmr"] \\
&& \Dmr_\sfn^{\Ocal}(\Bscr_{m+} \sm \Rmr) \sm \Rmr \ar[d, "\partialup_\sfn \sm 1"'] \\
&& \Dmr_\sfn^{\Ocal}(\Bscr_{m})_+ \sm\Dmr_\sfn^{\Ocal}( \Rmr) \sm \Rmr  \ar[rr, "1 \sm \upmu^\Rmr(\uptheta^\Rmr_{\sfn} \sm 1) "'] && \Dmr_\sfn^{\Ocal}(\Bscr_{m})_+ \sm \Rmr
\end{tikzcd}
\end{equation}
\begin{rmk} \label{rmk:notvirtualneeded}
Here, we make use of the fact that  $|\sfy_{k}| = k \Vmr + |\sfy_0|$  is an isomorphism class of a non-virtual $\Gmr$-representation.  This condition allows for the existence of a diagonal map $\Smr^{|\sfy_k| } \longrightarrow \Smr^{n |\sfy_k|}$ which induces $\Delta'_k$ in the diagram above. Note that the map $\Delta'_k$ is nothing but a suspension of the zero section of the bundle $\td{\upgamma}^{(n)}_{|\sfy_k|}$. 
\end{rmk}
\begin{lem} \label{lem:exy} For all $k \in \NN$, we have: 
\begin{equation} \label{eqn:exycap}
 (\sfx_{m(k+1)} \circ \sfy_{k+1}) \frown {\bf e}_{xy} = \sfx_{mk}\circ \sfy_{k}  
 \end{equation}
\end{lem}

\begin{proof} We will prove this result by showing that the diagram in \eqref{eqn:compositionEulerian} commutes. 
In this diagram, the map labeled 
   \begin{enumerate}[ (1) \ \ ]
\item is induced by $\Delta: \Bscr_n \longrightarrow \Bscr_n \times  \Bscr_n$;
\item is the composite map:
\begin{equation*}
    \begin{tikzcd}
        & \Sigma^{|\sfy_{k+1}|} \Bscr_{n+} \sma  \Bscr_{n+} \sma  \Rmr \ar[d,equal]\\ 
        &  \Sigma^{|\sfy_{k}|} \Bscr_{n+} \sma \Sigma^{(n-1)\Vmr} \Bscr_{n+} \sma  \Rmr \ar[d,"\Delta_k' \sma \sfe_y^{m-1} \sma 1_\Rmr"] \\
        & \Sigma^{(n-1)\Vmr} \Dmr^{\Ocal}_{{\sf n}}(\Smr^{\sfy_k}) \sma
        \Sigma^{(m-1)(n-1)\Vmr} \Bscr_{n+} \sma \Rmr \sma \Rmr \ar[d,"1 \sm \upmu"] \\
        & \Sigma^{(n-1)\Vmr} \Dmr^{\Ocal}_{{\sf n}}(\Smr^{\sfy_k}) \sma
        \Sigma^{(m-1)(n-1)\Vmr} \Bscr_{n+} \sma \Rmr \ar[cong, d]\\ &
        \Sigma^{m(n-1)\Vmr} \Dmr^{\Ocal}_{{\sf n}}(\Smr^{\sfy_k}) \sma
     \Bscr_{n+} \sma \Rmr
    \end{tikzcd}
\end{equation*}
\item is $\upsigma_{(mn-1)\Vmr}(\Delta_k') \sm 1_\Rmr$;

\item is induced by $\Tmr_{m-1}$ in \Cref{eqn:Tm};
  
\item is induced by $\upmu( \sfe_y \sm1):  \Bscr_{n+} \sma \Rmr \to \Sigma^{(n-1)\Vmr}\Rmr \sma
  \Rmr \to \Sigma^{(n-1)\Vmr}\Rmr $;

\item is induced by the composition: 
\[ 
\begin{tikzcd}
\Dmr_{\sfn}^{\Ocal}(\Bscr_{n+} \sm \Rmr) \dar["\partialup_\sfn^\Rmr \sm 1_\Rmr"] \\
\Dmr_{\sfn}^\Ocal(\Bscr_{n+}  ) \sm \Dmr_\sfn^{\Ocal}(\Rmr)  \dar["\Dmr_\sfn^{\Ocal}(\Delta) \sm \uptheta_{\sfn}^\Rmr "] \\
\Dmr_{\sfn}^\Ocal(\Bscr_{n+}  \sm  \Bscr_{n+} ) \sm \Rmr  \dar["\partialup_\sfn \sm 1_\Rmr"] \\
\Dmr_{\sfn}^\Ocal(\Bscr_{n+} )  \sm  \Dmr_{\sfn}^\Ocal( \Bscr_{n+} )  \sm \Rmr \dar["1 \sm \upzeta_{(m-1)\Vmr}^{(n)} \sm 1_\Rmr"] \\ 
\Dmr_{\sfn}^\Ocal(\Bscr_{n+} ) \sm \Dmr^{\Ocal}_\sfn(\Smr^{(m-1) \Vmr}) \sm \Rmr;
\end{tikzcd}
\]  
\item is the composition of the last two maps in \eqref{eqn:precomp};
\item is the composition: 
\[ 
\begin{tikzcd}
\Dmr_{\sfn}^{\Ocal}(\Bscr_{n+} ) \sm \Rmr \dar["\Dmr_\sfn^{\Ocal}(\Delta)"] \\
\Dmr_{\sfn}^\Ocal(\Bscr_{n+}  \sm  \Bscr_{n+} )  \sm \Rmr  \dar["\partialup_\sfn  "] \\
\Dmr_{\sfn}^\Ocal(\Bscr_{n+} )  \sm  \Dmr_{\sfn}^\Ocal( \Bscr_{n+} )  \sm \Rmr \dar["1 \sm \upzeta_{(m-1)\Vmr}^{(n)} \sm 1_\Rmr"] \\ 
\Dmr_{\sfn}^\Ocal(\Bscr_{n+} ) \sm \Dmr^{\Ocal}_\sfn(\Smr^{(m-1) \Vmr}) \sm \Rmr
\end{tikzcd}
\]  
\end{enumerate}
Now observe that: 
\begin{itemize}
\item[{\bf(S1)}] commutes because $\sfy_{m(k+1)} \frown \sfe_y^{m} = \sfy_{mk}$
  (Eulerian of $\upchi_2$);
\item[{\bf(S2)}] commutes by inspection;
\item[{\bf(S3)}] commutes by the relation between Euler class and Thom class;
\item[{\bf(S4)}] commutes because $\sfe_y \smile \sfe_y^{m-1} = \sfe_y^{m}$;
\item[{\bf(S5)}] commutes because $\Dmr^{\Ocal}_{{\sf n}}(\sfx_k
  \frown \sfe_x) = \Dmr^{\Ocal}_{{\sf n}}(\sfx_{k-1})$ (Eulerian of
  $\upchi_1$);
\end{itemize}
${\bf(S6) - \bf(S10)}$ commute by naturality. The composition of the blue arrows and the red arrows yields the two sides of \eqref{eqn:exycap}, proving the result. 

\rotatebox{90}{\parbox{1.1\textheight}{
$ $ \vspace{40pt} \\
 \begin{equation} \label{eqn:compositionEulerian}
 \hspace{50pt}
 \begin{tikzcd}[ampersand replacement=\&, nodes={font=\scriptsize}]
\Smr^{|{\sf x}_{k+1}m|+|\sfy_{(k+1)}|} \ar[dd, color = DarkBlue,
"\sfx_{m(k+1)}"'] \ar[rrr, color = DarkBrown,  "="] \ar[rrrdd,
phantom, "\text{\large \bf (S1)}"] \& \& \& \Sigma^{(mn-1)\Vmr}\Smr^{|{\sf
    x}_{mk}|+|\sfy_{k}|} \ar[dd, color = DarkBrown,"\upsigma_{(mn-1)\Vmr}(\sfx_{mk})"] \\\\
  \Sigma^{|y_{k+1}|}\Bscr_{n+} \sma  \Rmr \rar["(1)"] \ar[dd,color = DarkBlue, "\Delta'_{k+1}"']  \ar[rdd, phantom, "\text{ \bf (S2)}"] \&
   \Sigma^{|y_{k+1}|}\Bscr_{n+} \sma \Bscr{n+} \sma  \Rmr  \rar["1 \sm {\bf e}_\sfy^{m} \sm 1"] \ar[dd, "(2)"']  \ar[rdd, "(3)"] \&
    \Sigma^{|y_{k+1}|} \Bscr_{n+} \sma \Sigma^{m(n-1)\Vmr} \Rmr \sma  \Rmr
    \rar["1 \sm \upmu"]   \ar[rdd, phantom, "\text{\bf (S4)}"]  \&
      \Sigma^{(mn-1)\Vmr} \Sigma^{|\sfy_{k}|} \Bscr_{n+} \sma \Rmr \ar[dd,color = DarkBrown, "\upsigma_{(mn-1)\Vmr}(\Delta_k') \sm 1_\Rmr"]   \\ \\
\Dmr^{\Ocal}_{{\sf n}}(\Smr^{|y_{k+1}|}) \sma \Rmr \ar[dd, color = DarkBlue, "\Dmr^{\Ocal}_{{\sf n}}(\sfy_{k+1}) \sm 1"'] \rar["\partialup_{\sfn} \sma 1_{\Rmr}"'] \ar[rdd, phantom, "\text{\bf (S5)}"]  \&
 \Dmr^{\Ocal}_{{\sf n}}(\Smr^{|y_{k}|}) \sma
 \Dmr^{\Ocal}_{{\sf n}}(\Smr^{(m-1)\Vmr}) \sma \Rmr
 \rar["(4)"'{name=target}] \arrow[phantom, from = uu, to = target, "{\bf (S3)}"]
 \ar[dd, "\Dmr^{\Ocal}_{{\sf  n}}(\sfy_{k}) \sm 1"'] \ar[rdd, phantom, "\text{\bf (S6)}"] 
 \&
 \Dmr^{\Ocal}_{{\sf n}}(\Smr^{|y_{k}|}) \sma
 \Sigma^{n(m-1)\Vmr} \Bscr_{n+} \sma \Rmr \rar["(5)"'] \ar[dd, "\Dmr^{\Ocal}_{{\sf n}}(\sfy_{k})\sm 1"']  \ar[rdd, phantom, "\text{\bf (S7)}"] \&
  \Sigma^{(mn-1)\Vmr}\Dmr^{\Ocal}_{{\sf n}}(\Smr^{|\sfy_{k}|}) \sma \Rmr   \ar[dd, color = DarkBrown, "\Dmr^{\Ocal}_{{\sf n}}(\sfy_{k})\sm 1"] \\ \\
\Dmr^{\Ocal}_{{\sf n}}( \Bscr_{m+} \sma \Rmr) \sma \Rmr  \ar[dd, color = DarkBlue, "(7)"'] \rar["(6)"'] \ar[rdd, phantom, "\text{\bf (S8)}"]  \&
\Dmr^{\Ocal}_{{\sf n}}( \Bscr_{m+} \sma \Rmr) \sma  \Dmr^{\Ocal}_{{\sf n}}( \Smr^{(m-1)\Vmr}) \sma  \Rmr \rar["(4)"'] \ar[dd, "(7)"'] \ar[rdd, phantom, "\text{\bf (S9)}"]  \& 
\Dmr^{\Ocal}_{{\sf n}}( \Bscr_{m+} \sma \Rmr) \sma
\Sigma^{n(m-1)\Vmr} \Bscr_{n+} \sma  \Rmr \rar["(5)"'] \ar[dd, "(7)"'] \ar[rdd, phantom, "\text{\bf (S10)}"]  \&
\Sigma^{(mn-1)\Vmr}\Dmr^{\Ocal}_{{\sf n}}( \Bscr_{m+} \sma \Rmr) \sma  \Rmr  \ar[dd, color = DarkBrown,"(7)"]      \\ \\
\Dmr^{\Ocal}_{{\sf n}}( \Bscr_{m+}) \sma \Rmr \rar[DarkBlue, "(8)"']    \&
\Dmr^{\Ocal}_{{\sf n}}( \Bscr_{m+}) \sma  \Dmr^{\Ocal}_{{\sf n}}( \Smr^{(m-1)\Vmr}) \sma  \Rmr \rar[DarkBlue, "(4)"']   \& 
\Dmr^{\Ocal}_{{\sf n}}( \Bscr_{m+}) \sma
\Sigma^{n(m-1)\Vmr} \Bscr_{n+} \sma  \Rmr \rar[DarkBlue, "(5)"']   \&
\Sigma^{(mn-1)\Vmr}\Dmr^{\Ocal}_{{\sf n}}( \Bscr_{m+} ) \sma  \Rmr 
\end{tikzcd}
\end{equation}
}}
\end{proof}
\begin{defn}[{\bf Product Law}] \label{defn:productES}
Suppose $\Vmr$ is an orthogonal $\Gmr$-representation that contains a trivial subrepresentation. We define a {\bf product operation}
\begin{equation}
\begin{tikzcd}
\odot: \Ecal^{(n)}_{\Rmr, \Vmr} \times \Ecal^{(m)}_{\Rmr, \Vmr} \rar & \Ecal^{(nm)}_{\Rmr, \Vmr}
\end{tikzcd}
\end{equation}
for $n,m \geq 2$ as follows:
Given $\Vmr$-stable $\Rmr$-Eulerian sequences  $\upchi^{(1)} =(\sfx_1, \sfx_2, \dots)  \in \Ecal^{(n)}_{\Rmr, \Vmr} $ and $\upchi^{(2)} =(\sfy_1, \sfy_2, \dots) \in \Ecal^{(m)}_{\Rmr, \Vmr}$ define 
\[ \upchi^{(1)} \odot \upchi^{(2)}  :=  (\iota_*(\sfx_{0} \circ \sfy_{0}), \iota_*(\sfx_{m} \circ \sfy_{1}), \iota_*(\sfx_{2m} \circ \sfy_{2}), \cdots  ). \] 
From \Cref{lem:EulerPB} and \Cref{lem:exy}, we conclude that $\upchi^{(1)} \odot \upchi^{(2)}$ is a $\Vmr$-stable $\Rmr$-Eulerian sequence of weight $nm$. 
\end{defn}

\begin{thm} \label{thm:composeES}
    Let $\upchi^{(1)}$ and $\upchi^{(2)}$ be $\Vmr$-stable $\Rmr$-Eulerian sequences of weight $n$ and $m$, respectively. Then, $\upchi^{(1)} \odot \upchi^{(2)}$ is an Eulerian sequence of weight $mn$ such that
    \begin{equation} \label{eqn:Steencompose}
    \Sfrak^{\upchi^{(1)}\odot \upchi^{(2)} }= \Sfrak^{\upchi^{(1)}}\circ \Sfrak^{\upchi^{(2)}}.
    \end{equation}
\end{thm}
\begin{proof} In light of \Cref{thm:stable}, it suffices to verify the equality \eqref{eqn:Steencompose} on  cohomology classes whose $\RO(\Gmr, \Vmr)$-degrees are positive multiples of $\Vmr$. 

Fix  $a \in \Rmr^{t\Vmr}(\Xmr)$ for a $\Gmr$-space $\Xmr$ and $t \in \NN$. Following  \eqref{eqn:Sqchi}, we express $\Sfrak^{\upchi^{(2)}}(a)$ as the composition:
\[ 
\begin{tikzcd}
 \Smr^{|{\sf y}_t|} \sma \Xmr_+ \rar["{\sf y}_t \sma 1"] & (\Bscr_m)_+ \sma \Xmr_+ \sma \Rmr \dar["\updelta"] \\
&\Dmr^{\Ocal}_{m}(\Xmr)_+ \sma \Rmr \rar["\Dmr_m^{\Ocal}(a) \sma 1"] 
& \Dmr_{m}^{\Ocal}(\Sigma^{t \Vmr}\Rmr) \sma \Rmr \rar["\uptheta^{\Rmr}_{m, t \Vmr}"] & \Sigma^{mt \Vmr}\Rmr.
\end{tikzcd} 
\]
Strictly speaking, this composite defines the suspension $\upsigma_{|\sfy_t|}(\Sfrak^{\upchi^{(2)}}(a))$. 

Note that the $\RO(\Gmr, \Vmr)$-grading of $\Sfrak^{\upchi^{(2)}}(x)$, which is 
\[ |\Sfrak^{\upchi^{(2)}}(a) | = m |a| - |\sfy_t|  = mt \Vmr + \| \upchi^{(2)} \| - (m-1)t\Vmr = t\Vmr + \| \upchi^{(2)} \| \] 
is not necessarily a multiple of $\Vmr$. Setting $\Wmr = |\sfy_t|$ and  applying  formula \eqref{eqn:Sqchi2} we describe the compostion $\Sfrak^{\upchi^{(1)}} \left(\Sfrak^{\upchi^{(2)}}(a)\right)$ as:
\[ 
\begin{tikzcd}
 \Smr^{|\sfx_{mt}| +|{\sf y}_t|} \sma \Xmr_+ \rar["{\sf x}_{mt} \sma 1"] & (\Bscr_n)_+ \sma \Sigma^{|\sfy_t|}\Xmr_+ \sma \Rmr \dar["\updelta"] \\
&\Dmr^{\Ocal}_{n}(\Sigma^{|\sfy_t|} \Xmr)_+ \sma \Rmr \rar["{\bf A}"] 
& \Dmr_{n}^{\Ocal}(\Sigma^{|x|}\Rmr) \sma \Rmr \rar & \Sigma^{m |x|}\Rmr
\end{tikzcd}
\]
where ${\bf A}$ is the map $\Dmr_n^{\Ocal}(\upsigma_{|\sfy_t|}(\Sfrak^{\upchi^{(2)}}(a)))$. 
Now the result will follow from the commutativity of the diagram:

\begin{equation} \label{eqn:compositionEulerian2}
\centering
\begin{tikzcd}[ nodes={font=\scriptsize}]
  \mr{S}^{|\sfx_{mt}| + |\sfy_t|} \sma \Xmr \arrow[d, "\upsigma_{|\sfy_{t}|(\sfx_{mt})}"'] \\
  \Sigma^{|\sfy_t|} \Bscr_{n+} \sma \Rmr \sma \Xmr \arrow[d, color=DarkBlue, "\updelta \sma 1_{\Rmr \sma \Xmr}"'] \ar[rd, color = DarkBrown, bend left, "{\bf (A)}"] \\
  \Dmr^{\Ocal}_n(\Smr^{|\sfy_t|}) \sma \Rmr \sma \Xmr \ar[r,"{\bf (B)}"] \dar[color=DarkBlue, "\Dmr_n^{\Ocal}(\sfy_t)\sma 1_{\Rmr \sma \Xmr}"'] & \Dmr_{n}^{\Ocal}(\Sigma^{|\sfy_{t}|} \Xmr) \sma \Rmr \dar[color= DarkBrown, "\Dmr_n^{\Ocal}(\sfy_t\sma 1_{\Rmr}) \sma 1_{\Xmr}"] \\
  \Dmr^{\Ocal}_n(\Bscr_{m +} \sma \Rmr) \sma \Rmr \sma \Xmr \ar[r, "{\bf (C)}"] \ar[dd, color=DarkBlue, "\partialup_{n} \sma 1_{\Rmr \sma \Xmr}"'] & \Dmr_{n}^\Ocal (\Bscr_{m+} \sma \Xmr \sma \Rmr) \sma \Rmr \ar[d, color= DarkBrown, "\Dmr^{\Ocal}_n(\updelta \sma 1_\Rmr) \sma \Rmr"'] \\
  & \Dmr_n^{\Ocal}(\Dmr_{m}^{\Ocal}(\Xmr) \sma \Rmr )  \sma \Rmr \ar[r, color= DarkBrown, "{\bf (U)}"] \ar[d, "\partialup_n \sma 1_\Rmr"]   &  \Dmr_n^{\Ocal}(\Dmr_m^{\Ocal}(\Sigma^{t \Vmr}\Rmr) \sma \Rmr)  \sma \Rmr \ar[d, color=DarkBrown, "1 \sma \upmu^\Rmr(\uptheta^{\Rmr}_n \sma 1_\Rmr)\circ (\partialup_n \sma 1_\Rmr)  "] \\
   \Dmr_n^{\Ocal}(\Bscr_{m+}) \sma \Dmr^{\Ocal}_n(\Rmr) \sma \Rmr \sma \Xmr  \dar[color=DarkBlue, "1 \sma \uptheta^{\Rmr}_n \sma 1_{\Rmr \sma \Xmr}"'] \ar[r,"{\bf (D)}"] & \Dmr_n^{\Ocal}(\Dmr_{m}^{\Ocal}(\Xmr)) \sma \Dmr^{\Ocal}_n(\Rmr) \sma \Rmr \ar[d, "1 \sma \uptheta_n^{\Rmr} \sma 1_\Rmr"] \ar[r, "{\bf (V)}"]  & \Dmr_n^{\Ocal}(\Dmr_m^{\Ocal}(\Sigma^{t \Vmr}\Rmr)) \sma \Rmr \dar[ color= DarkBrown, "\Dmr^{\Ocal}_n(\uptheta^\Rmr_{m, t}) \sma 1_\Rmr"] \\
   \Dmr_n^{\Ocal}(\Bscr_{m+}) \sma \Rmr \sma \Xmr \dar[color=DarkBlue, "\iota \sma 1_{\Rmr \sma \Xmr}"'] \rar["{\bf (E)}"] &\Dmr^{\Ocal}_n(\Dmr^{\Ocal}_m(\Xmr)) \sma \Rmr \sma \Rmr \rar["{\bf (W)}"] \dar["\iota \sma \upmu^\Rmr"] & \Dmr_n^{\Ocal}(\Sigma^{mt\Vmr} \Rmr) \sma \Rmr \dar[color= DarkBrown, "\upmu^\Rmr(\uptheta_{n, mt}^\Rmr \sma 1_\Rmr)"] \\ 
   \Bscr_{nm+} \sma \Rmr \sma \Xmr \rar[color=DarkBlue, "{\bf (F)}"] & \Dmr_{nm}^{\Ocal}(\Xmr) \sma \Rmr \rar[color=DarkBlue, "{\bf (X)}"] & \Sigma^{nmt}\Rmr
 \end{tikzcd}
\end{equation}
where 
\begin{itemize}
\item the maps ${\bf (A)}$ --- ${\bf (F)}$ are induced by the diagonal map of $\Xmr$ 
\item the maps ${\bf (U)}$ --- ${\bf (X)}$ are induced by the cohomology class $a \in \Rmr^{t \Vmr}(\Xmr)$ and the evident multiplicative structure of $\Rmr$.   
\end{itemize}

Indeed, the composing the top vertical arrow with the blue arrows represent and the red arrow represent the left-hand side and the right-hand side of \eqref{eqn:Steencompose}. 
\end{proof}

\begin{rmk} We may extend the definition of $\Vmr$-stable $\Rmr$-Eulerian sequence to weight $1$ using $\td{\uptau}_1 = 0$ in \eqref{redbundle}, thereby setting $\td{\upgamma}_{\Vmr}^{(1)}$ as the $0$-dimensional bundle over $\Bscr_1 \simeq \ast$.  Then the $\Rmr$-Euler class of $\td{\upgamma}_{\Vmr}^{(1)} $ is the identity element  $1 \in\pi_0^{\Gmr}(\Rmr)  \cong \Rmr^{0}(\Bscr_{1+})$. Thus, for any $a \in \pi_{\Wmr}^{\Gmr}(\Rmr)$, where $\Wmr \in \RO(\Gmr, \Vmr)$, the constant sequence 
\[ {\bf a} = (a, a, \dots ) \]
is an $\Rmr$-Eulerian sequence of weight $1$. 
\end{rmk}
\begin{notn} 
We define the $\RO(\Gmr, \Vmr)$-graded homotopy groups of $\Rmr$ as the subset of $\pi_{\star}^\Gmr (\Rmr)$ whose degrees lie in $\RO(\Gmr, \Vmr)$:
$$\MM^{\Rmr}_{\Vmr}:= \{ x \in \pi_{\star}^\Gmr (\Rmr): |x| \in \RO(\Gmr, \Vmr) \}. $$
We define $\MM^{\Rmr}_{\Vmr, +}$ as the subring of $\MM^{\Rmr}_{\Vmr}$ generated by elements in non-virtual degrees:
$$\MM^{\Rmr}_{\Vmr, +}:= \{ x \in \pi_{\star}^\Gmr (\Rmr): |x| \text{ is non-virtual in } \RO(\Gmr, \Vmr) \}. $$
\end{notn}
\begin{rmk}
We extend the product $\odot$ in \Cref{defn:productES} by allowing $n$ and $m$ to equal $1$.  Under this extended pairing, the weight $1$ Eulerian sequence 
\[ {\bf 1} = (1,1,\dots )\]
satisfies the identity condition namely ${\bf 1} \odot \upchi = \upchi = \upchi \odot {\bf 1}$. 
 It is easy to see that the point-wise sum turns $\Ecal_{\Rmr, \Vmr}^{(n)}$ into an Abelian group for each $n$. Furthermore, the product law $\odot$ endows the collection  
\[\Ecal_{\Rmr, \Vmr} = \bigsqcup_{n \geq 1} \Ecal_{\Rmr, \Vmr}^{(n)}\]  
a structure of a ring. The weight $1$ elements form a subring isomorphic to $\MM^{\Rmr}_{\Vmr, +}$, and  $\Ecal_{\Rmr, \Vmr}$ admits a structure of a left  as well as a right module over $\MM^{\Rmr}_{\Vmr, +}$. 
Moreover, by  \Cref{thm:composeES}, we get a ring homomorphism 
\[ 
\begin{tikzcd}
\Sfrak^{(-)}: \Ecal_{\Rmr, \Vmr} \rar & {[\Rmr, \Rmr]^{\Gmr}_{-\star}}.
\end{tikzcd}
\]
Thus, it is natural to ask if the left $\MM^{\Rmr}_{\Vmr}$-algebra generated by the image of $\Sfrak^{(-)}$ account for all $\Vmr$-stable $\Rmr$-cohomology operations. We expect this to be the case when $\Rmr = \Hmr\ull{\FF}_p$ and $\Vmr = \uprho_\Gmr$ as outlined in \Cref{conj:genall}. 
\end{rmk}

\section{New equivariant cohomology operations}\label{sec:newCohomologyOps} \label{sec:new}

In this section, we construct genuine stable $\Hmr\ull{\FF}_p$-cohomology operations by identifying $\uprho_{\Gmr}$-stable Eulerian sequences in $\Hmr_{\star}^{\Gmr}\left((\Bmr_\Gmr\Sigma_p)_+; \ull{\FF}_p\right)$ for all finite $\Gmr$ and primes $p$. 

To find $\Hmr\ull{\FF}_p$-Eulerian sequences, first we need to fix an $\Hmr\ull{\FF}_p$-Euler class of the $n$-fold sum of the  $\Gmr$-vector bundle (see \Cref{rmk:equivEuler})
\[
\widetilde{\upgamma}_{\Gmr, p} : = 
\begin{tikzcd}
\mr{E}_{\Gmr}\Sigma_p \times_{\Sigma_p} (\uprho_{\Gmr} \otimes \td{\uptau}_p)
\dar \\
\mr{B}_{\Gmr}\Sigma_p
\end{tikzcd}
\]
for some nonzero $n \in \NN$. 
\begin{rmk} The $\Gmr$-bundle 
$\widetilde{\upgamma}_{\Gmr, p}$ is nothing but $\widetilde{\upgamma}^{\All_p}_{\uprho_\Gmr}$ according to \Cref{notn:equiBundle}.
\end{rmk}
\begin{notn} \label{notn:Gsub} When underlying group $\Gmr$ is clear from the context, we will simply use $\widetilde{\upgamma}_{p}$ to denote  $\widetilde{\upgamma}_{ \Gmr, p}$. 
 \end{notn}

By \cite[Theorem 1.12]{BZ}, $\widetilde{\upgamma}_{\Gmr, p}$ is a homogeneous bundle, therefore, we use the equivariant first Steifel Whitney class (as in \cite[Definition~3.2]{BZ})  to detect orientation.  When $p = 2$, any homogeneous bundle is $\Hmr\ull{\FF}_2$-orientable as $\FF_2^{\times}$ is the trivial group.  In particular, $\widetilde{\upgamma}_{\Gmr, 2}$ is $\Hmr\ull{\FF}_2$-orientable.  
 
\begin{prop}\label{prop:oddPrimeOrient} When $p>2$,  $\widetilde{\upgamma}_{\Gmr, p}$ is $\Hmr\ull{\ZZ}$-orientable if and only if $\abs{\Gmr}$ is even. 
\end{prop}

\begin{proof} 
    In this proof, we subscribe to \Cref{notn:Gsub} and write $\widetilde{\upgamma}_{p}$ as  $\widetilde{\upgamma}_{\Gmr,p}$. The obstruction to $\widetilde{\upgamma}_{\Gmr, p}$ being $\Hmr\ull{\ZZ}$-orientable is the first equivariant Stiefel-Whitney class 
    \[w_1^{\Hmr\ull{\ZZ}}(\widetilde{\upgamma}_{\Gmr,p})\in \Hmr^1(\Bmr_{\Gmr}\Sigma_p;\ull{\FF}_2)\]
    according to \cite[Theorem 1.15]{BZ}. Since the restriction of $\widetilde{\upgamma}_{\Gmr, p}$ to the trivial group is $|\Gmr| \widetilde{\upgamma}_{\sfe, p}$, we conclude  
    \[ \res^\Gmr_\sfe (w_1^{\Hmr\ull{\ZZ}}(\widetilde{\upgamma}_{\Gmr,p})) = \abs{\Gmr}\,w_1^{\Hmr\ZZ}(\widetilde{\upgamma}_{\sfe, p}), \]
    which is nonzero when $|\Gmr|$ is odd. 
    
    Conversely, suppose $\abs{\Gmr}$ is even and fix $\Kmr\subset \Gmr$ of order $2$. Using \cite[Theorem 3.14]{BZ} and the fact that $\widetilde{\upgamma}_{\Gmr,p}$ is induced up from $\widetilde{\upgamma}_{\sfe, p}$ (as in \cite[Notation 3.13]{BZ}), we conclude
    \[w_1^{\Hmr\ull{\ZZ}}(\td{\upgamma}_{\Gmr, p}) = \tr_\sfe^{\Gmr}w_1^{\Hmr\ZZ}(\td{\upgamma}_{\sfe, p}) = \tr_{\Kmr}^{\Gmr}\,\tr_{\sfe}^{\Kmr}w_1^{\Hmr\ZZ}(\td{\upgamma}_{\sfe, p}).\]
It is a standard fact that $w_1^{\Hmr\ZZ}(\td{\upgamma}_{\sfe, p})$ is nonzero and can be written as ${\sf sgn}^*(\iota)$, where 
\[ 
\begin{tikzcd}
{\sf sgn}^*: \ZZ/2\langle \iota \rangle \cong \Hmr^{1} (\Bmr\Sigma_{2+}; \FF_2) \rar["\cong"] &  \Hmr^{1} (\Bmr\Sigma_{p+}; \FF_2)
\end{tikzcd}
\]
is the isomorphism induced by the sign homomorphism ${\sf sgn}:\Sigma_p \to \Sigma_2$ and $\iota$ is the generator as indicated. Thus, 
\begin{eqnarray*}
\tr_{\sfe}^{\Kmr}w_1^{\Hmr\ZZ}(\td{\upgamma}_{\sfe, p}) &=& \tr_{\sfe}^{\Kmr} {\sf sgn}^*(\iota) \\
&=& {\sf sgn}^* \tr_{\sfe}^{\Kmr} (\iota) \\
&=& 0
\end{eqnarray*}
as $\Hmr^{1}_{\Kmr}(\Bmr_{\Kmr}\Sigma_{2+}; \ull{\FF}_2) =0$ by \cite{HK}. 
\end{proof}

\begin{rmk} \label{rmk:2fold}
The  proposition above, combined with the fact that the ring $\ZZ \to \FF_p$ induce an  injection on  units  
\[ 
\begin{tikzcd}
\ZZ^{\times} \rar[hook] & \FF_p^{\times}
\end{tikzcd}
\]
when $p$ is odd,  shows that $\td{\upgamma}_{\Gmr, p}$ is not $\Hmr\ull{\FF}_p$-orientable when $p$ and $|\Gmr|$ are odd. However,  the $2$-fold sum of $\widetilde{\upgamma}_{ \Gmr, p}$ is $\Hmr\ull{\FF}_p$-orientable for all $\Gmr$ by  \cite[Theorem 1.19]{BZ}). Thus, when $p$ and $|\Gmr|$ are odd, we will work with of $2 \widetilde{\upgamma}_{ \Gmr, p}$. 
\end{rmk}

\begin{notn} \label{notn:fix}For the rest of the section we fix a finite group $\Gmr$ and a prime $p$, and  let
\begin{itemize}
\item $\epsilon = [1 + (|\Gmr | \mod 2)] (p-1)/2$, i.e., 
\[ 
\epsilon = \left\lbrace \begin{array}{ccc}
(p-1)/2 & \text{if $|\Gmr|$ is even} \\
(p-1) & \text{if $|\Gmr|$ is odd,} 
\end{array} \right.
\]
\item $\td{\bf u}_{\Gmr, 2}$ denote an $\Hmr\ull{\FF}_2$-Thom class of $\td{\upgamma}_{\Gmr, 2}$, 
\item $\td{\bf u}_{\Gmr, p}$ denote an $\Hmr\ull{\FF}_p$-Thom class of $(1 + |\Gmr| \mod 2) \td{\upgamma}_{\Gmr, p}$ when $p \neq 2$, 
\item $\td{\bf e}_{\Gmr, p}$ denote the $\Hmr\ull{\FF}_p$-Euler class corresponding to $\td{\bf u}_{\Gmr, p}$ at all prime $p$.
\end{itemize}
\end{notn}
\begin{rmk} Note that $\td{\bf e}_\Gmr \in \Hmr_{\Gmr}^{\epsilon (p-1) \uprho_\Gmr}(\Bmr_\Gmr \Sigma_p; \ull{\FF}_p)$ is a nonzero class as its image under the restriction map to trivial group is nonzero. 
\end{rmk}

At the prime $2$, we filter $\Bmr_{\Gmr}\Sigma_2$ as
\begin{equation} \label{eqn:projfilt}
\begin{tikzcd}
\ast \rar[hook] & \PP(\uprho) \rar[hook] & \PP(2 \uprho) \rar[hook] & \PP(3 \uprho) \rar[hook] & \cdots  \rar[hook] & \Bmr_{\Gmr}\Sigma_2, 
\end{tikzcd}
\end{equation}
where 
\[ \PP(-) := \Smr( - \otimes \widetilde{\uptau}_2) \times_{\Sigma_2} \ast \] 
denotes the $\Gmr$-equivariant projective space as a functor of $\Gmr$-representations. This results in
an Atiyah-Hirzebruch like spectral sequence 
\begin{equation} \label{eqn:mod2SS}
\begin{tikzcd}
\Emr^1_{\ast,\Vmr} := \overset{\infty}{\underset{n \in \NN}{\bigoplus}} \ \Hmr_{\Vmr}^{\Gmr}(\PP((n+1)\uprho)/\PP(n\uprho);\ull{\FF}_2)\ar[Rightarrow, r] &\Hmr_{\Vmr}^{\Gmr}(\Bmr_{\Gmr}\Sigma_{2+};\ull{\FF}_2)
\end{tikzcd}
\end{equation}
calculating the homology of $\Hmr_{\Vmr}^{\Gmr}(\Bmr_{\Gmr}\Sigma_{2+};\ull{\FF}_2)$. 

\begin{notn} \label{notn:kappa} Fix an injection $\kappa: \Cmr_p \hookrightarrow \Sigma_p$ and let 
\begin{equation} \label{eqn:kappa}
\begin{tikzcd}
\upkappa: \Bmr_{\Gmr}\Cmr_p \rar[hook] & \Bmr_{\Gmr}\Sigma_p
\end{tikzcd}
\end{equation}
denote the map  induced by $\upkappa$ on the classifying spaces. Note that the map $\upkappa$, up to homotopy, is independent of the choice of the injection $\kappa$. This is because all injections from $\Cmr_p$ to $\Sigma_p$ are conjugates of each other. 
\end{notn}

When $p$ is odd, we focus on identifying $\Hmr\ull{\FF}_p$-Eulerian sequence in $\Hmr_{\star}^{\Gmr}(\Bmr_{\Gmr} \Cmr_p; \ull{\FF}_p)$ defined using the Euler class $\upkappa^{*}\td{\bf e}_{\Gmr, p}$ of the bundle $\epsilon \upkappa^* \td{\upgamma}_{\Gmr,p}$.  This is because if $\upchi =({\sf x}_0, {\sf x}_1, \dots )$ is a $\Hmr\ull{\FF}_p$-Eulerian sequence in $\Hmr_{\star}^{\Gmr}(\Bmr_{\Gmr} \Cmr_p; \ull{\FF}_p)$ then 
\[ \upkappa_*\upchi = (\upkappa_* {\sf x}_0,  \upkappa_* {\sf x}_1, \dots ) \]
is a $\Hmr\ull{\FF}_p$-Eulerian sequence in $\Hmr_{\star}^{\Gmr}(\Bmr_{\Gmr} \Sigma_p; \ull{\FF}_p)$ as 
\[ 
\upkappa_*({\sf x}_{i+1})  \frown \td{\bf e}_{\Gmr, p} =  \upkappa_*( \sfx_{i+1} \frown \upkappa^* \td{\bf e}_{\Gmr, p})  =  \upkappa_*({\sf x}_i). 
\]
\begin{notn} Let $\uprho_{\Gmr, \CC}$ denote the complex regular representation of $\Gmr$. We will simply use $\uprho_\CC$ to denote $\uprho_{\Gmr, \CC}$  when the group $\Gmr$ is clear from the context. 
\end{notn}
 Similar to \eqref{eqn:projfilt}, $\Bmr_{\Gmr}\Cmr_p$ admits a filtration 
\begin{equation} \label{eqn:lensfilt}
\begin{tikzcd}
\ast \rar[hook] & \LL_p(\epsilon \uprho_\CC) \rar[hook] & \LL_p(2 \epsilon\uprho_\CC) \rar[hook] & \LL_p(3 \epsilon \uprho_\CC) \rar[hook] & \cdots  \rar[hook] & \Bmr_{\Gmr}\Cmr_p, 
\end{tikzcd}
\end{equation}
where 
\[ \LL_{p}(-) := \Smr(- \otimes_{\CC}  {\sf r}_{2 \pi/p})\times_{\Cmr_p}\ast \] 
denotes the $\Gmr$-equivariant lens space as a functor of complex $\Gmr$-representations defined using ${\sf r}_{2 \pi/p}$, the rotation by $2 \pi/p$ representation of $\Cmr_p$.  This yields an Atiyah-Hirzebruch like spectral sequence 
\begin{equation}\label{eqn:modpSS1}
\begin{tikzcd}
  \Emr^1_{\ast,\Vmr} := \overset{\infty}{\underset{n \in \NN}{\bigoplus}} \  \Hmr_{\Vmr}^{\Gmr}(\LL_p(\epsilon(n+1)\uprho_\CC)/\LL_p(\epsilon n\uprho_\CC);\ull{\FF}_p)\ar[Rightarrow, r] &\Hmr_{\Vmr}^{\Gmr}(\Bmr_{\Gmr}\Cmr_{p+};\ull{\FF}_p)\end{tikzcd}
\end{equation}
calculating the homology groups of $\Bmr_{\Gmr}\Cmr_p$. 

To identify classes in first pages of \eqref{eqn:mod2SS} and \eqref{eqn:modpSS1} we  pause briefly  to discuss the equivariant analogs of some of the classical results of Atiyah \cite{AtiyahThomCmplx} on projective spaces. 

\subsection{Some equivariant analogs of Atiyah's result} \

In 1961, Atiyah proved that the quotient $\RR\PP^{n+k}/\RR\PP^n$ is the Thom complex of the $n$-fold sum of the tautological bundle over $\RR\PP^{k}$ \cite[Proposition 4.3]{AtiyahThomCmplx}. The equivariant analog of this result takes the following form.  
 \begin{prop} \label{prop:stuntedProjThom} There exist a $\Gmr$-equivariant homeomorphism 
 \[  \Th\left(n \,\td{\upgamma}_{\Gmr, 2}^{(k)}\right) \cong \PP((n+k) \uprho)/ \PP(n \uprho) \]
  for all  $n, k \in \NN$, where $\td{\upgamma}_{\Gmr, 2}^{(k)}$ is the pullback of $\td{\upgamma}_{\Gmr, 2}$ along $\PP(k \uprho) \hookrightarrow \Bmr_{\Gmr}\Sigma_2$.
\end{prop}
\begin{proof} We first note that $\td{\upgamma}_{\Gmr,2}^{(k)}$ can be explicitly described as 
  \[ 
    \td{\upgamma}_{\Gmr, 2}^{(k)} := 
    \begin{tikzcd}
    \Smr(k \uprho \otimes \td{\uptau}_2) \times_{\Sigma_2} (\uprho \otimes \td{\uptau}_2 ) \dar  \\
    \PP(k \uprho), 
    \end{tikzcd}
    \]
where $\Smr(\Vmr)$ is the space of subspace of unit length vectors in $\Vmr$. Let 
\[ \Dmr(\Vmr) := \{ v \in \Vmr: \|v \| \leq 1 \} \] 
denote the subspace of $\Vmr$ consisting of vectors of length less than or equal to $1$. Note than when $\Vmr$ is an orthogonal representation then both $\Smr(\Vmr)$ and $\Dmr(\Vmr)$ inherits the action of the ambient group. Thus, $\Smr(k \uprho \otimes \td{\uptau}_2)$ and $\Dmr(k \uprho \otimes \td{\uptau}_2)$ are $\Gmr \times \Sigma_2$-spaces. 

 Now notice that the  $\Gmr\times \Sigma_2$-equivariant map 
    \[f:\Smr(k\uprho\otimes \td{\uptau}_2)\times \Dmr(n\uprho\otimes \td{\uptau}_2)\to \Smr((n + k)\uprho\otimes \td{\uptau}_2)\]
    defined by the formula $f(v,w) = \left(w, v\sqrt{1 - \| w \|^2 }\right)$, 
   restricts to an equivariant homeomorphism
    \[\Smr(k\uprho\otimes \td{\tau}_2)\times (\Dmr(n\uprho\otimes \td{\tau}_2) - \Smr(n\uprho\otimes \td{\tau}_2))\to \Smr((n+k)\uprho\otimes \td{\tau}_2) - \Smr(n\uprho\otimes \td{\tau}_2). \]
    Taking $\Sigma_2$-orbits, we get a $\Gmr$-equivariant homeomorphism
    \[\Dmr\left(n\td{\upgamma}_{2 }^{(k)}\right) - \Smr\left(n\td{\upgamma}_{ 2}^{(k)}\right)\to \PP((n+k)\uprho) - \PP(n\uprho),\]
where $\Dmr(-)$ and $\Smr(-)$ on the left hand side are the unit disk bundle and the unit sphere bundle functors. By one-point compactifying the map above, we get the desired homeomorphism. 
\end{proof}
 Thus, we have the following $\Hmr\ull{\FF}_2$-Thom isomorphism. 
\begin{cor} \label{cor:ProjThomiso} For any finite group $\Gmr$ \[ \Hmr_{ \Vmr + n \uprho}^\Gmr( \PP(n+1)\uprho)/\PP(n \uprho); \ull{\FF}_2 ) \cong  \Hmr_{\Vmr }^\Gmr( \PP(\uprho)_+; \ull{\FF}_2 ).\]
\end{cor}
An analogous result can be obtained for equivariant lens spaces using the bundle 
 \begin{equation} \label{eqn:bundleomega}
 \td{\upomega}_{\Gmr, p}^{(k)} := 
    \begin{tikzcd}
    \Smr(k \uprho_{\CC} \otimes_{\CC} \sfr_{2\pi/p} ) \times_{\Cmr_p} (\uprho_{\CC} \otimes_{\CC} \sfr_{2\pi/p} ) \dar  \\
    \LL_p(k \uprho), 
    \end{tikzcd}
  \end{equation}
  where $\uprho_\CC$ is the complex regular representation of $\Gmr$. 
  \begin{rmk} For a complex $\Gmr \times \Cmr_p$-representation $\Vmr$, let ${\sf u}(\Vmr)$ denote its underlying real representation. Then 
  \[{\sf u}( \uprho_{\CC} \otimes_{\CC} \sfr_{2\pi/p} ) \cong {\sf u} \left( (\uprho \otimes \CC) \otimes_{\CC} \sfr_{2\pi/p} \right) \cong  {\sf u} \left( \uprho \otimes (\CC \otimes_{\CC} \sfr_{2\pi/p}) \right) \cong \uprho \otimes {\sf u}(\sfr_{2\pi/p}).  \]  
 Thus $\td{\upomega}_{\Gmr, p}^{(k)}$ is a bundle of rank $2 \uprho$ in the sense of \cite{BZ}. 
  \end{rmk}
\begin{prop} \label{prop:stuntedLensThom} There exist a $\Gmr$-equivariant homeomorphism 
 \[  \Th\left(n \,\td{\upomega}_{\Gmr, p}^{(k)}\right) \cong \LL_p((n+k) \uprho_\CC)/ \LL_p(n \uprho_\CC) \] for all  $n, k \in \NN$, where $\td{\upomega}_{\Gmr, p}^{(k)}$ is the pullback of $\td{\upomega}_{\Gmr, p}$ as in \Cref{eqn:bundleomega}. 
\end{prop}
\begin{proof}  Notice that the  $\Gmr\times \Cmr_p$-equivariant map 
  \[ \begin{tikzcd}
  f:\Smr(k\uprho_\CC\otimes_\CC {\sf r}_{2 \pi/p})\times \Dmr(n\uprho_\CC\otimes_\CC {\sf r}_{2 \pi/p})\rar & \Smr((n + k)\uprho_\CC\otimes_\CC {\sf r}_{2 \pi/p})
  \end{tikzcd}
  \]
    defined by the formula $f(v,w) = \left(w, v\sqrt{1 - \| w \|^2 }\right)$,  
   restricts to an equivariant homeomorphism
    \[
    \begin{tikzcd}
    \Smr(k\uprho_\CC\otimes_\CC {\sf r}_{2 \pi/p})\times \left(\Dmr(n\uprho_\CC\otimes_\CC {\sf r}_{2 \pi/p}\right) - \Smr(n\uprho_\CC\otimes_\CC {\sf r}_{2 \pi/p}))\ar[d] \\
     \Smr((n+k)\uprho_\CC\otimes_\CC {\sf r}_{2 \pi/p}) - \Smr(n\uprho_\CC\otimes_\CC {\sf r}_{2 \pi/p}).
     \end{tikzcd} \]
    Taking $\Cmr_p$-orbits, we get a $\Gmr$-equivariant homeomorphism
    \[
    \begin{tikzcd}
    \Dmr\left(n\td{\upomega}_{2 }^{(k)}\right) - \Smr\left(n\td{\upomega}_{ 2}^{(k)}\right)\ar[r] & \LL((n+k)\uprho_\CC) - \LL(n\uprho_\CC). 
    \end{tikzcd}\]
By one-point compactifying the map above, we get the desired homeomorphism. 
\end{proof} 
Since $2 \td{\upomega}^{(k)}_{\Gmr, p}$ is $\Hmr\ull{\FF}_p$-orientable by $\Cref{rmk:2fold}$, we conclude: 
\begin{cor} \label{cor:stuntedLensThom} For any finite group $\Gmr$ 
  \[  \Hmr_{ \Vmr +   4n \uprho}^\Gmr \left( \LL_p( (2n+k)\uprho_\CC)/\LL_p(2n \uprho_\CC); \ull{\FF}_p \right) \cong  \Hmr_{\Vmr }^\Gmr( \LL_p( k\uprho_\CC)_+; \ull{\FF}_p ). \]
\end{cor}

Using \Cref{cor:ProjThomiso}, we may rewrite the spectral sequence \eqref{eqn:mod2SS} as
\begin{equation} \label{eqn:mmod2SS}
\begin{tikzcd}
\Emr^1_{\ast,\Vmr} := \overset{\infty}{\underset{k \in \NN}{\bigoplus}} \  \Hmr_{\Vmr}^{\Gmr}(\Sigma^{k \uprho}\PP(\uprho)_+;\ull{\FF}_2)\ar[Rightarrow, r] &\Hmr_{\Vmr}^{\Gmr}(\Bmr_{\Gmr}\Sigma_{2+};\ull{\FF}_2). 
\end{tikzcd}
\end{equation}
Likewise, we use \Cref{cor:stuntedLensThom} to rewrite \eqref{eqn:modpSS1} as 
\begin{equation} \label{eqn:mmodpSS1}
\begin{tikzcd}
 \Emr^1_{\ast,\Vmr} := \overset{\infty}{\underset{k \in \NN}{\bigoplus}}  \  \Hmr_{\Vmr}^{\Gmr}(\Sigma^{2\epsilon k \uprho}\LL_p(\epsilon \uprho_\CC)_+;\ull{\FF}_p)\ar[Rightarrow, r] &\Hmr_{\Vmr}^{\Gmr}(\Bmr_{\Gmr}\Cmr_{p+};\ull{\FF}_p)
    \end{tikzcd}
\end{equation}
when $p$ is an odd prime.

 Atiyah  showed that \cite[Lemma 4.5]{AtiyahThomCmplx} the sum of the tangent bundle $\Tmr \RR\PP^{n-1}$ of $\RR\PP^{n-1} := \PP(\RR^n)$ and a $1$-dimensional trivial bundle is isomorphic to the $n$-fold Whitney sum $n\td{\upgamma}_2^{(n)}$ of the tautological line bundle. Using this and  Atiyah duality \cite[Proposition 3.2]{AtiyahThomCmplx}, one can reproduce Poincar\'e duality
\begin{eqnarray*}
\Hmr^{k}((\RR\PP^{n-1})_+; \FF_2) &\cong& \Hmr_{-k}(\Th(-\Tmr\RR\PP^{n-1}); \FF_2) \\
&\cong& \Hmr_{-k}(\Th(-n\td{\upgamma}_2^{(n)} + \upepsilon_1); \FF_2) \\
&\cong& \Hmr_{-k}(\Sigma^{1-n}(\RR\PP^{n-1})_+;\FF_2) \\
&\cong& \Hmr_{n-1-k}((\RR\PP^{n-1})_+;\FF_2),
\end{eqnarray*}
where the third isomorphism is the evident $\Hmr\FF_2$-Thom isomorphism. 

We would now like to establish an equivariant analog of \cite[Lemma 4.5]{AtiyahThomCmplx} to establish the equivariant Poincar\'e duality results for $\PP(n \uprho)$ and $\LL_p(n \uprho)$. 
\begin{lem} \label{lem:tangentPL}  The following are isomorphisms of $\Gmr$-equivariant vector bundles:
\begin{enumerate}
\item $\Tmr \PP(k \uprho) \oplus\upepsilon_1 \cong k\td{\upgamma}^{(k)}_{2}$.
\item $\Tmr \LL_p(k\uprho) \oplus\upepsilon_1 \cong k\td{\upomega}^{(k)}_{p} $ when $p$ is odd. 
\end{enumerate}
\end{lem}
\begin{proof} Note that the normal bundle of $\SS(k \uprho \otimes \td{\uptau}_2) \hookrightarrow k \uprho \otimes \td{\uptau}_2$ is isomorphic to $\upepsilon_1$, the trivial $1$-dimensional real $(\Gmr \times \Sigma_2)$- vector bundle. Therefore, we get $(\Gmr \times \Sigma_2)$-equivariant isomorphism
    \[\Tmr \Smr(k \uprho\otimes \td{\uptau}_2)\oplus \upepsilon_1 \cong \Smr(k \uprho\otimes \td{\uptau}_2)\times (k \uprho\otimes \td{\uptau}_2)\]
of $(\Gmr \times \Sigma_2)$-vector bundle. By taking $\Sigma_2$-orbit, we get {\it (1)}. 

Similarly, the homeomorphism of {\it (2)} is the map on the $\Cmr_p$-orbits of the $(\Gmr \times \Cmr_p)$-isomorphism
\[ \Tmr \Smr(k \uprho_\CC \otimes_\CC {\sf r}_{2 \pi/p})\oplus \upepsilon_1 \cong \Smr(k \uprho_\CC \otimes_\CC {\sf r}_{2 \pi/p})\times (k \uprho_\CC \otimes_\CC {\sf r}_{2 \pi/p}), \]
where $\upepsilon_1$ is the trivial $1$ dimensional  real $(\Gmr \times \Cmr_p)$-representation. 
\end{proof}

\begin{lem}[Poincar\'e duality]\label{lem:poincareDuality} Let $\Gmr$ be a finite group. Then there are isomorphisms of Abelian groups
\begin{enumerate}
\item $\Hmr^{\star}_{\Gmr} (\PP(k\uprho)_+;\ull{\FF}_2) \cong \Hmr_{k\uprho - 1 - \star}^{\Gmr} (\PP(k\uprho)_+;\ull{\FF}_2)$
\item $\Hmr_{\star}^{\Gmr} (\LL_p(\epsilon k\uprho_\CC)_+;\ull{\FF}_p) \cong \Hmr^{2 \epsilon k \uprho - 1 - \star}_{\Gmr} (\LL_p(\epsilon k\uprho_\CC)_+;\ull{\FF}_p)$
\end{enumerate}
for all $k \in \NN$. 
\end{lem}

\begin{proof}Since the tangent bundle of $\PP(k \uprho)$ is homogeneous \cite[Lemma 2.21]{BZ}, we may combine equivariant Atiyah duality \cite[XV1$\mathsection$8]{Alaska} with \Cref{lem:tangentPL} to obtain
\begin{eqnarray*}
 \Hmr_{\Gmr}^{\star}(\PP(k\uprho)_+;\ull{\FF}_2) &\cong& \Hmr^{\Gmr}_{-\star}(\Th( -\Tmr\PP(k\uprho));\ull{\FF}_2) \\
 &\cong& \Hmr^{\Gmr}_{-\star }(\Th(-k\td{\upgamma}_{\Gmr, 2}^{(k)} + \upepsilon_1);\ull{\FF}_2)\\
&\cong &\Hmr^{\Gmr}_{k\uprho - 1 - \star}(\PP(k\uprho)_+);\ull{\FF}_2).
\end{eqnarray*}
Likewise, when $p$ is odd, we have 
\begin{eqnarray*}
 \Hmr_{\Gmr}^{\star}(\LL_p(\epsilon k\uprho_\CC)_+;\ull{\FF}_p) &\cong& \Hmr^{\Gmr}_{-\star}(\Th( -\Tmr\LL_p(\epsilon k\uprho_\CC));\ull{\FF}_p) \\
 &\cong& \Hmr^{\Gmr}_{-\star }(\Th(-\epsilon k\td{\upomega}_{\Gmr, p}^{(k)} + \upepsilon_1);\ull{\FF}_p)\\
&\cong &\Hmr^{\Gmr}_{2\epsilon k \uprho - 1 - \star}(\LL_p(\epsilon k\uprho_\CC)_+);\ull{\FF}_p),
\end{eqnarray*}
where the last isomorphism is an $\Hmr\ull{\FF}_p$-Thom isomorphism which holds because $\epsilon$ is an even number when $|\Gmr|$ and $p$ are odd (see \Cref{prop:oddPrimeOrient} and \Cref{rmk:2fold}).
\end{proof}

\subsection{Identifying $\Hmr\ull{\FF}_p$-Eulerian sequences} \ 

 In this subsection, we begin by analyzing the path components of $\PP(k\uprho)^{\Gmr}$ and $\LL_p(k\uprho_\CC)^{\Gmr}$. This will lead us to a calculation of  $\Hmr_{0}^{\Gmr}(\PP(\uprho)_+;\ull{\FF}_2)$ as well as \linebreak $\Hmr_{0}^{\Gmr}(\LL(\epsilon \uprho_\CC)_+;\ull{\FF}_p)$, and consequently, new  elements in the $\mr{E}_1$-page of  \eqref{eqn:mmod2SS} and \eqref{eqn:mmodpSS1}. We then show that these elements are nonzero permanent cycles resulting in identification of $\Hmr\ull{\FF}_p$-Eulerian sequences.

\begin{notn}  
For a finite real or complex $\Gmr$-representation $\Vmr$ and an irreducible $\Gmr$-representation $\lambda$, we let $\Vmr_\lambda$  be the subrepresentation of $\Vmr$ such that $\Vmr/\Vmr_\lambda$ does not contain an irreducible sub-representation isomorphic to $\lambda$. 
\end{notn}

The fixed-points of $\PP(k\uprho_\Gmr)$ and $\LL_p(k\uprho_\Gmr)$ have been studied and described explicitly in \cite{KL}. We simply state their conclusion. 

\begin{lem}\label{lem:GFixedPointsProjLensSpace} Let $\operatorname{Irr}_1(\Gmr)$ is the set of isomorphism classes of $1$-dimensional irreducible representations. Then 
    \[\PP(\Vmr)^{\Gmr} = \coprod_{\lambda\in \operatorname{Irr}_1(\Gmr)} \PP(\Vmr_\lambda).\]
    For $p>2$,
    \[\LL_p(\Vmr)^{\Gmr} = \coprod_{[\lambda]\in \td{\operatorname{Irr}}_1(\Gmr)}\LL_p(\Vmr_\lambda),\]
    where $\td{\operatorname{Irr}_1}(\Gmr)$ is the set of isomorphism classes of complex $1$-dimensional irreducible representations whose character factors through $\Cmr_p$. 
 \end{lem}      
    
For each $\lambda\in \operatorname{Irr}_1(\Gmr)$, we consider a based map 
\begin{equation*}
\begin{tikzcd}
b_{\lambda}:\Smr^0\rar[ "b_{\lambda}' "] & \PP(\uprho_\lambda)_+\rar[hook, "i"] & \PP(\uprho)_+,
\end{tikzcd}
\end{equation*}
where $b_{\lambda}'$ sends the non-basepoint of $\Smr^0$ to a point in $\PP(\omega_\lambda)$ and $i$ is the inclusion along the $\Gmr$-fixed points. By \Cref{lem:GFixedPointsProjLensSpace}, such a map is unique up to a contractible choice  and determines a unique class in $[b_{\lambda}] \in \pi_0^{\Gmr}(\PP(\uprho)_+)$.
\begin{rmk} \label{rmk:nonzeroHurb}
The $\Hmr\ull{\FF}_2$-Hurewicz image of $[b_{\lambda}] \in \pi_0^{\Gmr}(\PP(\uprho)_+)$, denote it by \[ \sfb_{\lambda} \in \Hmr_0^{\Gmr}(\PP(\uprho)_+; \ull{\FF}_2),\] is nonzero because its underlying nonequivariant Hurewicz image is nonzero. 
\end{rmk}
\begin{rmk} \label{rmk:diflambda}
Moreover, $\sfb_{\lambda} = \sfb_{\lambda'}$ if and only if $\lambda = \lambda'$ in $\operatorname{Irr}_1(\Gmr)$. This is because images of $\sfb_{\lambda}$ and $\sfb_{\lambda'}$ under modified geometric fixed-point functor 
\[ 
\begin{tikzcd}
\td{\varphi}^\Gmr: \Hmr^{\Gmr}_0(\PP(\uprho)_+; \ull{\FF}_2) \rar &  \bigoplus_{\lambda\in \operatorname{Irr}_1(\Gmr)}\Hmr_0(  \PP(\lambda)_+ ; \FF_2)
\end{tikzcd}
\]
are different classes when $\uplambda \neq \uplambda'$.
\end{rmk}
  The classes $\sfb_\lambda$ leads to a nonzero element
\[ \oll{\sfb}_{k \uprho, \lambda} \in \Emr^{1}_{k, k \uprho}\]
for each $k \in \NN$ in the $\Emr_1$-page of \eqref{eqn:mmod2SS}. 
\begin{rmk}[When $p$ is odd] \label{rmk:oddb}In this case, for each $\lambda \in \td{\operatorname{Irr}_1}(\Gmr)$, we consider a basepoint-preserving map: 
\[ \begin{tikzcd}
b_{\lambda}:\Smr^0\rar[hook, "b_{\lambda}' "] & \LL_p((\epsilon\uprho_{\CC})_\lambda)_+\rar["i"] & \LL_p(\epsilon\uprho_\CC)_+.
\end{tikzcd}
\]
Here, $b_{\lambda }^{'}$ sends the non-basepoint of $\Smr^0$ to a point in $\LL_p((\epsilon\uprho_{\CC})_\lambda)$. The Hurewicz image, ${\sf b}_{\lambda} \in \Hmr^{\Gmr}_0(\LL_p((\epsilon\uprho_{\CC})_\lambda)_+; \ull{\FF}_p)$, is a nonzero class because its restriction is also nonzero. This results in the element \[ \oll{\sfb}_{2\epsilon k \uprho,\lambda} \in \Emr^{1}_{k, 2\epsilon k \uprho} \] in the  \eqref{eqn:mmodpSS1} for each $k \in \NN$ and $\lambda \in \td{\operatorname{Irr}}_1(\Gmr)$.
\end{rmk}

Our next goal is to prove \Cref{prop:bperm}, where we show that $\oll{\sfb}_{k \uprho,  \lambda}$ in \eqref{eqn:mmod2SS} are nonzero permanent cycles. We will need the following two lemmas. 
\begin{lem} \label{lem:vanish} Suppose $r$ is an integer. Then 
\[
\Hmr^{\Gmr}_{r \uprho -1}(\PP(\uprho)_+; \ull{\FF}_2) \cong
\left\lbrace \begin{array}{llc}
\FF_2  & \text{when $r =1$,} \\
0 & \text{otherwise.}
\end{array} \right.
\]
Moreover, the restriction map 
\begin{equation} \label{eqn:resPrho}
\begin{tikzcd} 
\upiota_{{\sf e}\ast}: \Hmr^{\Gmr}_{ \uprho -1}(\PP(\uprho)_+; \ull{\FF}_2) \rar & \Hmr_{|\Gmr| -1}(\RR\PP^{|\Gmr| -1}; \FF_2) 
\end{tikzcd} 
\end{equation}
is an isomorphism. 
\end{lem}
\begin{proof} When $r \leq 0$, then
\begin{eqnarray*}
 \Hmr^{\Gmr}_{r \uprho -1}(\PP(\uprho)_+; \ull{\FF}_2) &\cong&  \Hmr^{\Gmr}_{r  -1}(\Sigma^{-r( \uprho -1)}\PP(\uprho)_+; \ull{\FF}_2)  \\
 &\cong& 0
 \end{eqnarray*}
as negative Bredon homology groups of any $\Gmr$-space is trivial. 

When $r >1$, we use \Cref{lem:poincareDuality} (Poincar\'e duality) to conclude that  
\begin{eqnarray*}
 \Hmr^{\Gmr}_{r \uprho -1}(\PP(\uprho)_+; \ull{\FF}_2) &\cong& \Hmr_{\Gmr}^{(1-r) \uprho}(\PP(\uprho)_+; \ull{\FF}_2) \\
 &\cong& \Hmr^{(1-r)}_{\Gmr}\left(\Sigma^{(r-1)(\uprho-1)} \PP(\uprho)_+; \FF_2 \right) \\
 &\cong& 0
 \end{eqnarray*}
 as negative Bredon cohomology of $\Gmr$-spaces are trivial. 
 
 When $r =1$, we use Poincar\'e duality and  the fact that the action of $\Gmr$ on $\Hmr \ull{\FF}_p$ is trivial to relate the $(\uprho -1)$-th Bredon homology group 
\begin{eqnarray*} 
 \Hmr^{\Gmr}_{ \uprho -1}(\PP(\uprho)_+; \ull{\FF}_2) &\cong& \Hmr^{0}_{\Gmr}\left(\Sigma^{(r-1)(\uprho-1)} \PP(\uprho)_+; \FF_2 \right) \\
 &\cong&  \Hmr^{0}\left( \mr{Orb}_\Gmr(\Sigma^{(r-1)(\uprho-1)} \PP(\uprho))_+; \FF_2 \right) 
 \end{eqnarray*}
 with the zeroth ordinary cohomology of the space of $\Gmr$-orbits of $\Sigma^{(r-1)(\uprho-1)} \PP(\uprho) $.  Since the underlying space of $\Sigma^{(r-1)(\uprho-1)} \PP(\uprho)$ is path connected, its $\Gmr$-orbits also form a path connected space, and therefore,  $ \Hmr^{\Gmr}_{ \uprho -1}(\PP(\uprho); \ull{\FF}_2) \cong \FF_2$.
 
 Note that the nonzero element $\sfc \in \Hmr^{\Gmr}_{ \uprho -1}(\PP(\uprho)_+; \ull{\FF}_2)$ is Poincar\'e dual to \[ 1 \in \Hmr_{\Gmr}^{0}(\PP(\uprho)_+; \ull{\FF}_2)\]
 whose restriction is $1 \in \Hmr_{\Gmr}^{0}( \RR\PP^{|\Gmr|-1}_+; \FF_2 )$. Since Poincar\'e duality isomorphism (PD) commutes with restriction 
\[ 
\begin{tikzcd}
\Hmr^{\Gmr}_{ \uprho -1}(\PP(\uprho)_+; \ull{\FF}_2) \rar["\upiota_{\sfe \ast}"]  \dar[ "\mr{PD}"'] & \Hmr_{|\Gmr| -1}(\RR\PP^{|\Gmr| -1}; \ull{\FF}_2) \dar["\mr{PD}"]  \\
\Hmr_{\Gmr}^{0}(\PP(\uprho)_+; \ull{\FF}_2) \rar["\cong", "\upiota_{\sfe \ast}"'] & \Hmr^{0}(\RR\PP^{|\Gmr| -1}; \ull{\FF}_2), 
\end{tikzcd}
\]
it follows that \eqref{eqn:resPrho} is also an isomorphism.  
\end{proof}
\begin{rmk} The element $\sfc \in \Hmr^{\Gmr}_{ \uprho -1}(\PP(\uprho)_+; \ull{\FF}_2)$ in the proof of the Lemma above maps to ${\bf b}_{|\Gmr|-1} \in \Hmr_{|\Gmr|-1}^\Gmr(\RR\PP^{|\Gmr| -1}; \FF_2)$ (in the notation of \Cref{ex:classicalEulerianSteen}) under restriction. 
\end{rmk}
 \begin{lem} \label{lem:vanish2}
Suppose $r$ is an integer. Then 
\[
\FF_2\{ \sfb_{\lambda} : \lambda \in \operatorname{Irr}_1(\Gmr) \} \subset \Hmr^{\Gmr}_{0 }(\PP(\uprho)_+; \ull{\FF}_2) 
\]
and $\Hmr^{\Gmr}_{r \uprho }(\PP(\uprho)_+; \ull{\FF}_2)=0 $ for $r \neq 0$. 
 \end{lem}
 \begin{proof} An argument similar to  \Cref{lem:vanish} establishes $\Hmr^{\Gmr}_{r\uprho}(\PP(\uprho)_+; \ull{\FF}_2) \cong 0$ for all integers $r \neq 0$. From \Cref{rmk:diflambda}, we conclude that  \[ \FF_2\{ \sfb_{\lambda} : \lambda \in \operatorname{Irr}_1(\Gmr) \}  \subset \Hmr^{\Gmr}_0(\PP(\uprho)_+; \ull{\FF}_2). \]
 The above inclusion may not be an isomorphism. This is a consequence of  \Cref{lem:shortExactSeqZeroHomology} and \Cref{ex:nontrivTR}. 
 \end{proof}

\begin{prop}   \label{prop:bperm} The element $\oll{\sfb}_{k \uprho,  \lambda}$ in \eqref{eqn:mmod2SS} is a nonzero permanent cycle for all $k \in \NN$. 
\end{prop}

\begin{proof} By \Cref{lem:vanish} 
\[ 
d_r(\oll{\sfb}_{k \uprho, \lambda}) \in \Emr^{r}_{k- r, k \uprho-1}\cong \Hmr_{k \uprho -1}^{\Gmr}(\Sigma^{(k-r) \uprho}\PP(\uprho)_+; \ull{\FF}_2)  
\]
is trivial for $r \geq 2$. Thus, to establish ${\sfb}_{k \uprho, \lambda}$ as a permanent cycle we must show that it is a $d_1$-cycle. 

 By \Cref{lem:vanish}, there is a unique $\FF_2$-generator 
\begin{equation} \label{eqn:c}
\oll{\sfc}_{k \uprho -1} \in \Emr^1_{k-1, k \uprho -1} 
\end{equation}
which can be a potential target of a $d_1$-differential on $\oll{\sfb}_{k \uprho, \lambda}$. In other words, 
\[d_1(\oll{\sfb}_{k \uprho, \lambda}) = t \oll{\sfc}_{k \uprho -1} \]
where $t \in \{ 0,1 \}$. We will now show that $t$ must equal $0$.  

In contrary, suppose this differential is nontrivial, i.e. $t =1$, it would imply a nontrivial differential in the spectral sequence 
\begin{equation} \label{eqn:truncmod2SS}
 \begin{tikzcd}
\Emr^1_{\ast,\Vmr} := \overset{k + r}{\underset{n =0 }{\bigoplus}} \  \Hmr_{\Vmr}^{\Gmr}(\Sigma^{n \uprho}\PP(\uprho)_+;\ull{\FF}_2)\ar[Rightarrow, r] &\Hmr_{\Vmr}^{\Gmr}(\PP((k+r) \uprho)_+;\ull{\FF}_2)
\end{tikzcd}
\end{equation}
for all $k \geq 1$, and consequently,
 \[ \Hmr^{k \uprho}_{\Gmr}(\PP((k+r) \uprho)_+;\ull{\FF}_2) \cong \Hmr_{k \uprho -1}^{\Gmr}(\PP((k+r) \uprho)_+;\ull{\FF}_2) \cong   0\] 
 by   Poincare duality (see \Cref{lem:poincareDuality}). It would follow that  $\Hmr\ull{\FF}_2$-Euler class of $k \tilde{\upgamma}_{\Gmr, 2}^{(k+r)}$ and hence its restrictions,  must be zero. However, this is a contradiction as the restriction of $k \tilde{\upgamma}_{\Gmr, 2}^{(k+r)} $ 
\[ \upiota_{{\sf e }} \left( k \tilde{\upgamma}_{\Gmr, 2}^{(k+r)} \right)  = k |\Gmr| \tilde{\upgamma}_{{\sf e}, 2}^{(k+r)}  \]
is the $k |\Gmr|$-fold sum of the tautological bundle over $\RR\PP^{( k+r) |\Gmr| -1}$, whose $\Hmr\FF_2$-Euler class is nontrivial. 

To show that $\oll{\sfb}_{k \uprho, \uplambda}$ is nonzero in the $\Emr_\infty$-page, we  consider the map of spectral sequences 
\begin{equation} \label{eqn:resSS2}
\begin{tikzcd}
 \underset{n \in \NN}{\bigoplus} \  \Hmr_{\star}^{\Gmr}(\Sigma^{n \uprho}\PP(\uprho)_+;\ull{\FF}_2)\ar[Rightarrow, r] \dar["\upiota_{\sfe *}"] &\Hmr_{\star}^{\Gmr}(\Bmr_{\Gmr}\Sigma_{2+};\ull{\FF}_2) \dar["\upiota_{\sfe *}"] \\
  \underset{n \in \NN}{\bigoplus} \  \Hmr_{\ast}(\Sigma^{n \uprho}\RR\PP^{ |\Gmr| -1}_+;\ull{\FF}_2)\ar[Rightarrow, r] &\Hmr_{\ast}(\RR\PP^\infty_+;\FF_2)
\end{tikzcd}
\end{equation}
induced by the restriction to trivial group. Note that the bottom spectral sequence collapses at the $\Emr_1$-page. Since, the image of $\oll{\sfb}_{n \uprho, \lambda}$ under restriction is a nonzero permanent cycle (follows from \Cref{rmk:nonzeroHurb}), it follows that 
 $\oll{\sfb}_{n \uprho, \lambda}$ cannot be a target of a differential, thus nonzero in the $\Emr_\infty$-page. 
\end{proof}
\begin{prop} \label{prop:cperm}  The elements $\oll{\sfc}_{k \uprho -1}$ defined in \eqref{eqn:c} is nonzero permanent cycle for all $k \geq 0$ in the spectral sequence \eqref{eqn:mmod2SS}. 
\end{prop}
\begin{proof} If $\oll{\sfc}_{k \uprho -1}$ supports or is a target of a differential in \eqref{eqn:mmod2SS} then the same will hold in the spectral sequence  \eqref{eqn:truncmod2SS} for some $r \gg 0$. Then, by  Poincar\'e duality, i.e. \Cref{lem:poincareDuality}, 
\[ \Hmr^{r \uprho}_{\Gmr}(\PP((k+r) \uprho)_+;\ull{\FF}_2) \cong \Hmr_{k \uprho -1}^{\Gmr}(\PP((k+r) \uprho)_+;\ull{\FF}_2) \cong   0\] 
which contradicts the fact that $\Hmr\ull{\FF}_2$-Euler class of $k \tilde{\upgamma}_{\Gmr, 2}^{(k+r)}$ is nonzero. 
\end{proof}

By \Cref{lem:vanish} and \Cref{lem:vanish2},  the nonzero permanent cycles $\oll{\sfb}_{k \uprho, \uplambda}$ and $\oll{\sfc}_{k \uprho -1}$   determine unique elements in $\Hmr_{\star}^{\Gmr}(\Bmr_{\Gmr}\Sigma_2; \ull{\FF}_2)$ with no indeterminacies. 
\begin{notn} \label{notn:newEuler1} When $\Gmr$ is nontrivial, let ${\bf b}_{k \uprho, \uplambda} \in  \Hmr^{\Gmr}_{k\uprho}((\Bmr_\Gmr \Sigma_2)_+; \ull{\FF}_2)$ denote the element detected by $\oll{\sfb}_{k \uprho, \uplambda}$  in \eqref{eqn:mmod2SS}. Likewise, we let ${\bf c}_{k \uprho-1} \in  \Hmr^{\Gmr}_{k\uprho-1}(\Bmr_\Gmr \Sigma_2)_+; \ull{\FF}_2)$ denote the element detected by $\oll{\sfc}_{k \uprho-1}$ of \eqref{eqn:c}. Using these elements,  we define  the sequences 
\[ \upbeta_{\lambda,(2)} := ({\bf  b}_{0, \uplambda}, {\bf b}_{\uprho, \uplambda},  {\bf b}_{2 \uprho, \uplambda}, \dots )  \]
\[ \upzeta_{(2)} := (0, {\bf c}_{\uprho-1},  {\bf c}_{2 \uprho -1}, \dots ) \]
which are candidates for  $\uprho$-stable $\Hmr\ull{\FF}_2$-Eulerian sequences.
\end{notn}

\begin{thm} \label{thm:zetaEulerian} The sequence $\upzeta_{(2)}$ is a $\uprho$-stable $\Hmr\ull{\FF}_2$-Eulerian sequence. 
\end{thm}
\begin{proof} We must verify that $\upzeta_{(2)}$ satisfies the condition of \Cref{defn:Eulerian}. Firstly note,  
\[ {\bf c}_{\uprho} \frown \td{\bf e}_{\Gmr, 2} = 0 \]
as $\mr{H}_{-1}^\Gmr (\Bmr_{\Gmr} \Sigma_2; \ull{\FF}_2) \cong 0$.  Now notice that 
\[ \upiota_{{\sf e} \ast}(\td{\bf e}_{\Gmr, 2}) =  \td{\bf e}_{{\sf e}, 2}^{|\Gmr|}\]
as the underlying nonequivariant bundle of  $\td{\upgamma}_{\Gmr, 2}$ is the $|\Gmr|$-fold sum of the tautological line bundle $\td{\upgamma}_{{\sf e}, 2}$ over $\Bmr \Sigma_2$. It follows from the arguments in the proof of \Cref{prop:bperm} that 
\[ \Hmr^{\Gmr}_{k \uprho -1}(\Bmr_\Gmr \Sigma_2; \ull{\FF}_2) \cong \FF_2  \]
generated by ${\bf c}_{k \uprho -1}$. Further, the  diagram in \eqref{eqn:resSS2} implies that the restriction map 
\[ 
\begin{tikzcd}
\upiota_{{\sf e}\ast}: \Hmr^{\Gmr}_{k \uprho -1}(\Bmr_{\Gmr} \Sigma_2; \ull{\FF}_2) \rar & \Hmr_{k |\Gmr| -1} (\Bmr\Sigma_2; \FF_2) 
\end{tikzcd}
\]
is an isomorphism for all $k \geq 1$ sending ${\bf c}_{k \uprho -1}$ to ${\bf b}_{k |\Gmr| -1}$ following the notations of \Cref{ex:classicalEulerianSteen}. Since, ${\bf b}_{(k+1) |\Gmr| -1} \frown \td{\bf e}_{{\sf e}, 2}^{|\Gmr|} = {\bf b}_{k \uprho -1}$, we conclude  that 
\[  {\bf c}_{(k+1) \uprho -1} \frown \td{\bf e}_{\Gmr, 2} = {\bf c}_{k \uprho -1}   \]
for all $k \geq 1$. Thus, $\upzeta_{(2)}$ is a $\uprho$-stable $\Hmr\ull{\FF}_2$-Eulerian sequences. 
\end{proof}
\begin{notn}  Let  $\Sq^{k \uprho_\Gmr +1} $ to denote the genuine stable cohomology operation $\Sfrak^{\upzeta_{(2)}[k]}$ corresponding to the  $k$-shift of $\upzeta_{(2)}$. 
\end{notn}
\begin{thm} The genuine stable cohomology operation 
$\Sq_\Gmr^{k \uprho  +1}$  is a nonzero element in $ [ \Hmr\ull{\FF}_2, \Hmr\ull{\FF}_2  ]_{\Gmr}^{k \uprho +1}$
 as its restriction  
\[ \upiota_{{\sf e} \ast} (\Sq^{k \uprho_\Gmr  +1} )  = \Sq^{k |\Gmr| + 1}\]
is the classical $(k |\Gmr| + 1)$-th classical mod $2$ Steenrod squaring operation. 
\end{thm}
 \begin{proof} Since $\upiota_{{\sf e} \ast}( {\bf c}_{k \uprho -1}) = {\bf b}_{k |\Gmr| -1}$ (see proof of \Cref{thm:zetaEulerian}), we get 
  \[ \upiota_{{\sf e} \ast} (\upzeta_{(2)}[k]) = \mathfrak{t}_{|\Gmr|}(\upbeta_1[k |\Gmr| +1])\] 
  and the result follows from 
   \Cref{ex:classicalEulerianSteen} and \Cref{rmk:equivEuler}. 
 \end{proof}
 In fact, we can make a much more general statement. First notice that, for all subgroups $\Kmr \subset \Gmr$, all the restriction maps involved in 
 \[ 
 \begin{tikzcd}
 \upiota_{\sfe \ast}: \Hmr^{\Gmr}_{k \uprho_\Gmr -1}(\Bmr_{\Gmr} \Sigma_2; \ull{\FF}_2) \rar["\upiota_{\Kmr \ast}"] & \Hmr^{\Kmr}_{k |\Gmr/\Kmr|\uprho_\Kmr -1}(\Bmr_{\Kmr} \Sigma_2; \ull{\FF}_2) \rar["\upiota_{\sfe\ast}^\Kmr"] & \Hmr_{k |\Gmr| -1} (\Bmr\Sigma_2; \FF_2) 
 \end{tikzcd}
 \] 
 are isomorphisms. Consequently, $\upiota_{\Kmr \ast}({\bf c}_{k \uprho_\Gmr} -1) = {\bf c}_{k |\Gmr/\Kmr| \uprho_{\Kmr} -1}$ and we get: 
 \begin{thm} \label{thm:resSq}  Suppose $\Kmr \subset \Gmr$. Then 
\[ \upiota_{\Kmr*}(\Sq^{k \uprho_\Gmr +1}(x))  = \Sq^{k |\Gmr/\Kmr|  \uprho_\Kmr + 1}( \upiota_{\Kmr \ast}(x)) \] 
for any $\Hmr\ull{\FF}_2$-cohomology class $x$. 
\end{thm}

\begin{rmk} We leave the identification of the $\Gmr$-geometric fixed point  of $\upzeta$ (as defined in \Cref{defn:geofixEulerSeq}) for the future, as our methods do not identify the elements 
\[ \sfc_* \td{\varphi}^{\Kmr}({\bf c}_{k \uprho -1}) \in \Hmr^{k-1}(\Bmr_\Kmr\Sigma_2; \FF_2),\]
where $\sfc$ is the collapse map of \eqref{eqn:collapse} and $\td{\varphi}^{\Kmr}$ is as in modified geometric fixed-point functor of \eqref{eqn:modgeofix}, in general. 
\end{rmk}

 \begin{rmk}  Whether the sequence $\upbeta_{\uplambda, (2)}$ satisfy the Eulerian criteria of \Cref{defn:Eulerian} is the subject of \Cref{subsec:Eulerianproof}. 
 \end{rmk}
 The odd primary analog of \Cref{lem:vanish} is the following lemma: 
\begin{lem} \label{lem:vanishodd} Let $\Gmr$ be any finite group and $r$ be an integer. Then 
\[
\Hmr^{\Gmr}_{2r \epsilon \uprho -1}(\LL_p(\epsilon\uprho_\CC)_+; \ull{\FF}_p) \cong
\left\lbrace \begin{array}{llc}
\FF_p  & \text{when $r =1$} \\
0 & \text{otherwise.} 
\end{array} \right. \]
Moreover, the restriction map 
\begin{equation} \label{eqn:resLrho}
\begin{tikzcd} 
\upiota_{{\sf e}\ast}: \Hmr^{\Gmr}_{ 2 \epsilon\uprho -1}(\LL_p(\epsilon\uprho_\CC)_+; \ull{\FF}_p) \rar & \Hmr_{2 \epsilon|\Gmr| -1}(\LL_p(\CC^{\epsilon|\Gmr|})_+; \ull{\FF}_2) 
\end{tikzcd} 
\end{equation}
is an isomorphism.  
\end{lem} 
\begin{proof}
The proof is identical to that of \Cref{lem:vanish}, so we leave it to the readers.
\end{proof}
From \Cref{lem:vanishodd}, we conclude $\mr{E}^1_{k, 2 k \epsilon \uprho -1} \cong \FF_p $ in \eqref{eqn:mmodpSS1}.  
\begin{lem}  
 Any nonzero element in $\mr{E}^1_{k, 2 k \epsilon \uprho -1} \cong \FF_p $ in the spectral sequence \eqref{eqn:mmodpSS1} is a nonzero permanent cycle. 
\end{lem}
\begin{proof} We consider a truncated  version \eqref{eqn:mmodpSS1} along with its restriction to trivial group 
\begin{equation} \label{eqn:resSSp}
\begin{tikzcd}
 \overset{n = k+r}{\underset{n =0}{\bigoplus}} \  \Hmr_{\star}^{\Gmr}(\Sigma^{2n \epsilon \uprho}\LL_p(\epsilon\uprho_\CC)_+;\ull{\FF}_p)\ar[Rightarrow, r] \dar["\upiota_{\sfe *}"] &\Hmr_{\star}^{\Gmr}(\LL_p((k+r) \epsilon \uprho_\CC)_+;\ull{\FF}_p) \dar["\upiota_{\sfe *}"] \\
  \overset{n = k+r}{\underset{n =0}{\bigoplus}} \  \Hmr_{\ast}(\Sigma^{k \uprho}\LL_p(\CC^{\epsilon|\Gmr|})_+;\ull{\FF}_p)\ar[Rightarrow, r] &\Hmr_{\ast}(\LL_p(\CC^{(k+r)\epsilon|\Gmr|})_+;\FF_p)
\end{tikzcd}
\end{equation}
where the bottom spectral sequence collapses at the $\mr{E}_1$-page. Thus, a nonzero element $\oll{\sfc} \in \mr{E}^1_{k, 2 k \epsilon \uprho -1}$  cannot be the target of a differential. 

Suppose $\oll{\sfc}$ supports a differential, then this differential will appear in the top spectral sequence of \eqref{eqn:resSSp}. This, along with \Cref{lem:vanishodd} and \Cref{lem:poincareDuality} (Poincar\'e duality), would imply 
\[  \Hmr^{2k \epsilon \uprho}_\Gmr(\LL_p(\epsilon (k+r)  \uprho_\CC)_+;\ull{\FF}_p ) \cong 0 \]
and hence, the $\Hmr\ull{\FF}_p$-Euler class of $\epsilon k\td{\upomega}_{\Gmr, p}^{\epsilon (k+r)}$ must be zero. However, this is a contradiction as its restriction is the  $\Hmr\FF_p$-Euler class of 
\[ \upiota_{\sfe}(\epsilon k\td{\upomega}_{\Gmr, p}^{\left( \epsilon (k+r)\right)}) = \epsilon k |\Gmr| \td{\upomega}_{\sfe, p}^{\left(\epsilon (k+r)|\Gmr|\right)},\]
which is nonzero. 
\end{proof}
The proof above reveals that not only  $\Hmr_{2 \epsilon k \uprho-1}^\Gmr \left(  (\Bmr_{\Gmr}\Cmr_p)_+; \ull{\FF}_2 \right) \cong \FF_p$, but also that the restriction map 
\begin{equation} \label{eqn:resiso}
\begin{tikzcd}
 \upiota_{\sfe *} : \Hmr_{2 \epsilon k \uprho-1}^\Gmr \left(  (\Bmr_{\Gmr}\Cmr_p)_+; \ull{\FF}_p \right) \rar & \Hmr_{2\epsilon k |\Gmr| -1}((\Bmr\Cmr_p)_+; \FF_p)
 \end{tikzcd}
  \end{equation}
is an isomorphism. The mod $p$ cohomology of $\Bmr \Cmr_p$ is isomorphic to 
\[ \Hmr^{\ast}((\Bmr\Cmr_p)_+; \FF_p) \cong \Lambda_{\FF_p}({\bf y})\llbracket {\bf u} \rrbracket,\]
where $|{\bf a}| = 1$ and ${\bf b}$ can be chosen to be $\Hmr\FF_p$-Euler class of the real $2$-dimensional bundle 
\[ 
\Rfrak_{k}:= 
\begin{tikzcd}
\Emr\Cmr_p \times_{\Cmr_p} \left(\sfr_{2\pi k /p} \right) \dar \\
\Bmr\Cmr_p
\end{tikzcd}
\]
for any $k \in \{ 1, \dots , p-1 \}$ (these Euler classes differ up to a multiple of unit). An element dual to ${\bf y}{\bf u}^{ \epsilon k |\Gmr| -1}$ generates  $ \Hmr_{2\epsilon k |\Gmr| -1}((\Bmr\Cmr_p)_+; \FF_p)$. 

Recall $\kappa$ and $\upkappa$ from \Cref{notn:kappa}. The pullback $\td{\uptau}_p$ along $\kappa$ is isomorphic to 
\[ \kappa^* \td{\uptau}_p  \cong \bigoplus_{k =1}^{(p-1)/2} \sfr_{\frac{2 k \pi}{p}}  \]
and therefore, 
\[ 
\upkappa^*\td{\upgamma}_{\sfe, p}  \cong \Rfrak_{1} \oplus \dots \oplus \Rfrak_{(p-1)/2}. 
\]
Consequently, $\Hmr\FF_p$-Euler class of $(1 + |\Gmr| \mod 2)\upkappa^*\td{\upgamma}_{\sfe, p}$ is nonzero, in fact it is a unit multiple of ${\bf u}^{\epsilon}$. Since, 
\begin{itemize}
\item  the underlying nonequivariant bundle of $\upkappa^*\td{\upgamma}_{\Gmr, p}$ is $|\Gmr|$-fold sum of  $\upkappa^*\td{\upgamma}_{\sfe, p}$, 
\item  capping with 
${\bf u}^{\epsilon}$ 
\[ 
\begin{tikzcd}
-\frown {\bf u}^{\epsilon}: \Hmr_{2 \epsilon (k+1) -1} \left( (\Bmr\Cmr_p)_+; \FF_p \right) \rar & \Hmr_{2 \epsilon k -1} \left( (\Bmr\Cmr_p)_+; \FF_p \right) 
\end{tikzcd}
\]
 is an isomorphism, 
 \item the restricted class $\iota_{\sfe *}(\td{\bf e}_{\Gmr, p})$  is a unit multiple of ${\bf u}^{\epsilon |\Gmr|}$, and 
 \item the restriction map \eqref{eqn:resiso} is an isomorphism,
 \end{itemize}
 we conclude that: 
 \begin{lem} \label{lem:isocap} For any finite group $\Gmr$, the  map 
 \[
 \begin{tikzcd}
- \frown \upkappa^*\td{\bf e}_{\Gmr, p}: \Hmr_{2 \epsilon (k+1) \uprho-1}^\Gmr \left(  (\Bmr_{\Gmr}\Cmr_p)_+; \ull{\FF}_2 \right) \rar & \Hmr_{2 \epsilon k \uprho-1}^\Gmr \left(  (\Bmr_{\Gmr}\Cmr_p)_+; \ull{\FF}_2 \right)
 \end{tikzcd}
 \] 
 is an isomorphism. 
 \end{lem}
\begin{notn} Give a finite group $\Gmr$ and a prime $p$,  choose a generator ${\bf c}_{ \uprho -1} \in \Hmr^{\Gmr}_{2 \epsilon\uprho -1} ((\Bmr_{\Gmr} \Cmr_p)_+; \ull{\FF}_p)$. Using \Cref{lem:isocap}, we define  the elements 
\[ 
{\bf c}_{k\uprho -1} \in \Hmr^{\Gmr}_{2 \epsilon k\uprho -1} ((\Bmr_{\Gmr} \Cmr_p)_+; \ull{\FF}_p)
\]
so that ${\bf c}_{(k+1) \uprho -1} \frown \upkappa^* \tilde{\bf e}_{\Gmr, p} = {\bf c}_{k \uprho -1}$.
\end{notn}
Now define  the sequence 
\[ 
\upzeta_{(p)} := (0, {\bf c}_{\uprho -1}, {\bf c}_{2\uprho -1}, \dots )
\]
which is a $\uprho$-stable $\Hmr\ull{\FF}_p$-Eulerian sequence by construction. Invoking \Cref{thm:stable}, we let $\mr{P}^{k \uprho_\Gmr +1}$ denote the genuine stable $\Hmr\ull{\FF}_p$-cohomology operation induced by  $\upzeta_{(p)}[k]$. Then:  

\begin{thm} The genuine stable cohomology operation  $\mr{P}^{2 \epsilon k \uprho_\Gmr +1}$ is a nonzero element in $ [ \Hmr\ull{\FF}_p, \Hmr\ull{\FF}_p  ]_{\Gmr}^{2 \epsilon k \uprho +1}$ as its restriction is the classical odd primary classical Steenrod operation 
\[ 
\upiota_{\sfe *}\mr{P}^{2 \epsilon k \uprho_\Gmr +1} = \beta \mr{P}^{2k |\Gmr| } 
\]
in the notation of \cite{SteenBook}.
\end{thm}
 \begin{proof} This follows from the fact that $\upiota_{\sfe *} ({\bf c}_{k \uprho -1})$ is the element dual to  ${\bf y} {\bf u}^{\epsilon k |\Gmr| -1}$ under Kronecker product,  \eqref{eqn:oddP} and \Cref{rmk:classicalsl}. 
 \end{proof}
 An argument identical to that of \Cref{thm:resSq}, shows that: 
 \begin{thm} \label{thm:resP}  Suppose $\Kmr \subset \Gmr$. Then 
\[ \upiota_{\Kmr*}(\Pmr^{2 \epsilon k \uprho_{\Gmr} +1} (x))  = \Pmr^{ 2 \epsilon|\Gmr/\Kmr| k \uprho_\Kmr + 1}( \upiota_{\Kmr \ast}(x)) \]
for any $\Hmr\ull{\FF}_p$-cohomology class $x$. 
\end{thm}

Another important consequence of \Cref{lem:vanishodd} is that the odd primary classes $\oll{\sfb}_{2\epsilon k \uprho, \lambda}$ from \Cref{rmk:oddb} are immune to all differentials:
\begin{cor} The class $\oll{\sfb}_{2\epsilon k \uprho,  \lambda}$ of \eqref{eqn:mmodpSS1} is a nonzero permanent cycle for all $k \in \NN$.
\end{cor}
\begin{proof} 
By \Cref{lem:vanish}, $\oll{\sfb}_{2\epsilon k \uprho,  \lambda}$ can only possibly support a $d_1$-differential. However, \Cref{lem:vanishodd} shows that the potential targets of any such permanent $d_1$-differential  are  already nonzero permanent cycles. Consequently, $\oll{\sfb}_{2\epsilon k \uprho,  \lambda}$ must also be permanent cycles. 

Furthermore, by  construction, as detailed in \Cref{rmk:oddb}), the restriction of $\oll{\sfb}_{2\epsilon k \uprho,  \lambda}$ in \eqref{eqn:resSSp} are nonzero permanent cycles. Therefore, $\oll{\sfb}_{2\epsilon k \uprho,  \lambda}$ cannot be the target of a differential. Hence, the result. 
\end{proof}
Arguing exactly the same way as \Cref{lem:vanish2}, we get:
 \begin{lem} \label{lem:vanishodd2}
Suppose $r$ is an integer. Then  
\[
\FF_2\{ \sfb_{\lambda} : \lambda \in \td{\operatorname{Irr}_1}(\Gmr) \}  \subset \Hmr^{\Gmr}_{0 }(\LL_p(\epsilon\uprho_\CC)_+; \ull{\FF}_2),
\]
and $\Hmr^{\Gmr}_{r \uprho }(\LL_p(\epsilon\uprho_\CC)_+; \ull{\FF}_2) =0$ when $r \neq 0$. 
 \end{lem}
 \begin{notn}
As a consequence of \Cref{lem:vanishodd2},  $\oll{\sfb}_{2\epsilon k \uprho,  \lambda}$ detects  a unique class, denote it by $\sfb_{2 \epsilon k \uprho, \lambda}$, in $\Hmr^{\Gmr}_{2 \epsilon k \uprho}((\Bmr_{\Gmr}\Cmr_p)_+; \ull{\FF}_p)$. 
 \end{notn} 
Next,  we will show that the elements in $ \Hmr^{2 \epsilon k \uprho}_\Gmr((\Bmr_{\Gmr}\Cmr_p)_+; \ull{\FF}_p)$ detected by ${\sfb}_{2\epsilon k \uprho,  \lambda}$ can be arranged to give rise to $\Hmr\ull{\FF}_{p}$-Eulerian sequences which we will denote by $\upbeta_{\lambda, (p)}$. 

\subsection{ The Eulerian sequence $\upbeta_{\lambda, (p)}$. } \label{subsec:Eulerianproof} \ 

We now establish the main technical result for this subsection: 
\begin{lem}\label{lem:shortExactSeqZeroHomology}
    Suppose $\Xmr$  path-connected space with an action of $\Gmr$. Then the kernel of the map  
    \begin{equation} \label{geomap}
    \begin{tikzcd}
    \Hmr^{\Gmr}_0(\Xmr_+;\ull{\FF}_p)\rar & \Hmr^{\Gmr}_0(\td{\Emr\Gmr}\wedge \Xmr_+;\ull{\FF}_p)
    \end{tikzcd}
        \end{equation}
is the image of the $\tr_\sfe^{\Gmr}: \Hmr_0( \upiota_{\sf e} \Xmr_+;\ull{\FF}_p) \longrightarrow \Hmr^{\Gmr}_0(\Xmr_+;\ull{\FF}_p) $.
    Moreover, if the transfer is nontrivial, then 
   \begin{equation} \label{SES}
    \begin{tikzcd}
    0\rar &  \Hmr_0(\upiota_{\sfe}\Xmr_+)\rar["\tr_{\sfe}^{\Gmr}"]&  \Hmr_0^\Gmr(\Xmr_+)\rar &  \Hmr_0^\Gmr(\td{\Emr\Gmr}\wedge \Xmr_+)\rar & 0.
    \end{tikzcd}
\end{equation}
    is a short exact sequence of $\FF_p$-vector spaces. 
\end{lem}

\begin{proof} By running the long exact sequence associated to the cofiber sequence $\Emr\Gmr_+ \to \mr{S}^0 \to \td{\Emr \Gmr}$, and using the fact that negative Bredon homology of any space is trivial, we notice that the map \eqref{geomap} is a surjection whose kernel is the image of map  
    \[
    \begin{tikzcd}
    \Hmr^{\Gmr}_0(\Emr\Gmr_+\wedge \Xmr_+;\ull{\FF}_p)\rar & \Hmr^{\Gmr}_0(\Xmr_+;\ull{\FF}_p).
    \end{tikzcd}
    \]
    Since $\Gmr_+\to \Emr\Gmr_+$ is the zeroth $\Gmr$-equivariant skeletal approximation, the map 
        \[
    \begin{tikzcd}
    \Hmr^{\Gmr}_0(\Gmr_+\wedge \Xmr_+;\ull{\FF}_p)\rar & \Hmr^{\Gmr}_0(\Emr\Gmr_+\wedge \Xmr_+;\ull{\FF}_p)
    \end{tikzcd}
    \]
    is an isomorphism. Moreover the composite  $\Gmr_+\wedge \Xmr_+\to \Emr\Gmr_+\wedge \Xmr_+\to \mr{S}^0\wedge \Xmr_+$ induces   the transfer map:  
    \[
    \begin{tikzcd}
    \tr_{\sfe}^{\Gmr}: \Hmr_0( \upiota_{\sf e} \Xmr_+;\ull{\FF}_p) \cong \Hmr^{\Gmr}_0(\Gmr_+\wedge \Xmr_+;\ull{\FF}_p) \rar & \Hmr^{\Gmr}_0(\Xmr_+;\ull{\FF}_p).
    \end{tikzcd}
    \]
  The claim regarding the short exact sequence holds because $\Hmr_0^\sfe(\upiota_{\sf e}\Xmr_+; \ull{\FF}_p)$ is isomorphic to $\FF_p$, a consequence of $\upiota_\sfe \Xmr$ being path connected. Therefore, if the transfer is nontrivial, it is necessarily an injection.
  \end{proof}

Although we do not compute the transfer in \Cref{SES} explicitly, it is not hard to show that there are many cases where it is nontrivial. For example: 

\begin{ex} \label{ex:nontrivTR}
    The map $\Xmr_+\to \Smr^0$ induces a map of Mackey functors in homology. The transfer $\tr_\sfe^{\Gmr}$ in $\ull{\Hmr}^{\Gmr}_0(\Smr^0;\ull{\FF}_2)$ is the transfer in the constant Mackey functor $\ull{\FF}_2$, i.e., multiplication by $|\Gmr|$. This is nonzero when $|\Gmr|$ is odd.  Thus, the inequalities of \Cref{lem:vanish2} and \Cref{lem:vanishodd2} are often strict.
\end{ex} 
\begin{rmk} \label{rmk:fixzero}
In the case when the transfer is nontrivial, the short exact sequence \eqref{SES} splits naturally. This is because 
\[
\begin{tikzcd}
\Hmr^{\Gmr}_0(\Xmr^{\Gmr}_+;\ull{\FF}_p)\rar & \Hmr^{\Gmr}_0(\Xmr_+;\ull{\FF}_p)\rar & \Hmr^{\Gmr}_0(\td{\Emr\Gmr}\wedge \Xmr_+;\ull{\FF}_p)
\end{tikzcd}
\]
is an isomorphism. To prove this, we first note that $ \Hmr^{\Gmr}_0(\Xmr^{\Gmr}_+;\ull{\FF}_p) \cong\Hmr_0(\Xmr^{\Gmr}; \FF_p ) $, which is  free $\FF_p$-vector space generated by $\Gmr$-connected components of $\Xmr.$  We then compute the right-hand side as follows:
\begin{align*}
    \begin{split}
        \Hmr^{\Gmr}_0(\td{\Emr\Gmr}\wedge \Xmr_+;\ull{\FF}_p) &\cong [\Smr^0, \Hmr\ull{\FF}_p\wedge \td{\Emr\Gmr}\wedge \Xmr_+]^{\Gmr}\\
            &\cong [\Smr^0, \Phi^{\Gmr}\Hmr\ull{\FF}_p\wedge \Xmr^{\Gmr}_+]\\
            &\cong (\Phi^{\Gmr}\Hmr\ull{\FF}_p)_0(\Xmr_+^\Gmr) \\
            & \cong (\Hmr\ull{\FF}_p)_0(\Xmr_+^\Gmr), 
    \end{split}
\end{align*}
where the last isomorphism folows from the fact that $\Phi^{\Gmr}\Hmr\ull{\FF}_p$ splits as a wedge of $\Hmr\FF_p$ and a $1$-connected $\Hmr\FF_p$-module. 
\end{rmk}
Now, we focus on the specific case when  $\Xmr = \PP(\uprho)$. From \Cref{ex:nontrivTR}. Let 
\[ \sft = {\sf tr}_{\sfe}^{\Gmr} ({\bf b}_0) \in \Hmr^{\Gmr}_0(\PP(\uprho)_+; \ull{\FF}_2),   \]
where ${\bf b}_0$ is the generator of $\Hmr_0(\RR\PP^{|\Gmr| -1}_+; \FF_2)$. Then: 
\begin{prop} \label{prop:tau} The class ${\sf tr}_{\sfe}^{\Gmr} ({\bf b}_{k |\Gmr|}) \in \Hmr^{\Gmr}_{k \uprho}\left((\Bmr_\Gmr \Sigma_2)_+; \ull{\FF}_2\right)$ is nonzero  iff $\sft =0$, for all $k \in \NN$.
\end{prop}
 \begin{proof} This follows immediately from the study of the transfer map between the spectral sequences of \eqref{eqn:resSS2}. 
 \end{proof}
 Next, we  show that:
 \begin{thm} \label{thm:capiso} For all $ k \geq 0$, the map 
 \[
 \begin{tikzcd}
 - \frown \td{\bf e}_{\Gmr, 2}: \Hmr^\Gmr_{(k+1) \uprho} ((\Bmr_\Gmr\Sigma_2)_+ ; \ull{\FF}_2)  \rar & \Hmr^\Gmr_{k \uprho} ((\Bmr_\Gmr\Sigma_2)_+ ; \ull{\FF}_2) 
 \end{tikzcd}
  \]
  is an isomorphism. 
 \end{thm}
\begin{proof}
It follows from the theory of equivariant vector bundle (see \cite[Lemma 2.18]{BZ}) that  $\Gmr$-fixed point $\td{\upgamma}_{\Gmr, 2}$ is
\[  (\td{\upgamma}_{\Gmr, 2})^\Gmr \cong \bigsqcup_{\lambda \in \mr{Irr}_1(\Gmr)}  \td{\upgamma}_{\sfe, 2}. \]
Thus, the modified $\Gmr$-geometric functor sends the $\Hmr\ull{\FF}_2$-Euler class $\td{\bf e}_{\Gmr, 2}$ to 
\[ \td{\varphi}^{\Gmr}(\td{\bf e}_{\Gmr, 2}) = \sum_{\uplambda} \td{\bf e}_{\uplambda},  \]
where $\td{\bf e}_{\uplambda} $ is an Euler class of $ \td{\upgamma}_{\sfe, 2}$ over the component corresponding to $\uplambda$. Easy to see that 
\[ 
\begin{tikzcd}
- \frown \td{\varphi}^{\Gmr}(\td{\bf e}_{\Gmr, 2}): \bigoplus_\uplambda \Hmr_{k+1}( (\Bmr\Sigma_2)_+; \FF_2) \rar & \bigoplus_\uplambda \Hmr_{k}( (\Bmr\Sigma_2)_+; \FF_2) 
\end{tikzcd}
\]  
is an isomorphism for all $k \geq 0$.  Since $\td{\varphi}^{\Gmr}$ satisfies the general formula
\[ \td{\varphi}^{\Gmr}( b \frown e ) = \td{\varphi}^{\Gmr}(b) \frown \td{\varphi}^{\Gmr}(e), \]
we have a commutative diagram 
\[ 
\begin{tikzcd}
\Hmr^\Gmr_{(k+1) \uprho} ((\Bmr_\Gmr\Sigma_2)_+ ; \ull{\FF}_2) \rar[ "\td{\varphi}^\Gmr"] \dar["- \frown \td{\bf e}_{\Gmr, 2} "']  & \Hmr_{(k+1)} ((\Bmr\Sigma_2) ; \FF) \dar["- \frown \td{\varphi}^{\Gmr}(\td{\bf e}_{\Gmr, 2})" ] \\
\Hmr^\Gmr_{k \uprho} ((\Bmr_\Gmr\Sigma_2)_+ ; \ull{\FF}_2) \rar["\td{\varphi}^\Gmr"] & \Hmr_{k} ((\Bmr\Sigma_2) ; \FF),
\end{tikzcd}
\]
where the right vertical arrow is an isomorphism for all $k \geq 0$.  When $\sft =0 $, \Cref{lem:shortExactSeqZeroHomology} and \Cref{rmk:fixzero} implies that the horizontal arrows in the diagram above are also isomorphisms. Consequently, the left vertical arrow, which is capping with $\td{\bf e}_{\Gmr, 2}$ is an isomorphism for all $k \geq 0$, as desired.   

When $\sft \neq 0$, we consider the diagram 
\[ 
\begin{tikzcd}
\Hmr_{(k+1)|\Gmr|}((\Bmr\Sigma_2)_+; \FF_2) \rar[hook, "{\sf tr}_{\sf e}^{\Gmr}"] \dar["- \frown \td{\bf e}_{{\sf e}, 2}^{|\Gmr|}"'] & \Hmr^\Gmr_{(k+1) \uprho} ((\Bmr_\Gmr\Sigma_2)_+ ; \ull{\FF}_2) \rar[two heads,  "\td{\varphi}^\Gmr"] \dar["- \frown \td{\bf e}_{\Gmr, 2}"']  & \bigoplus_{\uplambda} \Hmr_{(k+1)} ((\Bmr\Sigma_2) ; \FF) \dar["- \frown \td{\varphi}^{\Gmr}(\td{\bf e}_{\Gmr, 2})" ] \\
\Hmr_{k|\Gmr|}((\Bmr\Sigma_2)_+; \FF_2)\rar[hook, "{\sf tr}_{\sf e}^{\Gmr}"']  & \Hmr^\Gmr_{k \uprho} ((\Bmr_\Gmr\Sigma_2)_+ ; \ull{\FF}_2) \rar[two heads, "\td{\varphi}^\Gmr"'] & \bigoplus_{\uplambda} \Hmr_{k} ((\Bmr\Sigma_2) ; \FF)
\end{tikzcd}
\]
where the rows are short exact sequences (follows from \Cref{lem:shortExactSeqZeroHomology} and \Cref{prop:tau}) and  the left square commutes because of \Cref{lem:rescap}. Since  the left  and the  right vertical arrows are isomorphism, it follows that the middle vertical arrow is also an isomorphism as desired. 
\end{proof}
A nearly identical argument yields the odd primary analog of the above result:
\begin{thm} \label{thm:capisop} For all $ k \geq 0$, the map 
 \[
 \begin{tikzcd}
 - \frown \upkappa^*\td{\bf e}_{\Gmr, p}: \Hmr^\Gmr_{2\epsilon (k+1) \uprho} ((\Bmr_\Gmr\Cmr_p)_+ ; \ull{\FF}_p)  \rar & \Hmr^\Gmr_{2\epsilon k \uprho} ((\Bmr_\Gmr\Cmr_p)_+ ; \ull{\FF}_p) 
 \end{tikzcd}
  \]
  is an isomorphism. 
 \end{thm}

\begin{lem} \label{lem:rescap} Suppose $\Xmr$ is a $\Gmr$-space,  $t \in \Hmr^{\star}_\Gmr(\Xmr_+; \ull{\FF}_p)$ a $\RO(\Gmr)$-graded cohomology class, and  $x \in \Hmr_*(\upiota_{\sfe}\Xmr_+; \FF_p)$ a homology class of the underlying nonequivariant space. Then:
\[ {\sf tr}_{\sfe}^{\Gmr} (x) \frown t = {\sf tr}_{\sfe}^{\Gmr}\left( x \frown \res_{\sfe}(t) \right) ). \]
\end{lem}
\begin{proof} We first observe that  the quotient map $\uppi: \Gmr \longrightarrow \Gmr/\Gmr \cong {\sf e}$ induces the transfer 
\[ \begin{tikzcd}
 {\sf tr}_{\sfe}^{\Gmr} = \uppi_*: \Hmr_{|\Vmr|}(\upiota_{\sfe} \Xmr_+; \FF_p) \cong \Hmr_\Vmr^{\Gmr}(\Gmr_+ \sma  \Xmr_+; \ull{\FF}_p) \rar & \Hmr_\Vmr^{\Gmr}( \Xmr_+; \ull{\FF}_p)
\end{tikzcd}
 \]
 in homology,  and the restriction map 
 \[ 
 \begin{tikzcd}
 \res_{\sfe} = \uppi^*: \Hmr^{\Wmr}(\Xmr_+; \ull{\FF}_p)  \rar &  \Hmr^\Wmr_{\Gmr}(\Gmr_+ \sma  \Xmr_+; \ull{\FF}_p)  \cong \Hmr^{|\Wmr|}_{\Gmr}( \Xmr_+; \FF_p)
\end{tikzcd}
 \]
 in cohomology,  for any $\Vmr, \Wmr  \in \RO(\Gmr)$,. Then the result follows from the natural relation 
 \[ 
 \uppi_*(x) \frown t = \uppi_*(x \frown \uppi^*(t))
\]
 satisfied by the cap product. 
\end{proof}
Using  the isomorphisms of \Cref{thm:capiso} and \Cref{thm:capisop}  we define the $\Hmr\ull{\FF}_p$-Eulerian sequence
 \[ 
 \upbeta_{\lambda, (p)} := (\hat{\sfb}_{0, \lambda},\hat{\sfb}_{1, \lambda}, \dots  )
 \]
 where $\hat{\sfb}_{0, \lambda} = \sfb_{0, \lambda}$  and use induction to define $\hat{\sfb}_{k}$ such that 
 \[  \hat{\sfb}_{k+1, \lambda} \frown \upkappa^*\td{\bf e}_{\Gmr, p} =  \hat{\sfb}_{k, \lambda} \]
 where $\upkappa$ is the map defined in \Cref{notn:kappa} for $p$ odd and identity  when $p =2$.  
 \begin{rmk} The elements $\hat{\sfb}_{k, \lambda}$ may differ from the element of $\sfb_{k \uprho, \lambda}$ up to an element in the  image of the transfer map $\tr_{\sfe}^\Gmr$. 
 \end{rmk}
\begin{notn}
 Using \Cref{main1}, we obtain a  genuine stable $\Hmr\ull{\FF}_p$-cohomology operation corresponding to $\upbeta_{\lambda, (p)}$ and its $k$-shifts. When $p=2$, we set  \[ \Sq^{k \uprho_\Gmr}_{\lambda} := \Sfrak^{\upbeta_{\lambda, (2)}[k]} .\]
 Further, when $\lambda = 1 \in \mr{Irr}_1(\Gmr)$, we simply use the notation $\Sq^{k \uprho_\Gmr}$. When $p$ is odd, we set 
 \[ \Pmr^{2 \epsilon k \uprho_\Gmr}_{\lambda} := \Sfrak^{\upbeta_{\lambda, (p)}[k]} \]
 and drop the subscript when $\lambda = 1 \in \td{\mr{Irr}}_1(\Gmr)$. 
\end{notn}
\begin{proof}[{\bf Proof of \Cref{main3}}]
Since restriction of the class $\sfb_{0, \lambda}$ to a subgroup $\Kmr$ is 
\[ \sfb_{0, \upiota_\Kmr\lambda} \in \Hmr_{0}^{\Kmr}((\Bmr_\Kmr\Cmr_p)_+ ; \ull{\FF}_p), \]
and restriction of the Euler class $\td{\bf e}_{\Gmr, p}$ to $\Kmr$ is  $\td{\bf e}_{\Kmr, p}^{|\Gmr/\Kmr|}$, we conclude that  
\[ \upiota_{\Kmr} ( \upbeta_{\lambda, (p)}) =  \mathfrak{t}_{|\Gmr/\Kmr|}\upbeta_{\upiota_{\Kmr}\lambda, (p)}.\]
This,  combined with \Cref{thm:opsRestrAndGeoFixedPts}, \Cref{thm:resSq} and \Cref{thm:resP} completes the proof. 
\end{proof}

\begin{proof}[{\bf Proof of \Cref{main4}}]
First, assume that $\Kmr$ is a normal subgroup of $\Gmr$. In this case, the Weyl group is $\Wmr(\Kmr)= \Gmr/\Kmr$.   
We observe that the modified $\Kmr$-geometric fixed-point of the class $\sfb_{0, \lambda}$  is 
\[ \sfb_{0, \lambda^\Kmr} \in \Hmr_{0}^{\Kmr}((\Bmr_\Kmr\Cmr_p)_+ ; \ull{\FF}_p), \]
where  we set $\sfb_{0, \lambda^\Kmr} =0 $ when $\lambda^\Kmr = {\bf 0}$. Furthermore, note  $\sfc_*\td{\varphi}^\Kmr(\td{\bf e}_{\Gmr, p}) = \td{\bf e}_{\Gmr/\Kmr, p}$. Therefore,  using \Cref{thm:capisop} and the commutativity of the diagram 
\[ 
\begin{tikzcd}
 \Hmr^\Gmr_{2\epsilon (k+1) \uprho_\Gmr} ((\Bmr_\Gmr\Cmr_p)_+ ; \ull{\FF}_p) \dar["\sfc_*\td{\varphi}^\Kmr(-)"] \ar[rr, "- \frown \upkappa^*\td{\bf e}_{\Gmr, p}", "\cong"'] && \Hmr^\Gmr_{2\epsilon k \uprho_\Gmr} ((\Bmr_\Gmr\Cmr_p)_+ ; \ull{\FF}_p) \dar["\sfc_*\td{\varphi}^\Kmr(-)"]  \\
 \Hmr^\Gmr_{2\epsilon (k+1) \uprho_{\Wmr(\Kmr)}} ((\Bmr_{\Wmr(\Kmr)}\Cmr_p)_+ ; \ull{\FF}_p)  \ar[rr, "- \frown \upkappa^*\td{\bf e}_{\Wmr(\Kmr), p}", "\cong"'] && \Hmr^\Gmr_{2\epsilon k \uprho_\Gmr} ((\Bmr_{\Wmr(\Kmr)}\Cmr_p)_+ ; \ull{\FF}_p) 
\end{tikzcd}
\]
 we conclude that  
\[ \varphi^{\Kmr} ( \upbeta_{\lambda, (p)}) =  \upbeta_{\lambda^\Kmr, (p)}.\]
Thus, the result follows from \Cref{thm:modgeoES}. 

When $\Kmr$ is not a normal subgroup, then we first restrict to the normalizer subgroup $\Nmr(\Kmr)$ before calculating modified $\Kmr$-geometric fixed-points. In this case, the arguments above, combined with \Cref{main3}, yields the stated results. 
\end{proof}

\bibliographystyle{amsalpha}
\bibliography{GSteenroD.bib}
\end{document}